\documentclass{amsart}

\usepackage{amsmath,amsthm,amsfonts,amssymb}

\usepackage[backref=page]{hyperref}
\hypersetup{colorlinks=true,citecolor=blue,linkcolor=blue,urlcolor=blue,
pdfstartview=FitH, pdfauthor=Valentin Blomer and Gergely
Harcos,pdftitle=Twisted $L$-functions over number fields and
Hilbert's eleventh problem}

\theoremstyle{plain}
\newtheorem{theorem}{Theorem}
\newtheorem{cor}{Corollary}
\newtheorem{lemma}{Lemma}

\theoremstyle{definition}
\newtheorem*{Acknowledgements}{Acknowledgements}

\theoremstyle{remark}
\newtheorem{remark}{Remark}

\renewcommand{\geq}{\geqslant}
\renewcommand{\leq}{\leqslant}

\DeclareMathOperator{\GL}{GL} \DeclareMathOperator{\SL}{SL}
\DeclareMathOperator{\SO}{SO} \DeclareMathOperator{\OO}{O}
\DeclareMathOperator{\sgn}{sgn} \DeclareMathOperator{\vol}{vol}
\DeclareMathOperator{\ord}{ord} \DeclareMathOperator{\lcm}{lcm}
\DeclareMathOperator*{\res}{res}

\newcommand{\eps}{\varepsilon}
\newcommand{\RR}{\mathbb{R}}
\newcommand{\CC}{\mathbb{C}}
\newcommand{\QQ}{\mathbb{Q}}
\newcommand{\ZZ}{\mathbb{Z}}
\newcommand{\NN}{\mathbb{N}}
\newcommand{\ma}{\mathfrak{a}}
\newcommand{\mb}{\mathfrak{b}}
\newcommand{\mc}{\mathfrak{c}}
\newcommand{\mo}{\mathfrak{o}}
\newcommand{\mpr}{\mathfrak{p}}
\newcommand{\mt}{\mathfrak{t}}
\newcommand{\mq}{\mathfrak{q}}
\newcommand{\my}{\mathfrak{y}}
\newcommand{\fg}{\mathfrak{g}}
\newcommand{\KK}{\mathcal{K}}
\newcommand{\N}{\mathcal{N}}
\newcommand{\Kplus}{K_{\infty,+}^{\times}}
\newcommand{\Kplusdiag}{K_{\infty,+}^{\text{diag}}}
\newcommand{\CSC}{\text{CSC}}
\newcommand{\FS}{\text{FS}}

\begin{document}

\author{Valentin Blomer}
\address{Department of Mathematics, University of Toronto, 40 St. George Street, Toronto, Ontario, Canada, M5S 2E4} \email{vblomer@math.toronto.edu}

\author{Gergely Harcos}
\address{Alfr\'ed R\'enyi Institute of Mathematics, Hungarian Academy of Sciences, POB 127, Budapest H-1364, Hungary} \email{gharcos@renyi.hu}
\title[Twisted $L$-functions over number fields]{Twisted $L$-functions over number fields and Hilbert's eleventh problem}

\thanks{The first author was in part supported by an NSERC grant 311664-05
and a Sloan Research Fellowship. The second author was supported by
European Community grants EIF~040371 and ERG~239277 within the 6th
and 7th Framework Programmes and by OTKA grants K~72731 and
PD~75126.}


\begin{abstract} Let $K$ be a totally real number field,
$\pi$ an irreducible cuspidal representation of $\GL_2(K)\backslash
\GL_2(\mathbb{A}_K)$ with unitary central character, and $\chi$ a
Hecke character of conductor $\mq$. Then $L(1/2,\pi\otimes \chi) \ll
(\N\mq)^{\frac{1}{2} -\frac{1}{8}(1-2\theta)+\eps}$, where $0\leq
\theta \leq 1/2$ is any exponent towards the Ramanujan--Petersson
conjecture ($\theta=1/9$ is admissible). The proof is based on a
spectral decomposition of shifted convolution sums and a generalized
Kuznetsov formula.
\end{abstract}


\setcounter{tocdepth}{2} \maketitle \tableofcontents

\section{Introduction}

In a recent article \cite{BH} the authors developed a new technique
to study shifted convolution sums in Hecke eigenvalues of the type
\begin{equation}\label{shift}
 \sum_{n - m = q} \lambda_{\pi_1}(n)\lambda_{\pi_2}(m) W_1(n/Y) W_2(m/Y)
\end{equation}
for two irreducible cuspidal representations $\pi_1$, $\pi_2$ of
$\GL_2(\QQ)\backslash\GL_2(\mathbb{A}_{\QQ})$ with conductor $1$,
reasonably regular weight functions $W_1$, $W_2$, a large number
$Y>0$, and $q \neq 0$. Such sums play an important role in the
theory of automorphic forms, in particular in the study of
automorphic $L$-functions, as they constitute a typical off-diagonal
term of the second moment. In this paper we generalize the
method of \cite{BH} to (congruence subgroups of) the Hilbert modular
group of a totally real number field $K$, and we give applications
to subconvexity of twisted $L$-functions over $K$. Before stating
our main result, we note that the subconvexity problem for $\GL_2$
over any fixed number field was recently solved by Michel and
Venkatesh in a beautiful preprint \cite{MV}, where the reader can
also find detailed references to previous work done in the subject.
Yet, as the authors of \cite{MV} remark, their emphasis was not on
obtaining best exponents but rather finding some nontrivial exponent
that works in all cases. The aim of the present paper is to
demonstrate that a relatively strong Burgess-type subconvexity bound
in the conductor aspect can be achieved for the family at hand, and
once the background on automorphic forms has been set up, the proof
requires comparatively little effort (cf.\ Section~\ref{heuristic}).

More precisely, let $\pi$ be an irreducible cuspidal representation
of $\GL_2( K) \backslash \GL_2(\mathbb{A}_K)$ with unitary central
character, and let $\chi$ be a Hecke character of conductor $\mq$.
Let $L(s,\pi \otimes \chi)$ denote the twisted $L$-function. Let
$0\leq\theta\leq 1/2$ be an approximation towards the
Ramanujan--Petersson conjecture. Currently $\theta=1/9$ is known by
the work of Kim and Shahidi~\cite{KSh}, while the
Ramanujan--Petersson conjecture predicts $\theta=0$.

\begin{theorem}\label{theorem1}
For any $\eps>0$ one has $L(1/2,\pi\otimes \chi)
\ll_{\pi,\chi_\infty,K,\eps} (\N\mq)^{\frac{1}{2} -
\frac{1}{8}(1-2\theta)+\eps}$.
\end{theorem}

\begin{remark} This result contains a bound for all values
$L(1/2+it,\pi\otimes\chi)$ on the critical line, because replacing
$1/2$ by $1/2+it$ has the same effect as replacing $\chi$ by
$\chi\otimes|\cdot|^{it}$.
\end{remark}

\begin{remark} The convexity bound in this context is
$(\N\mq)^{\frac{1}{2}+\eps}$. The first subconvex bound over totally
number fields is a result by Cogdell, Piatetski-Shapiro and
Sarnak~\cite{CPSS,Co}, in which they obtained for $\pi$ induced by a
holomorphic Hilbert cusp form\footnote{In this bound and in the next
one we tried to optimize parameters, the original statements are
somewhat weaker.}
\begin{displaymath}
 L(1/2,\pi\otimes \chi) \ll (\N\mq)^{\frac{1}{2} - \frac{1-2\theta}{14+4\theta} +\eps}.
\end{displaymath}
They used a very effective spectral method based on bounds for
triple products \cite{Sa,Sa2}. As an application of an ingenious and
flexible geometric method, Venkatesh~\cite[Theorem~6.1]{Ve} proved a
subconvex bound over all number fields and for all irreducible
cuspidal representations
\begin{displaymath}
 L(1/2,\pi\otimes \chi) \ll (\N\mq)^{\frac{1}{2} - \frac{(1-2\theta)^2}{14-12\theta}+\eps}.
\end{displaymath}
Our method is quite different from both of these works, and
Theorem~\ref{theorem1} supersedes both results. It may be noted that
although most applications of subconvexity require only \emph{any}
nontrivial saving in the exponent, there are situations where the
quality of the subconvex exponent is critical, an example being
\cite{CCU} where $L(1/2, \pi \otimes \chi) \ll (\text{cond }
\chi)^{\frac{1}{2} - \frac{1}{16}+\eps}$ or an equivalent bound for
metaplectic Fourier coefficients (over $\QQ$) is needed.
\end{remark}

\begin{remark} An inspection of the proof shows that with somewhat more
precise estimates the implied constant in Theorem~\ref{theorem1}
turns out to be polynomial in the analytic conductor of $\pi$ and
the archimedean parameters of $\chi$ with an exponent depending on
$\eps$.
\end{remark}

The proof of Theorem~\ref{theorem1} builds on the ideas of several
earlier works, most notably of \cite{DFI1,CPSS,V1,Ve,BH}. Applying
an approximate functional equation, a typical off-diagonal term in
the amplified second moment is essentially of the form \eqref{shift}
with a slightly more general summation condition $\ell_1 n - \ell_2
m = q$ for any nonzero $q \in \mq$. Often the estimation of such
expressions rests on some variant of the circle method (see e.g.\
\cite{DFI1, BHM}) in order to detect the summation condition.
However, this seems difficult to implement over number fields with a
nontrivial unit and class group, in contrast to the more structural
approach in \cite{BH} which we will follow here. The proof is
written in an interesting mixture of classical and modern language:
on the one hand, we use an adelic setup to treat the number field
situation appropriately. On the other hand, at the heart of the
amplification is Iwaniec's idea of playing off various subgroups
against each other, and so we need to keep track carefully of the
various levels occurring in the course of the argument.

Perhaps the most appealing application of Theorem~\ref{theorem1} is
to combine it with the formula of Walds\-purger~\cite{Wa} and its
extensions by Shimura~\cite{Sh2}, Khuri-Makdisi~\cite{KM},
Kojima~\cite{Ko}, Baruch--Mao~\cite{BM} and others in order to bound
the Fourier coefficients of half-integral weight Hilbert modular
forms. For $K =\QQ$, the original breakthrough was achieved by
Iwaniec~\cite{Iw1}, and the currently strongest bounds are given in
\cite{BH1}. For a totally real number field $K$ other than $\QQ$,
there does not seem to be an explicit reference in the literature.

\begin{cor}\label{cor5}
Let $(\tilde{\pi},V_{\tilde\pi})$ be an irreducible cuspidal
representation of $\,\widetilde{\SL}_2(K)\backslash
\widetilde{\SL}_2(\mathbb{A}_K)$, orthogonal to one-dimensional
theta series, and let $r \in \mo$ be a nonzero squarefree integer.
Define the $r$-th normalized Fourier coefficient
$\rho_{\tilde{\phi}}(r)$ of a pure tensor
$\tilde{\phi}=\otimes_v\tilde\phi_v \in V_{\tilde{\pi}}$ by the left
hand side of \eqref{baruchmao} below. Then
\begin{displaymath}
 \sqrt{|\N r|} \rho_{\tilde{\phi}}(r) \ll_{\tilde{\phi}, K, \eps} |\N r|^{\frac{1}{4}- \frac{1}{16}(1-2\theta) + \eps}.
\end{displaymath}
\end{cor}

\begin{remark}The ``trivial'' bound in this context is $|\N
r|^{\frac{1}{4}+\eps}$ on the right hand side, while the Ramanujan
conjecture (implied by the Lindel\"of hypothesis for twisted
$L$-functions) states the bound $|\N r|^{\eps}$.\end{remark}

One particular situation where such bounds are needed, are
asymptotic formulae for the number of representations of totally
positive integers by ternary quadratic forms, see \cite{Bl} for an
overview of this topic over $\QQ$. Hilbert's eleventh problem asks
more generally which integers are integrally represented by a given
$n$-ary quadratic form $Q$ over a number field $K$. If $Q$ is a
binary form, it corresponds to some element in the class group of a
quadratic extension of $K$ (see \cite{Cox} for a nice account over
$\QQ$). If $Q$ is indefinite at some archimedean place,
Siegel~\cite{Si} for $n \geq 4$ and Kneser~\cite{Kn} and
Hsia~\cite{H} for $n=3$ proved a local-to-global principle, so
Siegel's mass formula~\cite{Si1} tells us exactly which integers are
represented by $Q$. If $Q$ is positive definite at every archimedean
place and $n \geq 4$, again Siegel's mass formula~\cite{Si1} and bounds for
Fourier coefficients of Hilbert modular forms give a complete answer
(some care has to be taken in the case $n=4$). The only remaining
case of $Q$ positive definite and $n=3$ was solved by Duke and
Schulze-Pillot~\cite{DSP} for $K = \QQ$. For arbitrary totally real
$K$, the result was established by Cogdell, Piatetski-Shapiro and
Sarnak~\cite{CPSS}; an account of the key ideas appeared in
\cite{Co}. In fact, the systematic study of subconvexity over number
fields was initiated by \cite{CPSS} about a decade ago motivated by
this striking application. The relevant subconvex bound was
subsequently generalized over arbitrary number fields by Venkatesh~\cite{Ve},
while our Corollary~\ref{cor5} allows a better approximation for the
number of representations.

\begin{cor}
Let $K$ be a totally real number field and let $Q$ be a positive
integral ternary quadratic form over $K$. Then there is an
ineffective constant $c>0$ such that every totally positive
squarefree integer $r \in \mo$ with $\N r \geq c$ is represented
integrally by $Q$ if and only if it is integrally represented over
every completion of $K$.
More precisely, the number of representations for such $r$ equals
$(\N r)^{\frac{1}{2}+o(1)}+O((\N r)^{\frac{7}{16}+\frac{\theta}{8}+o(1)})$.
\end{cor}

\begin{remark} This result, with a slightly weaker error term,
was originally proved in \cite{CPSS}. The representation of non-squarefree
integers is quite subtle, but in principle can again be
characterized by more involved local considerations, cf.\
\cite{SP}.
\end{remark}

Another application of Theorem~\ref{theorem1} can be found in
\cite[Theorem~1.2]{Coh} and \cite[Theorem~3.2]{Zh} (cf.\ also
\cite[Section 1.1]{Ve}) that generalizes work of Duke~\cite{D}:
under the assumption of a subconvex bound as above it is proved that
a certain family of Heegner points and certain $d$-dimensional
subvarieties are equidistributed on the Hilbert modular surface
$\SL_2(\mo_K)\backslash \mathcal{H}^d$.\\

The core of Theorem~\ref{theorem1}, from which it will follow in a
fairly straightforward procedure, is the spectral decomposition of
smooth shifted convolution sums which implies strong upper bounds
for these sums. This is stated as Theorem~\ref{theorem4} in
Section~\ref{section3} after the necessary notation is developed. We
give another application of this result in Theorem~\ref{theorem5} of
Section~\ref{section5}: we prove the analytic continuation and
spectral decomposition of the Dirichlet series associated to shifted
convolution sums with polynomial growth on vertical lines. This
problem goes back to Selberg~\cite{Se}.

Section~\ref{section2} contains the necessary background on
automorphic forms. This section turned out to be very long; although
much of the material presented there is essentially known, many of
the results and computations in the number field case do not seem to
be explicit in the literature. Therefore we felt that it makes the
paper more useful (also a reference for future work in this subject)
and readable if we give rather complete details.

\begin{Acknowledgements}
This paper would not exist without the insight and guidance of Peter
Sarnak who suggested Gergely Harcos in 2000 to work on this project;
we are grateful for his support over the years. Important parts of
this paper were written during a visit of Valentin Blomer to R\'enyi
Institute of Mathematics in October 2008, and during the workshop
``Analytic Theory of GL(3) Automorphic Forms and Applications" at the
American Institute of Mathematics in November 2008. We thank both
institutions for their hospitality and financial support, and the
organizers of the workshop for their kind invitation. Finally we
thank Jim Arthur for valuable discussions related to this work.
\end{Acknowledgements}

\section{Part I: Background on automorphic forms}\label{section2}

\subsection{Basic Notation}

\subsubsection{Number fields and adele rings}\label{numberfields}

Let $K$ be a totally real number field over $\QQ$ of degree $d$,
discriminant $D_K$, different $\mathfrak{d}$ and ring of integers
$\mo$. Throughout the paper we regard $K$ as fixed and all constants
may depend on $K$, even if not stated explicitly, and they may also
depend on $\eps$ which denotes an arbitrarily small positive number,
not necessarily the same on each occurrence. We embed $K$ as a
$\QQ$-algebra into $K_{\infty}:= \RR^d$ using the $d$ real field
embeddings $r \mapsto (r^{\sigma_1}, \ldots, r^{\sigma_d})$. We
denote by $\Kplus:=\RR^d_{>0}$ the set of totally positive elements
of $K_{\infty}$, and we put \[\Kplusdiag:= \{(x, \ldots, x) \mid x
\in \RR_{>0} \} \subseteq \Kplus.\] For $r \in K$ we write
\[\sgn(r):=(\sgn(r^{\sigma_1}), \ldots, \sgn(r^{\sigma_d})) \in
\{\pm 1\}^d,\] and we write $r >> 0$ for a totally positive integer
$r \in \mo$. We denote by $U^{+} \subseteq U$ the group of totally
positive units and the group of units of $\mo$, respectively.

Let $\mathbb{A}$ be the adele ring of $K$, with $K$ being embedded
diagonally (this defines in particular a multiplication $K \times
\mathbb{A} \to \mathbb{A}$). We shall often write $\mathbb{A} =
K_{\infty} \times \mathbb{A}_{\text{fin}}$. We shall label the
archimedean places with elements of $\{1,\dots,d\}$ and the
non-archimedean places with prime ideals $\mpr$ of $K$ in an obvious
way. As usual, we shall denote the module of an idele $x \in
\mathbb{A}^\times$ by $|x|:=|x_\infty||x_\text{fin}|$, where
$|x_\infty|:=\prod_{j=1}^d|x_j|$ and
$|x_\text{fin}|:=\prod_\mpr|x_\mpr|$. We denote by $\psi:\mathbb{A}
\to S^1$ the unique continuous additive character which is
trivial on $K$, agrees with $x\mapsto e(x_1+\cdots+x_d)$ on
$K_\infty$, and on $K_\mpr$ it is trivial on
$\mathfrak{d}_\mpr^{-1}$ but nontrivial on
$\mpr^{-1}\mathfrak{d}_\mpr^{-1}$. Here and later a subscript $\mpr
$ indicates completion with respect to the corresponding valuation
$v_\mpr$. If $\Omega:=\prod_\mpr\mo_\mpr^\times$ is the unique
maximal compact subgroup of $\mathbb{A}_{\text{fin}}^{\times}$, then
$\Omega\backslash\mathbb{A}_{\text{fin}}^{\times}$ is isomorphic in
a natural way to the multiplicative group $I(K)$ of nonzero
fractional ideals of $K$. We shall occasionally identify an idele in
$\mathbb{A}_{\text{fin}}^{\times}$ with its image under the
corresponding surjective homomorphism
$\mathbb{A}_{\text{fin}}^{\times}\to I(K)$. This homomorphism also
gives rise to a natural action of $\mathbb{A}_{\text{fin}}^{\times}$
on $I(K)$. We write $\sim$ for equivalence in the ideal class group
\[C(K):=K^\times\backslash I(K)\cong
K^\times\Omega\backslash\mathbb{A}_{\text{fin}}^{\times}\cong
K^\times K_\infty^\times\Omega\backslash\mathbb{A}^{\times}.\] We
write $h:=\#C(K)$ for the class number of $K$. Let $\N : K \to \QQ$
be the norm, which we extend to an $\RR$-multilinear map $K_\infty
\to \RR$; the norm of a fractional ideal $\mathfrak{m}\in I(K)$ will
also be denoted by $\N\mathfrak{m}$. Note that the norm of an
infinite idele $y\in K_\infty^\times$ is $|y|$, but the norm of the
fractional ideal $(y)=y\mo\in I(K)$ corresponding to a finite idele
$y\in\mathbb{A}_\text{fin}$ is $|y|^{-1}$.

We denote by $\mu(\ma )$, $\varphi(\ma )$ and $\tau(\ma )$ the
obvious generalizations of the M\"obius, the Euler, and the divisor
functions to nonzero ideals $\ma \subseteq\mo$. We will often use
the basic estimates $\#\{\mathfrak{m} \subseteq \mo \mid
\N\mathfrak{m} \leq x \} \asymp_K x$ for $x \geq 1$ and
$\tau(\mathfrak{m}) \ll_{K, \eps} (\N\mathfrak{m})^{\eps}$ for any
$\eps> 0$.

\subsubsection{Matrix groups}\label{matrixgroups} For any ring $R$ we define
the following important subgroups of $\GL_2(R)$:
\[Z(R):=\left\{\begin{pmatrix} a &0\\ 0 & a\end{pmatrix} \bigg|\ a\in R^{\times}\right\}
\qquad\text{and}\qquad P(R) := \left\{\begin{pmatrix} a & b\\ 0 &
d\end{pmatrix} \bigg|\ a, d \in R^{\times}, \ b \in R\right\}.\] For
$\vartheta \in (\RR/2\pi\ZZ)^d$ we write
\begin{displaymath}
 k(\vartheta) := \left(\begin{matrix}\ \ \ \cos\vartheta & \sin\vartheta\\ -\sin\vartheta & \cos\vartheta \end{matrix}\right)
 \in \SO_2(K_\infty).
\end{displaymath}
For nonzero ideals $\my, \mc \subseteq \mo_\mpr$ we define
\[\KK(\my, \mc):=\left\{\begin{pmatrix} a &b\\ c & d\end{pmatrix}\in \GL_2(K_\mpr)
 \,\bigg|\ a, d \in \mo_\mpr, \ b \in (\my\mathfrak{d}_\mpr)^{-1},
 \ c \in \my\mathfrak{d}_\mpr\mc,
 \ ad-bc\in\mo_\mpr^\times\right\}.\]
If $\my = \mo_\mpr$, we just write $\KK(\mc)$ instead of
$\KK(\mo_{\mpr}, \mc)$. For nonzero ideals $\my, \mc\subseteq\mo$ we
define
\[\KK(\mc) := \prod_{\mpr}\KK(\mc_\mpr)\subseteq
 \GL_2(\mathbb{A}_{\text{fin}}),\qquad \KK: = \SO_2(K_\infty) \times
\KK(\mo)\subseteq\GL_2(\mathbb{A}),\] and
\begin{equation}\label{gamma}
\Gamma(\my, \mc) :=\Bigl\{g_\infty\in\GL_2(K_{\infty})\,\Big|\
\text{$g_\infty g_\text{fin}\in\GL_2(K)$ for some
$g_\text{fin}\in\prod_{\mpr} \KK(\my_\mpr, \mc_\mpr)$}\Bigr\}.
\end{equation}
The definition of $\KK(\mc)$ is by no means the only way of
specifying a subgroup of ``level $\mc$''. We are following
\cite{Mi,Sh1,KM} here. In \cite{BMP1, V1}, for instance, the
different $\mathfrak{d}$ is not included in the definition of
$\KK(\mc)$, and the reader will easily see that all proofs in this
paper would go through with very minor modifications, had we chosen
a different definition of $\KK(\mc)$.

\subsubsection{Measures}\label{measures}
On $K_\infty$ we use the normalized Lebesgue measure
$|D_K|^{-1/2}dx_1\cdots dx_d$. On $K_\mpr$ we normalize the Haar
measure so that $\mo_\mpr$ has measure $1$. On $\mathbb{A}$ we use
the Haar measure $dx$ which is the product of these measures, this
induces the Haar probability measure on $K\backslash\mathbb{A}$. On
$K_\infty^\times$ we use the Haar measure
$(dy_1/|y_1|)\cdots(dy_d/|y_d|)$. On $K_\mpr^\times$ we normalize
the Haar measure so that $\mo_\mpr^\times$ has measure $1$. On
$\mathbb{A}^\times$ we use the Haar measure $d^{\times}y$ which is
the product of these measures, this induces some Haar measure on
$K^\times\backslash\mathbb{A}^\times$. On $\KK$ and its factors we
use the Haar probability measures. On
$Z(K_\infty)\backslash\GL_2(K_\infty)$ we use the Haar measure which
satisfies
\[\int_{Z(K_\infty)\backslash\GL_2(K_\infty)}f(g)\,dg=
\int_{K_\infty^\times}\int_{K_\infty}\int_{\SO_2(K_\infty)}
f\left(\begin{pmatrix}y &x\\0
&1\end{pmatrix}k\right)\,dk\,dx\,\frac{d^\times y}{|y|}.\] On
$\GL_2(K_\mpr)$ we normalize the Haar measure so that
$\KK(\mo_\mpr)$ has measure $1$. On
$Z(K_\infty)\backslash\GL_2(\mathbb{A})$ we use the product of these
measures, this induces the Haar measure on
$Z(\mathbb{A})\backslash\GL_2(\mathbb{A})$ satisfying (cf.\
\cite[(3.10)]{GJ})
\[\int_{Z(\mathbb{A})\backslash\GL_2(\mathbb{A})}f(g)\,dg=
\int_{\mathbb{A}^\times}\int_\mathbb{A}\int_\KK
f\left(\begin{pmatrix}y &x\\0
&1\end{pmatrix}k\right)\,dk\,dx\,\frac{d^\times y}{|y|}.\]

\subsection{Spectral decomposition and Eisenstein series}\label{spectralsection}

Let $\omega : K^{\times} \backslash \mathbb{A}^{\times} \to S^1$
be a Hecke character, regarded also as a character of
$Z(K)\backslash Z(\mathbb{A})$. Without loss of
generality\footnote{Note that in Theorem~\ref{theorem1} replacing
$\pi$ by $\pi\otimes|\det|^{it}$ and $\chi$ by
$\chi\otimes|\cdot|^{-it}$ leaves $\pi\otimes\chi$ unchanged. In
fact for the proof of Theorem~\ref{theorem1} we only need the
results of this section for trivial $\omega$.} we shall assume that
$\omega$, viewed as a character of $\mathbb{A}^{\times}$, is trivial
on $\Kplusdiag$. The group $\GL_2(\mathbb{A})$ acts by right
translation on the Hilbert space
\[L^2(\GL_2(K)\backslash \GL_2(\mathbb{A}), \omega)\]
of measurable functions $\phi:\GL_2(\mathbb{A})\to\CC$ satisfying
\begin{gather*}
\phi(\gamma zg) = \omega(z) \phi(g), \qquad
\gamma \in \GL_2(K), \ \ z \in Z(\mathbb{A}), \ \
g \in \GL_2(\mathbb{A}),\\
\langle g, g\rangle := \int_{\GL_2(K)Z(\mathbb{A})\backslash
\GL_2(\mathbb{A})} |\phi(g)|^2 \,dg < \infty.
\end{gather*}
A function $\phi \in L^2(\GL_2(K)\backslash \GL_2(\mathbb{A}),
\omega)$ is called cuspidal if
\begin{displaymath}
 \int_{K\backslash \mathbb{A}}\phi\left(\left(\begin{matrix}1 & x\\0 & 1\end{matrix}\right) g\right) dx = 0 \quad \text{ for almost all } g \in \GL_2(\mathbb{A}).
\end{displaymath}
We have a $\GL_2(\mathbb{A})$-invariant decomposition
\begin{displaymath}
 L^2(\GL_2(K)\backslash \GL_2(\mathbb{A}), \omega) = L_{\text{cusp}} \oplus L^{\perp}_{\text{cusp}}
 \end{displaymath}
into the space of cuspidal functions and its orthogonal complement.
The cuspidal space decomposes into a Hilbert space direct
sum\footnote{Here and later we do not indicate closure for
notational simplicity.} of irreducible automorphic representations:
\begin{displaymath}
 L_{\text{cusp}} = \bigoplus_{\pi\in \mathcal{C}_\omega} V_{\pi}.
\end{displaymath}
The orthogonal complement $L^{\perp}_{\text{cusp}} $ is described in
detail in \cite[Sections 3--5]{GJ}, see also \cite[Section~3.7]{Bu}:
For any Hecke character $\chi$ satisfying $\chi^2 = \omega$ (which
is necessarily trivial on $\Kplusdiag$) let $V_\chi$ be the subspace
generated by the function $g\mapsto\chi(\det g)$, then we have a
$\GL_2(\mathbb{A})$-invariant (orthogonal) decomposition
\[ L^{\perp}_{\text{cusp}}=L_{\text{sp}} \oplus L_{\text{cont}},\qquad
L_{\text{sp}}:=\bigoplus_{\chi^2=\omega} V_\chi,\] where
$L_{\text{cont}}$ can be described as follows.

For two Hecke quasicharacters $\chi_1,
\chi_2:K^\times\backslash{\mathbb{A}}^\times\to\CC^\times$ with
$\chi_1\chi_2=\omega$ let $H(\chi_1, \chi_2)$ denote the space of
functions $\varphi : \GL_2(\mathbb{A}) \to \CC$ satisfying
\[\int_{\KK} |\varphi(k)|^2 dk < \infty\]
and
\begin{equation}\label{trafoH0}\varphi\left(\begin{pmatrix} a & x\\0 & b\end{pmatrix}g\right) = \chi_1(a)\chi_2(b) \left|\frac{a}{b}\right|^{1/2}\varphi(g), \qquad
x \in \mathbb{A},\ \ a, b \in \mathbb{A}^{\times}.\end{equation} We
can regard this as the space of functions $\varphi\in L^2(\KK)$
satisfying
\begin{equation}\label{trafoH}
\varphi\left(\left(\begin{matrix} a & x\\0 &
b\end{matrix}\right)k\right) = \chi_1(a)\chi_2(b) \varphi(k), \qquad
\left(\begin{matrix} a & x\\0 & b\end{matrix}\right)\in\KK.
\end{equation}
There is a unique
$s\in\CC$ such that $\chi_1(a)=|a|^s$ and $\chi_2(a)=|a|^{-s}$ for
all $a\in\Kplusdiag$. Accordingly, for $s \in \CC$ we
introduce\footnote{Note that on \cite[p.~224]{GJ},
$\mu\circ\nu^{-1}(a)=|a|^s$ should read
$\mu\circ\nu^{-1}(a)=|a|^{2s}$, cf.\ \cite[(3.11)]{GJ}.}
\begin{equation}\label{H(s)}
 H(s) := \bigoplus_{\substack{\chi_1\chi_2 = \omega\\\text{$\chi_1\chi_2^{-1} = |\cdot|^{2s}$ on $\Kplusdiag$}}} H(\chi_1,
 \chi_2),
\end{equation}
and we view the space $H:=\int_\CC H(s)\,ds$ as a holomorphic fibre
bundle with base $\CC$. For a section $\varphi \in H$ we use the
obvious notations $\varphi(s)\in H(s)$ and $\varphi(s,g) \in \CC$.
The bundle $H$ is trivial, because any $\varphi(s_0)\in H(s_0)$
extends uniquely to a section $\varphi\in H$. There is a
$\GL_2(\mathbb{A})$-equivariant isomorphism
\begin{displaymath}
 S:L_{\text{cont}} \to L'_{\text{cont}} := \int_0^\infty H(iy)\,dy,
\end{displaymath}
given explicitly by \cite[(4.23)]{GJ} on a dense subspace. If we
equip $L'_{\text{cont}}$ with the inner product
\begin{displaymath}
 \langle\varphi_1, \varphi_2\rangle := \frac{2}{\pi} \int_0^\infty \langle\varphi_1(iy), \varphi_2(iy)\rangle\,dy
 = \frac{2}{\pi} \int_0^\infty \int_{\KK} \varphi_1(iy,k)\bar{\varphi}_2(iy,k) \,dk \, dy,
\end{displaymath}
then this map is an isometry by \cite[Section~4, Part~D]{GJ}; in
combination with the theory of Eisenstein series
\cite[Section~5]{GJ} it yields a spectral decomposition of
$L_\text{cont}$. For a section $\varphi\in H$ and for $g \in
\GL_2(\mathbb{A})$ we define the Eisenstein series
\begin{equation}\label{eisensteindef}
 E(\varphi(s), g) := \sum_{\gamma \in P(K)\backslash \GL_2(K)} \varphi(s,\gamma g)
,\qquad\Re s>1/2.\end{equation} This is a holomorphic function which
extends meromorphically to $s \in \CC$ with no poles on $\Re s=0$.
Moreover, for any $s\neq 0$ which is not a pole of $E(\varphi(s),
g)$, we can extract $\varphi(s)\in H(s)$ from the meromorphic
continuation of the constant term as given by \cite[(5.3)]{GJ}. The
above discussion suggests that for $y\in\RR^\times$ we consider the
complex vector space
\[V(iy) := \{E(\varphi(iy)) \mid \varphi(iy)\in H(iy)\}\]
equipped with the inner product
\begin{equation}\label{eisensteininnerproduct}\langle
E(\varphi_1(iy)),E(\varphi_2(iy))\rangle:=
\langle\varphi_1(iy),\varphi_2(iy)\rangle.\end{equation} By
\eqref{H(s)} we have a $\GL_2(\mathbb{A})$-invariant (orthogonal)
decomposition
\[V(iy)=\bigoplus_{\substack{\chi_1\chi_2=\omega \\ \chi_1\chi_2^{-1} = |\cdot|^{2iy}
 \text{ on } K_{\infty, +}^{\text{diag}}}} V_{\chi_1, \chi_2},\]
where
\begin{displaymath}
 V_{\chi_1, \chi_2} := \left\{E(\varphi(iy)) \mid \varphi(iy)\in H(\chi_1, \chi_2)\right\}.
\end{displaymath}
We note that $V_{\chi_1, \chi_2}=V_{\chi_2, \chi_1}$, in particular
$V(iy)=V(-iy)$, by \cite[(4.3), (4.24), (5.15)]{GJ}. Now we have a
$\GL_2(\mathbb{A})$-invariant decomposition
\[L_\text{cont}=\int_0^\infty V(iy)\,dy=\int_0^\infty
 \bigoplus_{\substack{\chi_1\chi_2=\omega \\ \chi_1\chi_2^{-1} = |\cdot|^{2iy}
 \text{ on } K_{\infty, +}^{\text{diag}}}} V_{\chi_1, \chi_2}\ dy.\]
More precisely, by \cite[(4.24), (5.15)--(5.17)]{GJ} any $\phi\in
L_\text{cont}$ can be written as
\[\phi(g)=\frac{1}{\pi}\int_0^\infty E(\varphi(iy),g)\,dy,\qquad\varphi:=S\phi\in L'_\text{cont},\]
and we have Plancherel's identity
\begin{align*}\langle \phi_1,\phi_2\rangle&=\frac{1}{\pi}\int_0^\infty \langle
E(\varphi_1(iy)),\phi_2\rangle\,dy =\frac{1}{\pi}\int_0^\infty
2\langle\varphi_1(iy),\varphi_2(iy)\rangle\,dy\\
&=\frac{2}{\pi}\int_0^\infty\langle
E(\varphi_1(iy)),E(\varphi_2(iy))\rangle\,dy.\end{align*}

To summarize, we have a $\GL_2(\mathbb{A})$-invariant orthogonal
decomposition
\begin{equation}\label{spectral}
 L^2(\GL_2(K)\backslash \GL_2(\mathbb{A}), \omega)= \bigoplus_{\pi\in \mathcal{C}_{\omega}} V_{\pi}
 \oplus \bigoplus_{\chi^2=\omega} V_\chi \oplus \int_0^\infty
 \bigoplus_{\substack{\chi_1\chi_2=\omega \\ \chi_1\chi_2^{-1} = |\cdot|^{2iy}
 \text{ on } K_{\infty, +}^{\text{diag}}}} V_{\chi_1, \chi_2}\ dy
\end{equation}
in the sense that each function in the $L^2$-space decomposes into a
convergent sum and integral of functions from each subspace, and a
Plancherel formula holds. For notational simplicity we shall write
the last term as $\int_{\mathcal{E}_\omega} V_{\varpi}\,d\varpi$,
where $\mathcal{E}_\omega$ is the set of {\it unordered pairs}
$\{\chi_1,\chi_2\}$ of Hecke characters with $\chi_1\chi_2 = \omega$
and nontrivial restrictions on $\Kplusdiag$.

\subsection{Casimir eigenvalues and conductors}

Let $(\pi, V_{\pi})$ be an infinite-dimensional irreducible
automorphic representation of $\GL_2(\mathbb{A})$ occurring in the
spectral decomposition \eqref{spectral}, i.e. one of $V_{\pi}$ with
$\pi\in \mathcal{C}_{\omega}$, or $V_{\chi_1,\chi_2}$ with
$\{\chi_1,\chi_2\}\in\mathcal{E}_\omega$, equipped with the right
$\GL_2(\mathbb{A})$-action. By Flath's Theorem~\cite{Fl}, $V_{\pi}$ decomposes
as a restricted tensor product over the places of $K$,
\begin{equation}\label{flath}
 V_{\pi} = \bigotimes_v V_{\pi_v}.
 \end{equation}
For each $1\leq j\leq d$, the Laplace--Beltrami operator of the
$j$-th component of $\GL_2(K_\infty)=\GL_2(\RR)^d$,
\begin{equation}\label{LB}
\Delta_j:=-y_j^2(\partial^2_{x_j}+\partial^2_{y_j})+y_j\partial_{x_j}\partial_{\vartheta_j},
\end{equation}
acts on the dense subset $V^{\infty}_{\pi}$ of smooth vectors by a
scalar
\[\lambda_{\pi,j}=:\frac{1}{4}-\nu_{\pi,j}^2\in\RR.\]
Here $\nu_{\pi,j}\in\frac{1}{2}\ZZ$ if $\pi_{\infty,j}$ belongs to
the discrete series, and by \cite{KSh} we have
\begin{equation}\label{boundLaplace}
\nu_{\pi,j}\in i\RR\cup[-\theta,\theta],\qquad \theta:=1/9,
\end{equation}
if $\pi_{\infty,j}$ belongs to the principal series or the
complementary series. We shall choose $\nu_{\pi,j}$ so that $\Re
\nu_{\pi,j}\geq 0$ and $\Im \nu_{\pi,j} \geq 0$. For notational
simplicity we write
\begin{equation}\label{short}
\begin{aligned}
\lambda_\pi&:=(\lambda_{\pi,j})_{j=1}^d\in\RR^d,\\
\nu_\pi&:=(\nu_{\pi,j})_{j=1}^d\in\RR^d,\end{aligned}\qquad
\begin{aligned}
\tilde\lambda_\pi&:=(1+|\lambda_{\pi,j}|)_{j=1}^d\in\RR_{>0}^d;\\
\tilde\nu_\pi&:=(1+|\nu_{\pi,j}|)_{j=1}^d\in\RR_{>0}^d;
\end{aligned}
\end{equation}
in particular, \[\N\tilde{\lambda}_{\pi} = \prod_{j=1}^d
(1+|\lambda_{\pi,j}|)\qquad\text{and}\qquad\N\tilde{\nu}_{\pi} =
\prod_{j=1}^d (1+|\nu_{\pi,j}|).\]

Let $\mc_\omega\subseteq\mo$ denote the conductor of the central
character $\omega$, and for a nonzero ideal $\mc \subseteq
\mc_\omega$ let
\begin{displaymath}
 V_{\pi}(\mc) :=\left\{\phi \in V_{\pi} \mid \text{$\phi(g\left(\begin{smallmatrix}a&b\\c&d\end{smallmatrix}\right)) =
 \omega_{\mc}(d)\phi(g)$ for all
 $g\in\GL_2(\mathbb{A})$ and $\left(\begin{smallmatrix}a&b\\c&d\end{smallmatrix}\right) \in
 \KK(\mc)$}\right\},
\end{displaymath}
where
\[\omega_{\mc}(x) :=
\prod_{\mpr \mid\mc} \omega_\mpr(x),\qquad x\in\mathbb{A}^\times.\]
 For $\mc\subseteq\mc'\subseteq\mc_\omega$ we have
$V_\pi(\mc')\subseteq V_\pi(\mc)$, because\footnote{Indeed, for
$\mpr \mid\mc$ and $\mpr \nmid\mc'$ we have $bc\in\mpr \mo_\mpr $,
hence $a,d\in\mo_\mpr^\times$ by $ad-bc\in\mo_\mpr^\times$, so that
$\omega_\mpr (d)=1$ by $\mpr \nmid\mc_\omega$.}
$\omega_\mc(d)=\omega_{\mc'}(d)$ for
$\left(\begin{smallmatrix}a&b\\c&d\end{smallmatrix}\right) \in
 \KK(\mc)$. We
define the conductor $\mc_{\pi}$ of $\pi$ as the largest ideal
$\mc\subseteq\mc_\omega$ such that $V_{\pi}(\mc) \neq \{0\}$ (cf.\
\cite[Theorem~1]{Ca} and \cite[Corollary~2]{Mi}). Analogously, for a
prime $\mpr $ and a nonzero ideal $\mc \subseteq
\mc_{\omega_{\mpr}}$ we define $V_{\pi_{\mpr}}(\mc)$ and the local
conductor $\mc_{\pi_{\mpr}}$. Note that $\mc_{\pi_{\mpr}} =
\mc_{\pi} \mo_{\mpr}$. Finally, we define the analytic conductor of
$\pi$ (cf.\ \cite{IS}) as
\begin{equation}\label{analyticconductor}
C(\pi):=(\N\mc_\pi)(\N\tilde\lambda_\pi).\end{equation}

For any nonzero ideals $\mt,\mc\subseteq\mo$ such that
$\mt\mc_\pi\mid\mc$ there is an isometric embedding of complex vector spaces
\begin{equation}\label{defshift}
R_\mt:V_\pi(\mc_\pi)\hookrightarrow V_\pi(\mc),\qquad\qquad
(R_\mt\phi)(g):= \phi\left(g\left(\begin{smallmatrix}t^{-1}&0\\0&1\end{smallmatrix}\right)\right),
\end{equation}
where $t\in\mathbb{A}_\text{fin}^\times$ is any finite idele
representing $\mt$. It follows from \eqref{flath} and the local
result of Casselman~\cite{Ca} that the spaces $V_\pi(\mc)$ decompose
(in general not orthogonally) as
\begin{equation}\label{RTdecomp}
V_\pi(\mc)=\bigoplus_{\mt\mid\mc\mc_{\pi}^{-1}} R_\mt V_\pi(\mc_\pi)
\qquad\text{for any $\mc\subseteq\mc_\pi$,}
\end{equation}
and $V_{\pi_{\mpr}}(c_{\pi_{\mpr}})$ is one-dimensional for each
prime $\mpr$. For each character $k(\vartheta)\mapsto\exp(i q \cdot
\vartheta)$, $q \in \ZZ^d$, of $\SO_2(K_{\infty})$ we define
\begin{displaymath}
 V_{\pi,q} := \left\{\phi \in V_{\pi} \mid \text{$\phi(gk(\vartheta)) = \exp(i q\cdot \vartheta) \phi(g)$ for all
 $\vartheta\in(\RR/2\pi\ZZ)^d$}\right\},
\end{displaymath}
and correspondingly we write
\[V_{\pi, q}(\mc):=V_{\pi,q}\cap V_{\pi}(\mc).\]
This gives an orthogonal decomposition (in a Hilbert space sense)
\begin{equation}\label{weightdecomp}
 V_{\pi}(\mc) = \bigoplus_{q \in \ZZ^d}V_{\pi, q}(\mc) \qquad\text{for any $\mc\subseteq\mc_\pi$},
\end{equation}
and also a decomposition of vector spaces
\begin{equation}\label{oldforms}
V_{\pi,q}(\mc)=\bigoplus_{\mt\mid\mc\mc_{\pi}^{-1}} R_\mt
V_{\pi,q}(\mc_\pi) \qquad\text{for any $\mc\subseteq\mc_\pi$},
\end{equation} where $V_{\pi, q}(\mc_{\pi})$ is at most
one-dimensional. Alternately, \eqref{oldforms} and global
multiplicity-one were also established by Miyake~\cite{Mi} which
then imply the above local results.

In the case of $V_\pi=V_{\chi_1,\chi_2}$ consisting of Eisenstein
series we replace the subscript $\pi$ by $\chi_1,\chi_2$ for
convenience, e.g.\ we write $\mc_{\chi_1,\chi_2}:=\mc_\pi$. For each
$1\leq j\leq d$ we have
\begin{equation}\label{eiseneig}
\lambda_{\chi_1,\chi_2,j}=\frac{1}{4}-s_j^2,\qquad
\nu_{\chi_1,\chi_2,j}=\pm s_j,\end{equation} where $s_j \in i\RR$
denotes the unique exponent such that $\chi_1\chi_2^{-1} =
|\cdot|^{2s_j}$ on the $j$-th component of $\Kplus$. We note that
$\chi_1\chi_2^{-1}=|\cdot|^{2iy}$ on $\Kplusdiag$, where
\begin{equation}\label{ydef}
 y:= \frac{s_1+s_2+\cdots+s_d}{i d}\in\RR^\times.
\end{equation}
It follows from the discussion in Section~\ref{spectralsection} that
$V_{\chi_1, \chi_2}$ and $H(\chi_1, \chi_2)$ are isomorphic
representations, in particular there is a decomposition
\begin{equation}\label{Htensor}
 H(\chi_1, \chi_2) = \bigotimes_v H_v(\chi_1, \chi_2).
\end{equation}
In addition,
\begin{equation}\label{Vparam}
 V_{\chi_1, \chi_2}(\mc) = \{E(\varphi(iy), \cdot) \in V_{\chi_1, \chi_2} \mid \varphi \in H(\chi_1, \chi_2, \mc)\},
\end{equation}
and by \eqref{RTdecomp},
\begin{equation*}
 H(\chi_1, \chi_2, \mc) = \bigoplus_{\mt\mid\mc\mc_{\pi}^{-1}} R_\mt H(\chi_1, \chi_2, \mc_{\pi}).
\end{equation*}
Here it is known (e.g.\ \cite[p.~306]{Ca}) that
$\mc_\pi=\mc_{\chi_1, \chi_2}=\mc_{\chi_1}\mc_{\chi_2}$. In
Section~\ref{explicit} we shall give a detailed proof of this fact
for trivial central character.

Finally for $\mc \subseteq \mc_\omega$ we define, in harmony with
the notation of the previous section,
\begin{equation}\label{defCE}
\begin{split}
\mathcal{C}_{\omega}(\mc) & := \bigl\{\pi\in\mathcal{C}_\omega\mid \mc \subseteq \mc_{\pi} \subseteq \mc_{\omega} \bigr\},\\
\mathcal{E}_{\omega}(\mc) & := \bigl\{\{\chi_1,
\chi_2\}\in\mathcal{E}_\omega\mid \mc \subseteq \mc_{\chi_1, \chi_2}
\subseteq \mc_{\omega} \bigr\},
\end{split}
\end{equation}
and we shall drop the subscript $\omega$ in case $\omega$ is
trivial.

\subsection{Normalized Whittaker functions}

Let $q\in\ZZ$. For $q$ even, let $\nu\in (\frac{1}{2}+\ZZ)\cup
i\RR\cup(-\frac{1}{2},\frac{1}{2})$. For $q$ odd, let $\nu\in\ZZ\cup
i\RR$. For these parameters we define the normalized Whittaker
function
\begin{equation}\label{normwhit}
 \tilde{W}_{\frac{q}{2}, \nu}(y) :=
 \frac{i^{\sgn(y)\frac{q}{2}}
 W_{\sgn(y)\frac{q}{2}, \nu}(4\pi|y|)}
 {\left\{\Gamma(\frac{1}{2}-\nu+\sgn(y)\frac{q}{2})\Gamma(\frac{1}{2}+\nu+\sgn(y)\frac{q}{2})\right\}^{1/2}},
\qquad y\in\RR^\times,
\end{equation}
where $W_{\alpha, \beta}$ is the standard Whittaker function, see
\cite[Chapter XVI]{WW}. The right hand side is understood as $0$ if
one of $\frac{1}{2} \pm \nu + \sgn(y)\frac{q}{2}$ is a nonpositive
integer, otherwise we have a positive number under the square-root
sign by the constraints on $\nu$. We note that the above definition
is invariant under $\nu\to -\nu$, and for future reference we record
that
\begin{equation}\label{normwhit2}
\tilde{W}_{\frac{q}{2},
\nu}(y)=\eta_{q,\nu}\frac{W_{\sgn(y)\frac{q}{2},
\nu}(4\pi|y|)}{\Gamma(\frac{1}{2}+\nu+\sgn(y)\frac{q}{2})},\qquad
q\in 2\ZZ,\quad \nu\in i\RR,\quad y\in\RR^\times,
\end{equation}
where $\eta_{q,\nu}$ is a constant of modulus $1$ depending on $q$
and $\nu$ but not on $y$.\footnote{This can be proved by induction on $q$, starting from
the trivial case $q=0$. Note that \eqref{normwhit2} is only stated for
special $q$ and $\nu$.}

By \cite[Section 4]{BrMo}, the functions $\tilde{W}_{q/2, \nu}$
$(q\in\ZZ)$ for fixed $\nu$ form an orthonormal basis of the Hilbert
space $L^2(\RR^{\times}, d^\times y)$ which justifies our
normalization:
\begin{equation}\label{orth}
 L^2(\RR^{\times}, d^\times y) = \bigoplus_{q \in \ZZ} \CC \tilde{W}_{\frac{q}{2},
 \nu},\qquad
\langle\tilde{W}_{\frac{q}{2},\nu},\tilde{W}_{\frac{q'}{2},\nu}\rangle=\delta_{q,q'}.
\end{equation}
We review the uniform bounds \cite[(24)--(26)]{BH}. For all $\nu$ we
have
\begin{equation}\label{whittaker1}
 \tilde{W}_{\frac{q}{2}, \nu}(y) \ll |y|^{1/2} \left(\frac{|y|}{|q| + |\nu| + 1}\right)^{-1-|\Re\nu|} \exp\left(-\frac{|y|}{|q| + |\nu| +
 1}\right).
\end{equation}
For $\nu\in\frac{1}{2}\ZZ\cup i\RR$ we have, for any $0 < \eps
<1/4$,
\begin{equation}\label{whittaker2}
 \tilde{W}_{\frac{q}{2}, \nu}(y) \ll_\eps |y|^{1/2-\eps} (|q| + |\nu| + 1).
\end{equation}
For $\nu\in(-\frac{1}{2},\frac{1}{2})$ we have, for any $0 < \eps <
1$,
\begin{equation}\label{whittaker3}
 \tilde{W}_{\frac{q}{2}, \nu}(y) \ll_\eps |y|^{1/2-|\nu| -\eps} (|q| + |\nu| + 1)
^{1+|\nu|}.
\end{equation}
For $q \in \ZZ^d$ and appropriate $\nu\in\CC^d$, we define
\begin{equation}\label{convent}
\tilde{W}_{\frac{q}{2}, \nu}(y) := \prod_{j=1}^d \tilde{W}_{\frac{q_j}{2},
\nu_j}(y_j),\qquad y \in K_\infty^\times.\end{equation}

\subsection{Hecke eigenvalues and Fourier expansion}

Let $\mc \subseteq \mc_\pi$ be a nonzero ideal. There is a character
$\eps_\pi:\{\pm 1\}^d\to\{\pm 1\}$ depending only on $\pi$ such that
every $\phi \in V_{\pi, q}(\mc)$ has a Fourier--Whittaker expansion
(cf.\ \cite[(2.11), (3.8)]{KM})
\begin{equation}\label{fourierwhittaker}
\phi\left(\begin{pmatrix} y & x\\ 0 & 1\end{pmatrix}\right) =
\rho_{\phi,0}(y) + \sum_{r \in K^\times} \rho_{\phi}(r
y_{\text{fin}})\eps_\pi(\sgn(r y_\infty))\tilde{W}_{q/2,
\nu_{\pi}}(r y_{\infty}) \psi(rx)
\end{equation}
for $y = y_{\infty}\times y_{\text{fin}}\in \mathbb{A}^{\times}$, $x
\in \mathbb{A}$. Note that for any
$y_\text{fin}\in\mathbb{A}_\text{fin}^\times$ the coefficient
$\rho_\phi (y_\text{fin})$ only depends on the fractional ideal
represented by $y_\text{fin}$ and it is nonzero only if this ideal
is integral. The normalization of $\tilde{W}_{q/2, \nu_{\pi}}$ is
further justified by the fact that these coefficients remain
unchanged if $\phi$ is replaced by any of its nonzero Maa{\ss}
shifts.

If $(\pi, V_{\pi})$ is one of the right $\GL_2(\mathbb{A})$-spaces
$V_\chi$ with $\chi\in X$ in \eqref{spectral}, then the expansion
\eqref{fourierwhittaker} only contains the constant term
$\rho_0(y)$.

Let us now assume that $(\pi, V_{\pi})$ is one of the right
$\GL_2(\mathbb{A})$-spaces $V_{\pi}$ with $\pi\in
\mathcal{C}_{\omega}$ in \eqref{spectral}, so that $\rho_0(y)=0$.
The finer structure of the coefficients $\rho_\phi$ can be revealed
by the theory of Hecke operators, as developed by Miyake~\cite{Mi}
(see also \cite[Section~2]{Sh1} and \cite[Section~2]{KM}). By
\eqref{oldforms} we can decompose any vector $\phi \in V_{\pi,
q}(\mc)$ as
\[\phi=\sum_{\mt\mid\mc\mc_{\pi}^{-1}} R_\mt\phi_\mt,\]
where each $\phi_\mt$ lies in $V_{\pi, q}(\mc_\pi)$. By
\eqref{fourierwhittaker} we infer
\begin{equation*}
\rho_{\phi}(\mathfrak{m})=\sum_{\mt\mid\mc\mc_{\pi}^{-1}}
\rho_{R_\mt\phi_\mt}(\mathfrak{m}) =\sum_{\mt\mid\mc\mc_{\pi}^{-1}}
\rho_{\phi_\mt}(\mathfrak{m}\mt^{-1}),\qquad\mathfrak{m}\subseteq\mo,
\end{equation*}
so that we can focus our attention to the case when $\mc=\mc_\pi$,
i.e.\ when $\phi\in V_{\pi, q}(\mc_\pi)$ is a newform. For each
nonzero ideal $\mathfrak{m} \subseteq \mo$ the Hecke operator
$T_{\mc_\pi}(\mathfrak{m})$ acts on $V_{\pi}(\mc_{\pi})$ by a scalar
$\lambda_{\pi}(\mathfrak{m})$. The function $\lambda_{\pi}$
satisfies
\begin{equation}\label{hecke}
 \lambda_{\pi}(\mathfrak{m})\lambda_{\pi}(\mathfrak{n}) =
 \sum_{\ma \mid\gcd(\mathfrak{m},\mathfrak{n})}
 \omega_{\pi}(\ma )\lambda_{\pi}(\mathfrak{m}\mathfrak{n}\ma^{-2}),
\end{equation}
and
\begin{equation}\label{hecke+}
 \lambda_{\pi}(\mathfrak{m}) = \omega_{\pi}(\mathfrak{m})
 \bar{\lambda}_{\pi}(\mathfrak{m}),\qquad\gcd(\mathfrak{m},\mc_\pi)=\mo,
\end{equation}
where $\omega_\pi:I(K)\to\CC$ is defined as follows: if $\ma \in
I(K)$ is coprime to $\mc_{\pi}$ then $\omega_{\pi}(\ma ): =
\omega(a)$ where $a \in \mathbb{A}_{\text{fin}}^{\times}$ is any
finite idele representing $\ma $ with $a_\mpr = 1$ for $\mpr \mid
\mc_{\pi}$, otherwise $\omega_{\pi}(\ma ): = 0$. In particular,
$\lambda_{\pi}$ is multiplicative on the set of nonzero integral
ideals. The non-archimedean analogue of \eqref{boundLaplace} is
\begin{equation}\label{boundlambda}
 \lambda_{\pi}(\mathfrak{m}) \ll_{\eps} (N\mathfrak{m})^{\theta +\eps},\qquad
 \theta:=1/9,
\end{equation}
for any $\eps> 0$, see \cite{KSh}. It follows, as stated after
\eqref{oldforms}, that each $V_{\pi, q}(\mc_{\pi})$ is at most
one-dimensional, and in fact
\begin{equation*}
 \rho_{\phi}(\mathfrak{m}) = \frac{\lambda_{\pi}(\mathfrak{m})}{\sqrt{\N\mathfrak{m}}}\rho_\phi(\mo),
 \qquad\mathfrak{m}\subseteq\mo,\qquad \phi\in
V_{\pi, q}(\mc_{\pi}).
\end{equation*}
To maintain this identity we define $\lambda_\pi(\ma)$ to be zero
for any nonintegral $\ma\in I(K)$. Comparing with
\eqref{fourierwhittaker} we see that for $\phi\in V_{\pi,
q}(\mc_{\pi})$ we have
\begin{equation}\label{fouriersimple}
\phi\left(\left(\begin{matrix} y & x\\ 0 &
1\end{matrix}\right)\right) = \sum_{r \in K^\times}
\frac{\lambda_{\pi}(r y_{\text{fin}})}{\sqrt{\N (r y_{\text{fin}})}}
W_{\phi}(r y_{\infty}) \psi(rx),\qquad y\in \mathbb{A}^{\times},\ \
x \in \mathbb{A},
\end{equation}
where
\begin{equation}\label{explicitKir}
 W_\phi(y)=\rho_\phi(\mo)\eps_\pi(\sgn(y))\tilde{W}_{q/2,
\nu_{\pi}}(y), \qquad y\in K_\infty^\times,\qquad \phi\in V_{\pi,
q}(\mc_{\pi}).
\end{equation}
An intrinsic definition of $W_\phi$ becomes apparent upon choosing
$y_\text{fin}=(1,1,\dots)$ and $x_\text{fin}=(0,0,\dots)$ in
\eqref{fouriersimple} and picking by orthogonality and
$\vol(K\backslash\mathbb{A})=1$ the term corresponding to $r=1$:
\begin{equation}\label{Kir}W_{\phi}(y) := \int_{K\backslash\mathbb{A}} \phi\left(\begin{pmatrix} y & x\\ 0 &
1\end{pmatrix}\right) \psi(-x) \,dx,\qquad y\in K_\infty^\times.
\end{equation}
We have verified \eqref{fouriersimple} and \eqref{Kir} for pure
weight newforms $\phi\in V_{\pi, q}(\mc_{\pi})$ but then, by
linearity, it extends to all smooth vectors $\phi\in
V_\pi^\infty(\mc_{\pi})$. Using also \eqref{orth}, we obtain a
linear map from $V_\pi^\infty(\mc_{\pi})$ to a dense subspace of
$L^2(K_\infty^\times,d^\times y)$ given by $\phi\mapsto W_\phi$. We
will prove in Section~\ref{Lfunctionsection} that
\begin{equation}\label{constant}
\langle\phi_1,\phi_2\rangle=C_\pi\langle
W_{\phi_1},W_{\phi_2}\rangle\end{equation} for some positive
constant $C_\pi$ depending only on $\pi$. It follows that the map
$\phi\mapsto W_\phi$ extends to a vector space isomorphism
$V_{\pi}(\mc_{\pi})\to L^2(K_\infty^\times,d^\times y)$, called the
(archimedean) Kirillov map of $\pi$, and Lemma~\ref{lemma2} below
shows that it is essentially an isometry (i.e.\ $C_\pi\approx 1$).
In particular, \eqref{fouriersimple} and \eqref{constant} hold for
all $\phi\in V_\pi(\mc_\pi)$.\\

Now let $\mc \subseteq \mc_{\pi}$ be any ideal. It will be important
to investigate in detail vectors in the larger space $V_{\pi}(\mc)$,
classically called ``oldforms". The proofs of the following facts
depend partly on the theory of Eisenstein series that we will
develop in later sections (independently of the present statements,
of course).

As mentioned earlier, the decomposition \eqref{RTdecomp} is in
general not orthogonal. However, by a Gram--Schmidt
orthogonalization process based on \eqref{shiftedinner} below, we
find for each pair of integral ideals $(\mathfrak{s}, \mt)$ with
$\mathfrak{s} \mid \mt \mid \mc\mc_{\pi}^{-1}$ complex numbers
$\alpha_{\mt, \mathfrak{s}}$ such that
\begin{equation}\label{newR}
 R^{(\mt)} := \sum_{\mathfrak{s} \mid \mt} \alpha_{\mt, \mathfrak{s}} R_{\mathfrak{s}} :
 V_{\pi}(\mc_{\pi}) \hookrightarrow V_{\pi}(\mc)
 ,\qquad \mt\mid\mc\mc_{\pi}^{-1},
\end{equation}
are isometric embeddings with pairwise orthogonal images, and
$R^{(\mo)}$ is the identical inclusion map. This yields an
orthogonal decomposition
\begin{equation}\label{RTdecomp2}
V_\pi(\mc)=\bigoplus_{\mt\mid\mc\mc_{\pi}^{-1}}
R^{(\mt)}V_\pi(\mc_\pi) \qquad\text{for any $\mc\subseteq\mc_\pi$,}
\end{equation}
and an extension of the Kirillov map \eqref{Kir} to each subspace
$R^{(\mt)}V_\pi(\mc_\pi)$. Namely, by \eqref{fouriersimple} every
$\phi \in R^{(\mt)}V_{\pi}(\mc_{\pi})$ has a Fourier expansion
\begin{equation}\label{Fourgen}
 \phi\left(\left(\begin{matrix} y & x\\ 0 & 1\end{matrix}\right)\right) =
 \sum_{r \in K^{\times}} \frac{\lambda^{(\mt)}_{\pi}(r y_{\text{fin}}) }{\sqrt{\N(r y_{\text{fin}})}} W_{\phi}(r y_{\infty}) \psi(rx)
 ,\qquad y\in \mathbb{A}^{\times},\ \ x \in \mathbb{A},
\end{equation}
where
\begin{equation}\label{oldKir}
W_\phi:=W_{(R^{(\mt)})^{-1}\phi}\qquad\text{and}\qquad
\lambda^{(\mt)}_{\pi}(\mathfrak{m}) := \sum_{\mathfrak{s}
\mid\gcd(\mt,\mathfrak{m})} \alpha_{\mt, \mathfrak{s}}
(\N\mathfrak{s})^{1/2} \lambda_{\pi}
(\mathfrak{m}\mathfrak{s}^{-1}).\end{equation} It is clear that
\eqref{constant} holds true when extended to $\phi_1,\phi_2\in
R^{(\mt)}V_\pi(\mc_\pi)$, and Lemma~\ref{lemma2} below shows that
\begin{equation}\label{newWphibound3}C(\pi)^{-\eps} \| \phi \| \ll_{K, \eps} \| W_\phi \|
\ll_{K, \eps} C(\pi)^{\eps} \| \phi \|,\qquad\phi\in
R^{(\mt)}V_\pi(\mc_\pi),\end{equation} with implied constants
depending only on $K$ and $\eps$.

\begin{remark}
If $\mc$ is squarefree, then the orthogonalization can be carried
out completely explicitly by combining \eqref{shiftedinner} below
with the Hecke relations \eqref{hecke}--\eqref{hecke+} above (see
e.g.\ \cite[Prop.~2.6]{ILS}), and one obtains $\alpha_{\mt,
\mathfrak{s}} \ll_\eps (\N\mt\mathfrak{s}^{-1})^{\theta-1/2 +
\eps}$. For general ideals $\mc$, this seems much harder.
\end{remark}

\subsection{Parametrizing Eisenstein series}\label{explicit}

For simplicity we shall assume in the following three sections that
the central character $\omega$ is trivial, since this is all we need
for our purposes. The general case, however, is quite similar. We
can assume that $\chi_1=\chi$ and $\chi_2=\chi^{-1}$, where $\chi$
is a Hecke character which is nontrivial on $\Kplusdiag$. Let us
denote the conductor of $\chi$ by $\mc_{\chi}$, and for an arbitrary
place $v$ of $K$ let us write $\chi_v$ for the restriction of $\chi$
to the quasifactor $K_v^\times$ of $\mathbb{A}^\times$. Note that
$\mc_{\chi_\mpr}=\mc_\chi\mo_\mpr$ for each prime $\mpr$. For every
prime $\mpr$ we fix a prime element $\varpi_\mpr\in\mo_\mpr$ (i.e.\
$v_\mpr(\varpi_\mpr)=1$) and we shall use the convention that
$v_\mpr(0)=\infty$. For the purpose of this paper we could get by
with less information than provided in this section, but we have
preferred to give rather precise results.

\begin{lemma}\label{eisenstein}
The conductor of $H(\chi, \chi^{-1})$ is $\mc_{\chi}^2$. More
precisely, let $\mpr$ be a prime, $\varpi:=\varpi_\mpr$,
$r:=v_\mpr(\mathfrak{d})$, $m:= v_\mpr(\mc_{\chi})$. For any integer
$n\geq 0$ the complex vector space $H_{\mpr}(\chi, \chi^{{-1}},
\mpr^n)$ has dimension $\max(0,n-2m+1)$. For $n\geq 2m$ an
orthogonal basis is $\{\varphi_{\mpr,j}\mid 0\leq j\leq n-2m\}$,
where the functions $\varphi_{\mpr,j}:\KK(\mo_\mpr)\to\CC$ are
defined as follows.
\begin{itemize}
\item[$\bullet$]When $m=0$ (i.e.\ $\chi$ is unramified at $\mpr $) and $k = \begin{pmatrix} \ast & \ast\\
b \varpi^{r} & \ast\end{pmatrix}\in\KK(\mo_\mpr)$,
\begin{equation}\label{functionphi4} \varphi_{\mpr,0}(k):=1;\qquad
\varphi_{\mpr,1}(k):=
\begin{cases}(\N\mpr)^{-1/2}, & v_\mpr(b)=0,\\
-(\N\mpr)^{1/2}, & v_\mpr(b)\geq 1;
\end{cases}\end{equation}
and for $j\geq 2$,
\begin{equation}\label{functionphi3}
\varphi_{\mpr,j}(k):=
\begin{cases}0,&v_\mpr(b)\leq j-2,\\
-(\N\mpr)^{j/2-1},&v_\mpr(b)=j-1,\\
(\N\mpr)^{j/2}\left(1-\frac{1}{\N\mpr}\right),&v_\mpr(b)\geq j.
\end{cases}
\end{equation}
\item[$\bullet$]When $m>0$ (i.e.\ $\chi$ is ramified at $\mpr$) and $k = \begin{pmatrix} a & \ast\\
b \varpi^{r} & \ast\end{pmatrix}\in\KK(\mo_\mpr)$,
\begin{equation}\label{functionphi2}\varphi_{\mpr,j}(k):=\begin{cases}
(\N\mpr)^{(m+j)/2}\chi_\mpr(a b^{-1}), & v_\mpr(b)=m+j,\\
 0, & v_\mpr(b)\neq m+j.
\end{cases}\end{equation}
\end{itemize}
\end{lemma}

\begin{remark}
The basis exhibited above is close to orthonormal. Using
\[\left[\KK(\mo_\mpr):\KK((\varpi^j))\right]=(\N\mpr)^j\left(1+\frac{1}{\N\mpr}\right),\qquad j\geq 1,\]
it is straightforward to see that
\begin{equation}\label{localphinorm}
1-\frac{1}{\N\mpr}\leq\|\varphi_{p,j}\|\leq 1,\qquad j\geq 0,
\end{equation}
with equality on the right hand side for $j=0$.
\end{remark}

\begin{proof}
For the argument below it is useful to keep in mind that for any
nonzero ideal $\mc \subseteq \mo_\mpr$
\[\KK(\mc)=\begin{pmatrix}\varpi^{r}&0\\0 & 1\end{pmatrix}^{-1}\KK_0(\mc)\begin{pmatrix}\varpi^{r}&0\\0 &
1\end{pmatrix},\qquad \KK_0(\mc):=\left\{\begin{pmatrix} a &b\\ c &
d\end{pmatrix}\in \GL_2(\mo_\mpr)
 \,\bigg|\ c \in \mc\right\}.\]
In particular, $\KK(\mc)$ has the same measure as $\KK_0(\mc)$.

We can regard $H_{\mpr}(\chi, \chi^{-1}, \mpr^n)$ as a subset of
functions on $N(\mo_\mpr)\backslash \KK(\mo_{\mpr})/\KK((\varpi^n))$
with $N(\mo_\mpr):= \left\{\left(\begin{smallmatrix} 1 & x\\ 0 &
1\end{smallmatrix}\right) \mid x \in
\mathfrak{d}_\mpr^{-1}\right\}$. A set of double coset
representatives for $N(\mo_\mpr)\backslash
\KK(\mo_\mpr)/\KK((\varpi^n))$ is given by any collection
\begin{equation*}
 \left\{\left(\begin{matrix} a & \ast\\ \varpi^{r+j} & \ast\end{matrix}\right) \in \KK(\mo_\mpr) \,\bigg|\
 0 \leq j \leq n, \, a \in \mo_\mpr^{\times}, \, a\text{ mod }\varpi^{\min(j, n-j)}\right\},
\end{equation*}
where for given $a$ and $j$ any choice of $\ast$ is admissible. To
see this, we observe first that by \cite[Proof on p.~25 and Errata
on p.~269]{Sh}, this set has the right cardinality. Moreover, two
such representatives determine different double cosets. Indeed,
multiplying a representative from the left by elements of
$N(\mo_\mpr)$ and from the right by elements of $\KK((\varpi^n))$
does not change the valuation of the lower left entry if $j < n$,
and it can at most increase the valuation if $j=n$, but all
representatives have $j \leq n$, so we conclude that different
values of $j$ correspond to different double
cosets. In addition, if $ \bigl(\begin{smallmatrix} a & \ast\\
\varpi^{r+j} &
\ast\end{smallmatrix}\bigr)$ and $ \bigl(\begin{smallmatrix} a' & \ast\\
\varpi^{r+j} & \ast\end{smallmatrix}\bigr)$ are in the same double
coset, then
\[\begin{pmatrix} a' & \ast\\ \varpi^{r+j} & \ast\end{pmatrix}^{-1} \begin{pmatrix} 1 & x\\ 0 & 1\end{pmatrix} \begin{pmatrix} a & \ast\\ \varpi^{r+j} & \ast\end{pmatrix}\in \KK((\varpi^n)),\]
whence $a'\varpi^j-a\varpi^j-x\varpi^{2j}\in (\varpi^n)$. This
forces $a=a'$ if $0 < j < n$, whereas in the remaining cases $j = 0$
and $j=n$ only $a=a' = 1$ is allowed.

By a variant of the argument above we
see that each $\bigl( \begin{smallmatrix} a& b \\
c & d \end{smallmatrix} \bigr)\in\KK(\mo_\mpr)$ is
represented by some $\bigl(\begin{smallmatrix} a'& \ast \\
\varpi^{r+j} & \ast \end{smallmatrix}\bigr)\in\KK(\mo_\mpr)$ with
$a'\in\mo_\mpr^\times$ and $j=\min(v_\mpr(c)-r,n)$. Now for any
$\varphi\in H_{\mpr}(\chi, \chi^{-1}, \mpr^n)$ the transformation
rule \eqref{trafoH} shows that
\[\varphi\left(\begin{pmatrix}a&b\\c&d\end{pmatrix}\right)=
\varphi\left(\begin{pmatrix}a'&\ast\\\varpi^{r+j}&\ast\end{pmatrix}\right)=
\chi_\mpr(a')\varphi\left(\begin{pmatrix}1&\ast\\\varpi^{r+j}&\ast\end{pmatrix}\right),\qquad
a\in\mo_\mpr^\times,\] hence $\varphi$ is
determined by the $n+1$ values $\varphi\bigl(\bigl(\begin{smallmatrix} 1 & \ast \\
\varpi^{r+j} & \ast
\end{smallmatrix}\bigr)\bigr)$ with $0\leq j\leq
n$. By the discussion of representatives we can further see that
$\varphi\bigl(\bigl(\begin{smallmatrix} 1 & \ast \\
\varpi^{r+j} & \ast
\end{smallmatrix}\bigr)\bigr)\neq 0$ implies $\chi_\mpr(a')=1$ for any $a'\in
1+(\varpi^{\min(j,n-j)})$, i.e.\ $m\leq j\leq n-m$. In other words,
$\varphi$ is determined by the $n-2m+1$ values $\varphi\bigl(\bigl(\begin{smallmatrix} 1 & \ast \\
\varpi^{r+j} & \ast
\end{smallmatrix}\bigr)\bigr)$ with $m\leq j\leq
n-m$, because the rest of the $n+1$ values are zero. The dependence
on these values is linear, hence the dimension of $H_{\mpr}(\chi,
\chi^{-1}, \mpr^n)$ is at most $n-2m+1$. For $m=0$ it is
straightforward to check that the $n+1$ functions $\varphi_{\mpr,j}$
for $0\leq j\leq n$ defined by \eqref{functionphi4} and
\eqref{functionphi3} lie in $H_{\mpr}(\chi, \chi^{-1}, \mpr^n)$ and
by
$\left[\KK(\mo_\mpr):\KK((\varpi^j))\right]=(\N\mpr)^j\left(1+\frac{1}{\N\mpr}\right)$
for $j\geq 1$ they are pairwise orthogonal. For $m>0$ it is
straightforward to check that the $n-2m+1$ functions
$\varphi_{\mpr,j}$ for $0\leq j\leq n-2m$ defined by
\eqref{functionphi2} lie in $H_{\mpr}(\chi, \chi^{-1}, \mpr^n)$ and
they are pairwise orthogonal because their supports are pairwise
disjoint. Any orthogonal system is linearly independent, hence the
proof of the lemma is complete.
\end{proof}

The lemma can be combined with \eqref{Htensor} to obtain an
orthogonal basis of $H(\chi, \chi^{-1}, \mc)$ for each ideal
$\mc\subseteq\mc_{\chi}^2$. Namely, for any
$\mt\mid\mc\mc_\chi^{-2}$ and any $q \in (2\ZZ)^d$, we define
$\varphi^{(\mt,q)}:\KK\to\CC$ to be the tensor product of the local
functions $\varphi^{(\mt,q)}_\infty(k(\vartheta)):= e^{i q
\vartheta}$ and $\varphi^{(\mt,q)}_\mpr:=\varphi_{\mpr,v_\mpr(\mt)}$
as in the lemma. These global functions form an orthogonal basis of
$H(\chi, \chi^{-1}, \mc)$; extending them to $\GL_2(\mathbb{A})$ by
\eqref{trafoH0}, the corresponding vectors
$\phi^{(\mt,q)}:=E(\varphi^{(\mt,q)})$ form an orthogonal basis of
$V_{\chi,\chi^{-1}}(\mc)$ by \eqref{Vparam} and
\eqref{eisensteininnerproduct}. We obtain isometric embeddings
$R^{(\mt)}:V_{\chi,\chi^{-1}}(\mc_\chi^2) \hookrightarrow
V_{\chi,\chi^{-1}}(\mc)$ by defining
$R^{(\mt)}:\phi^{(\mo,q)}/\|\phi^{(\mo,q)}\|\mapsto\phi^{(\mt,q)}/\|\phi^{(\mt,q)}\|$
for all $q \in (2\ZZ)^d$. These yield an orthogonal decomposition
\begin{equation}\label{RTdecomp3}
V_{\chi,\chi^{-1}}(\mc)=\bigoplus_{\mt\mid\mc\mc_\chi^{-2}}
R^{(\mt)} V_{\chi,\chi^{-1}}(\mc_\chi^2) \qquad\text{for any
$\mc\subseteq\mc_\chi^2$,}
\end{equation}
similarly as in the cuspidal case, see \eqref{RTdecomp2}. In the
next section we shall exhibit for each nonzero ideal
$\mt\subseteq\mo$ a vector space isomorphism $R^{(\mt)}
V_{\chi,\chi^{-1}}(\mc_\chi^2)\to L^2(K_\infty^\times,d^\times y)$,
written as $\phi\mapsto W_\phi$, such that every $\phi \in R^{(\mt)}
V_{\chi,\chi^{-1}}(\mc_\chi^2)$ has a Fourier expansion with similar
features as in the cuspidal case (cf.\
\eqref{Fourgen}--\eqref{newWphibound3}):
\begin{equation}\label{fouriergeneralEisen2}
\phi\left( \left(\begin{matrix} y & x\\ 0 &
1\end{matrix}\right)\right) = \rho_{\phi,0}(y) + \sum_{r \in
K^\times} \frac{\lambda^{(\mt)}_{\chi,\chi^{-1}}(r
y_{\text{fin}})}{\sqrt{\N (r y_{\text{fin}})}} W_{\phi}(r
y_{\infty}) \psi(rx),\qquad y\in \mathbb{A}^{\times},\ \ x \in
\mathbb{A},
\end{equation}
where
\begin{equation}\label{boundFourierEisen3}
\lambda^{(\mt)}_{\chi,\chi^{-1}}(\mathfrak{m})\ll_{K,\eps}
(\N\gcd(\mt,\mathfrak{m}))(\N\mathfrak{m})^\eps,\qquad
\mathfrak{m}\subseteq\mo,
\end{equation}
\begin{equation}\label{newWphibound2}
\|W_{\phi}\|\ll_{K,\eps}
(\N\mt)^{\eps}C(\chi,\chi^{-1})^{\eps}\|\phi\|,\qquad \phi\in
R^{(\mt)} V_{\chi,\chi^{-1}}(\mc_\chi^2).\end{equation}

\bigskip

Next we prove a density result about the Eisenstein spectrum.

\begin{lemma}\label{babyWeyl} For a nonzero
ideal $\mc\subseteq\mo$ write $\mc = \mc_1^2 \mc_2$ with $\mc_2$
squarefree. In the notation of \eqref{eiseneig} and \eqref{defCE} we
have, for any $X\geq 1$ and any $\eps>0$,
\begin{displaymath}
 \int\limits_{\substack{\varpi\in\mathcal{E}_\omega(\mc)\\|\nu_{\varpi,j}|\leq X}} 1 \,\, d\varpi \ll_{K,\eps}
X^{d}(\N\mc_1)^{1+\eps}.
\end{displaymath}
\end{lemma}

\begin{proof} We need to estimate the measure of the set of Hecke character pairs
$\{\chi,\chi^{-1}\}$ for which $\mc\subseteq\mc_{\chi,\chi^{-1}}$
and $\chi=|\cdot|^{s_j}$ with some $s_j\in i[-X,X]$ on the $j$-th
component of $\Kplus$. We write $\chi=\chi_\infty\chi_\text{fin}$
with the obvious convention. By Lemma~\ref{eisenstein}, $\mc_\chi$
must divide $\mc_1$, hence the number of possibilities for the
restriction of $\chi_\text{fin}$ to
$\Omega=\prod_\mpr\mo_\mpr^\times$ is
$\ll_{\eps}(\N\mc_1)^{1+\eps}$. We fix therefore any character
$\xi:\Omega\to S^1$ and estimate the measure of the set of
Hecke character pairs $\{\chi,\chi^{-1}\}$ for which
$\chi_\text{fin}$ agrees with $\xi$ on $\Omega$ and $(s_1, \ldots,
s_d)\in i[-X,X]^d$. If this set is empty, we are done, otherwise we
fix some element $\chi_0$ in it. Any $\chi$ in the set has the
feature that $\tilde\chi:=\chi\chi_0^{-1}$ is trivial on $\Omega$,
moreover $\tilde\chi=|\cdot|^{s_j}$ with $(s_1, \ldots, s_d)\in
i\mathcal{B}$, where $\mathcal{B}:=[-2X,2X]^d$. As $\tilde\chi$ is
also trivial on $K$, its infinite part $\tilde\chi_\infty$ is
trivial on the set of totally positive units $U^{+}$ embedded in
$\Kplus$. Let $\{u_1, \ldots, u_{d-1}\}$ be a generating set of
$U^{+}$, and put
\[M := \begin{pmatrix}1 &\cdots& 1\\ \log u_1^{\sigma_1} & \cdots & \log u_1^{\sigma_d} \\ \vdots & &\vdots \\ \log u_{d-1}^{\sigma_1} & \cdots & \log u_{d-1}^{\sigma_d} \end{pmatrix} \in \RR^{d\times
d}.\] Then $s\in i\mathcal{B}$ regarded as a column vector satisfies
$Ms\in i\Lambda(y)$, where $y\in[-2X,2X]$ is as in \eqref{ydef} and
$\Lambda(y):= \{yd\} \times (2\pi \ZZ)^{d-1}$ is a lattice in an
affine hyperplane of $\RR^d$. Note that $M$ is non-singular and
depends only the field $K$, hence it suffices to show that
\[\int_{-2X}^{2X}\#(\Lambda(y)\cap M\mathcal{B})\,dy\ll_K X^d.\]
The integrand is $\ll_K X^{d-1}$, hence the required bound follows.
\end{proof}

\subsection{Explicit Fourier expansion of Eisenstein series}

The aim of this section is to verify the relations
\eqref{fouriergeneralEisen2}--\eqref{newWphibound2}. In particular,
we need to define $W_\phi:K_\infty\to\CC$ for each $\phi \in
R^{(\mt)} V_{\chi,\chi^{-1}}(\mc_\chi^2)$ and identify the
coefficients $\lambda^{(\mt)}_{\chi,\chi^{-1}}(\mathfrak{m})$ for
nonzero ideals $\mathfrak{m}\subseteq\mo$. By linearity and
orthogonality, we can assume that $\phi$ is one of the pure tensors
$\phi^{(\mt,q)}:=E(\varphi^{(\mt,q)})\in R^{(\mt)}
V_{\chi,\chi^{-1}}(\mc_\chi^2)$ introduced after the proof of
Lemma~\ref{eisenstein}. The superscript indicates a nonzero ideal
$\mt\subseteq\mo$ and a vector $q \in (2\ZZ)^d$: we shall keep it
fixed and drop it from the notation for simplicity.

First of all we observe that \eqref{localphinorm} implies
\begin{equation}\label{globalphinorm}(\N\mt)^{-\eps}
\ll_{K,\eps}\|\varphi\|\leq 1\end{equation} for any $\eps>0$. We
insert the definition \eqref{eisensteindef} into the Fourier
decomposition\footnote{Strictly speaking, we should extend $\varphi$
to a section $\varphi(s)\in H(s)$, perform the calculation for $\Re
s>1/2$, and deduce the result for $\Re s\geq 0$ by meromorphic
continuation. We also note that $\vol(K\backslash\mathbb{A})=1$ by
our choice of Haar measures.}
\begin{displaymath}
 E\left(\varphi, \left(\begin{matrix} y & x\\ 0 & 1\end{matrix}\right)\right) =
 \sum_{r \in K} \int_{K \backslash \mathbb{A}} E\left(\varphi, \left(\begin{matrix} y & \xi \\ 0 & 1\end{matrix}\right)\right)\psi(-r\xi)\,d\xi \,
 \psi(rx),
 \end{displaymath}
and use the fact that by the Bruhat decomposition a complete set of
representatives of $P(K)\backslash \GL_2(K)$ is given by
$\left(\begin{smallmatrix} 1 & 0\\ 0 & 1\end{smallmatrix}\right)$
and the matrices $\left(\begin{smallmatrix} 0 & -1\\ 1 &
\ast\end{smallmatrix}\right)$ in $\GL_2(K)$. We obtain
\begin{displaymath}
 E\left(\varphi, \left(\begin{matrix} y & x\\ 0 & 1\end{matrix}\right)\right) =
 \varphi\left(\begin{pmatrix} y & 0\\ 0 & 1\end{pmatrix}\right) + \sum_{r \in K} \int_{\mathbb{A}} \varphi\left(\left(\begin{matrix} 0 & -1\\ y & \xi\end{matrix}\right)\right)\psi(-r\xi)\,d\xi\, \psi(rx).
\end{displaymath}
On the right hand side $r=0$ contributes to the constant term of
$E(\varphi)$. From now on we shall assume that $r\in K^\times$. We
define $\delta\in\mathbb{A}^\times$ by $\delta_\infty:=(1,\dots,1)$
and $\delta_\mpr:=\varpi_\mpr^{v_\mpr(\mathfrak{d})}$ for each prime
$\mpr$. By the change of variable $\xi\to y\delta^{-1}\xi$ the
integral over $\mathbb{A}$ becomes, using also \eqref{trafoH0} and
$\chi(r)=|r|=1$,
\begin{equation}\label{globalint}\chi^2(\delta)\chi^{-1}(ry)|ry|^{1/2}
\int_{\mathbb{A}}\varphi\left(\left( \begin{matrix} 0 & -\delta^{-1}\\
\delta & \xi\end{matrix} \right)\right)\psi(-ry\delta^{-1}\xi)\,
d\xi.\end{equation} As $\varphi$ is a pure tensor, we can write this
expression as a product of local factors in a natural fashion.

By our choice of the Haar measure on $K_\infty$, the infinite part
of \eqref{globalint} is $|D_K|^{-1/2}$ times the product over $1\leq
j\leq d$ of the following expressions:
\begin{equation}\label{jfactor}\chi_j(\sgn(ry))|r^{\sigma_j}y_j|^{1/2-s_j}\int_{-\infty}^\infty\varphi_j\left(\left(
\begin{matrix} 0 & -1\\ 1 & \xi\end{matrix}
\right)\right)e(-r^{\sigma_j}y_j\xi)\,d\xi,\end{equation} where
$s_j\in i\RR$ denotes the unique exponent such that
$\chi_j=|\cdot|^{s_j}$ on $\RR_{>0}$ (cf.\ notation following
\eqref{eiseneig}). Using the Iwasawa decomposition
\begin{equation}\label{archIwasawa}\begin{pmatrix} 0 & -1\\
1 & \xi\end{pmatrix}=
\begin{pmatrix} \frac{1}{\sqrt{\xi^2+1}} & \frac{-\xi}{\sqrt{\xi^2+1}}\\ 0 & \sqrt{\xi^2+1}\end{pmatrix}
\begin{pmatrix} \frac{\xi}{\sqrt{\xi^2+1}} & \frac{-1}{\sqrt{\xi^2+1}}\\ \frac{1}{\sqrt{\xi^2+1}} &
\frac{\xi}{\sqrt{\xi^2+1}}\end{pmatrix}\end{equation} we see that
\[\varphi_j\left(\begin{pmatrix} 0 & -1\\ 1 & \xi\end{pmatrix}\right)
=\frac{\chi_j^{-1}(\xi^2+1)}{\sqrt{\xi^2+1}}\left(\frac{\xi-i}{\sqrt{\xi^2+1}}\right)^{q_j}
=\frac{1}{(\xi^2+1)^{1/2+s_j}}\left(\frac{\xi-i}{\xi+i}\right)^{q_j/2},\]
where we used that $q_j$ is even. Therefore \eqref{jfactor} equals
\[\chi_j(\sgn(ry))|r^{\sigma_j}y_j|^{1/2-s_j}
\int_{-\infty}^\infty\frac{e(-r^{\sigma_j}y_j\xi)}{(\xi^2+1)^{1/2+s_j}}
\left(\frac{\xi-i}{\xi+i}\right)^{q_j/2}\,d\xi.\]
The integral on
the right hand side remains unchanged when $r^{\sigma_j}y_j$ and
$q_j$ are replaced by $|r^{\sigma_j}y_j|$ and
$\sgn(r^{\sigma_j}y_j)q_j$, hence by \cite[(2.16)]{BrMo} we deduce
(after the change of variable $\xi\to-\xi$) that the previous
display equals
\[\chi_j(\sgn(ry))(-1)^{\frac{q_j}{2}}\pi^{\frac{1}{2}+s_j}
\frac{W_{\sgn(r^{\sigma_j}y_j)\frac{q_j}{2},
s_j}(4\pi|r^{\sigma_j}y_j|)}{\Gamma(\frac{1}{2}+s_j+\sgn(r^{\sigma_j}y_j)\frac{q_j}{2})}.\]
Now we can combine \eqref{eiseneig}, \eqref{normwhit2},
\eqref{convent} to conclude that the infinite part of
\eqref{globalint} can be written as
\begin{equation}\label{FourEis0}
\eta_\infty \chi_\infty(\sgn(ry_\infty))\tilde{W}_{q/2, \nu_{\chi,
\chi^{-1}}}(r y_{\infty}),
\end{equation}
where $\eta_\infty = \eta_\infty(q,\chi_\infty)\in\CC$ is a constant
of modulus $\pi^{d/2}|D_K|^{-1/2}$.

We now calculate the local factor of \eqref{globalint} corresponding
to a prime $\mpr$. For simplicity we shall omit the subscripts in
$\varpi_\mpr$ and $\delta_\mpr$. The calculation is based on the
following Iwasawa decomposition for $\xi\neq 0$:
\begin{equation}\label{iwasawa}
 \left( \begin{matrix} 0 & -\delta^{-1}\\ \delta & \xi\end{matrix} \right) =
 \begin{cases}
 \left(\begin{matrix}1 & 0 \\ 0 & 1\end{matrix} \right)\left(\begin{matrix}0 & -\delta^{-1} \\ \delta & \xi\end{matrix} \right), & v_\mpr(\xi)
 \geq 0,\\[14pt]
 \left(\begin{matrix}\xi^{-1} & -\delta^{-1} \\ 0 & \xi\end{matrix} \right) \left(\begin{matrix}1 & 0 \\ \delta\xi^{-1} & 1\end{matrix} \right), & v_\mpr(\xi) < 0.
 \end{cases}
\end{equation}
We write $m:=v_{\mpr}(\mc_{\chi})\geq 0$, $n:=v_{\mpr }(ry)\geq 0$,
and we recall that $\varphi_\mpr=\varphi_{\mpr,v_\mpr(\mt)}$ is as
in Lemma~\ref{eisenstein}.

We first consider the case when $m=0$, i.e.\ $\chi_\mpr$ is
unramified. We assume that $v_\mpr(\mt)=0$, then $\varphi_\mpr$ is
constant $1$ on $\KK(\mo_\mpr)$, and by \eqref{iwasawa} we have
\begin{align*}
\int_{K_\mpr}\varphi_\mpr\left(\left( \begin{matrix} 0 & -\delta^{-1}\\
\delta & \xi\end{matrix} \right)\right)&\psi(-ry\delta^{-1}\xi)\,
d\xi \\
&=\int_{v_\mpr(\xi)\geq 0}\psi(-ry\delta^{-1}\xi)\,
d\xi+\int_{v_\mpr(\xi)<0}\chi_\mpr^{-2}(\xi)|\xi|^{-1}\,\psi(-ry\delta^{-1}\xi)\,d\xi\\
&=1+\sum_{j=1}^\infty\chi_\mpr(\varpi^{2j})(\N\mpr)^{-j}\int_{v_\mpr(\xi)=-j}\psi(-ry\delta^{-1}\xi)\,d\xi.
\end{align*}
We calculate the inner integral by observing
\begin{equation}\label{layer}\int_{v_\mpr(\xi)= -j}\psi(-ry\delta^{-1}\xi)\,d\xi=
\begin{cases}(\N\mpr)^j\left(1-\frac{1}{\N\mpr}\right),&1\leq j\leq n,\\
-(\N\mpr)^n,&j=n+1,\\
0,& j\geq n+2.
\end{cases}\end{equation} We obtain
\begin{align*}
\int_{K_\mpr}\varphi_\mpr\left(\left( \begin{matrix} 0 & -\delta^{-1}\\
\delta & \xi\end{matrix} \right)\right)\psi(-ry\delta^{-1}\xi)\,
d\xi
&= 1 -\frac{\chi_\mpr(\varpi^{2n+2})}{\N\mpr}+\left(1-\frac{1}{\N\mpr}\right)\sum_{j=1}^n\chi_\mpr(\varpi^{2j})\\
&=\left(1-\frac{\chi_\mpr(\varpi^2)}{\N\mpr}\right)\sum_{j=0}^n\chi_\mpr(\varpi^{2j}).
\end{align*}
We proved that for $\chi_\mpr$ unramified and $v_\mpr(\mt)=0$ the
local factor of \eqref{globalint} corresponding to $\mpr$ equals
\begin{equation}\label{eisencoeff}
|ry|_{\mpr}^{1/2} \chi_\mpr^2(\delta)\left(1 -
\frac{\chi_\mpr^2(\varpi)}{\N\mpr}\right) \sum_{j=0}^n
\chi_\mpr(\varpi^{2j-n}).
\end{equation}
For $v_\mpr(\mt)=1$ a similar calculation based on
\eqref{functionphi4} and \eqref{iwasawa} shows that
\[\int_{K_\mpr}\varphi_\mpr\left(\left( \begin{matrix} 0 & -\delta^{-1}\\
\delta & \xi\end{matrix} \right)\right)\psi(-ry\delta^{-1}\xi)\,
d\xi
=\frac{1+\chi_\mpr(\varpi^{2n+2})}{(\N\mpr)^{1/2}}-(\N\mpr)^{1/2}\left(1-\frac{1}{\N\mpr}\right)
\sum_{j=1}^n\chi_\mpr(\varpi^{2j}),\] where we understand any empty
sum as zero. If $\chi_\mpr^2(\varpi)\neq -1$, then we conclude that
the local factor of \eqref{globalint} corresponding to $\mpr$ has
absolute value equal to $|1+\chi_\mpr^2(\varpi)|(\N\mpr)^{-1/2}$ for
$n=0$ and not exceeding $(n+1)(\N\mpr)^{(1-n)/2}$ in general. If
$\chi_\mpr^2(\varpi)=-1$, then we conclude that the local factor of
\eqref{globalint} corresponding to $\mpr$ equals
\begin{equation}\label{eisencoeffstriking}
|ry|_{\mpr}^{1/2}\chi_\mpr^2(\delta)\chi_\mpr^{-1}(\varpi)(\N\mpr)^{1/2}\left(1-\frac{\chi_\mpr^2(\varpi)}{\N\mpr}\right)\sum_{j=0}^{n-1}
\chi_\mpr(\varpi^{2j-n+1}),\end{equation} an expression very similar
to \eqref{eisencoeff}. For $v_\mpr(\mt)\geq 2$ a similar calculation
based on \eqref{functionphi3} and \eqref{iwasawa} shows that
\begin{align*}\int_{K_\mpr}\varphi_\mpr\left(\left( \begin{matrix} 0 & -\delta^{-1}\\
\delta & \xi\end{matrix} \right)\right)&\psi(-ry\delta^{-1}\xi)\,
d\xi=\\
&-\chi_\mpr(\varpi^{2v_\mpr(\mt)-2})(\N\mpr)^{-v_\mpr(\mt)/2}\int_{v_\mpr(\xi)=
1-v_\mpr(\mt)}\psi(-ry\delta^{-1}\xi)\,d\xi\\
&+\left(1-\frac{1}{\N\mpr}\right)\sum_{j=v_\mpr(\mt)}^\infty\chi_\mpr(\varpi^{2j})(\N\mpr)^{v_\mpr(\mt)/2-j}
\int_{v_\mpr(\xi)= -j}\psi(-ry\delta^{-1}\xi)\,d\xi.\end{align*}
Using \eqref{layer}, we conclude that the local factor of
\eqref{globalint} corresponding to $\mpr$ vanishes for $n\leq
v_\mpr(\mt)-3$, has absolute value equal to $(\N\mpr)^{-1}$ for
$n=v_\mpr(\mt)-2$ and not exceeding
$(n-v_\mpr(\mt)+3)(\N\mpr)^{(v_\mpr(\mt)-n)/2}$ for $n\geq
v_\mpr(\mt)-1$.

We turn to the case when $m>0$, i.e.\ $\chi_\mpr$ is ramified. We
combine \eqref{functionphi2} with \eqref{iwasawa} to see that the
local factor of \eqref{globalint} corresponding to $\mpr$ equals
\[\chi_\mpr^2(\delta)\chi_\mpr^{-1}(ry)|ry|_{\mpr}^{1/2}(\N\mpr)^{(m+v_\mpr(\mt))/2}\int_{v_\mpr(\xi)=-m-v_\mpr(\mt)}
\chi_\mpr^{-2}(\xi)|\xi|^{-1}\chi_\mpr(\xi)\,\psi(-ry\delta^{-1}\xi)\,
d\xi.\] By the change of variable $\xi\to (ry)^{-1}\xi$ this is the
same as
\[\chi_\mpr^2(\delta)(\N\mpr)^{(n-m-v_\mpr(\mt))/2}\int_{v_\mpr(\xi)=n-m-v_\mpr(\mt)}\chi_\mpr^{-1}(\xi)\psi(-\delta^{-1}\xi)\,d\xi.\]
For $n=v_\mpr(\mt)$ the integral is a Gau{\ss} sum of absolute value
$(\N\mpr)^{m/2}$. For $n<v_\mpr(\mt)$ we pick
$z\in\mpr^{-1}\mo_\mpr^\times$ such that $\psi(\delta^{-1}z)\neq 1$.
Changing $\xi$ to $\xi+z=\xi(1+\xi^{-1}z)$ does not affect the value
$\chi_\mpr^{-1}(\xi)$ and therefore introduces an additional factor
$\psi(-\delta^{-1}z)\neq 1$, hence the integral vanishes. For
$n>v_\mpr(\mt)$ we pick $z \in 1 + \mpr^{m-1}\mo_\mpr^\times$ if $m>
1$ and $z\in \mo_{\mpr}^{\times}$ if $m=1$ such that $\chi_\mpr(z)
\neq 1$. Changing $\xi$ to $\xi z=\xi-\xi(1-z)$ does not affect the
value $\psi(-\delta^{-1}\xi)$ and therefore introduces an additional
factor $\chi_\mpr^{-1}(z) \neq 1$, hence again the integral
vanishes. We proved that for $\chi_\mpr$ ramified the local factor
of \eqref{globalint} corresponding to $\mpr$ equals
\begin{equation}\label{eisencoeff2}
\begin{cases}\eta_{\mpr},& v_{\mpr}(ry)=v_\mpr(\mt),\\
 0,&v_{\mpr}(ry)\neq v_\mpr(\mt),
\end{cases}
\end{equation}
where $\eta_{\mpr}=\eta_{\mpr}(\chi_\mpr,\varpi_\mpr)\in\CC$ is a
constant of modulus $1$.

By the above discussion (in particular by \eqref{globalint},
\eqref{FourEis0}, \eqref{eisencoeff}, \eqref{eisencoeffstriking},
\eqref{eisencoeff2}), we can see that the Fourier coefficients
$\rho_{E(\varphi)}(\mathfrak{m})$ in \eqref{fourierwhittaker} are
supported on ideals $\mathfrak{m}$ divisible by
\begin{equation}\label{tchi}
\mt_\chi:=\prod_{\substack{\mpr\mid\mt,\,\mpr\nmid\mc_\chi\\v_\mpr(\mt)=1\\\chi_\mpr^2(\varpi_\mpr)=-1}}\mpr
\prod_{\substack{\mpr\mid\mt,\,\mpr\nmid\mc_\chi\\v_\mpr(\mt)\geq
3}}\mpr^{v_\mpr(\mt)-2}\prod_{\substack{\mpr\mid\mt,\,\mpr\mid\mc_\chi}}\mpr^{v_\mpr(\mt)},
\end{equation}
and
\begin{equation}\label{sizerho0}
 |\rho_{E(\varphi)}(t_\chi)| = \frac{\pi^{d/2}|D_K|^{-1/2}}
 {|L^{(\mt\mt_\chi^{-1})}(1, \chi^2)|(\N\mt\mt_\chi^{-1})^{1/2}F_{\chi,\mt}},
\end{equation}
where $L^{(\mt\mt_\chi^{-1})}(\cdot,\chi^2)$ denotes a partial Hecke
$L$-function\footnote{Note that $L(s,\chi^2)$ is holomorphic and
nonzero at $s=1$, because $\chi^2$ is a nontrivial Hecke
character.}, and
\begin{equation}\label{fchit}
F_{\chi,\mt}:=\prod_{\substack{\mpr\mid\mt,\,\mpr\nmid\mc_\chi\\v_\mpr(\mt)=1\\\chi_\mpr^2(\varpi_\mpr)\neq
-1}} |1+\chi_\mpr^2(\varpi_\mpr)|^{-1}.
\end{equation}
For convenience we note that
\[\mt\mt_\chi^{-1}=\prod_{\substack{\mpr\mid\mt,\,\mpr\nmid\mc_\chi\\v_\mpr(\mt)=1\\\chi_\mpr^2(\varpi_\mpr)\neq-1}}\mpr
\prod_{\substack{\mpr\mid\mt,\,\mpr\nmid\mc_\chi\\v_\mpr(\mt)\geq
2}}\mpr^2,\] so that for $\mpr\nmid\mc_\chi$ the relation
$\mpr\nmid\mt\mt_\chi^{-1}$ is equivalent to $\mpr\mid\mt$,
$v_\mpr(\mt)=1$, $\chi_\mpr^2(\varpi_\mpr)=-1$. The coefficients
$\rho_{E(\varphi)}(\mathfrak{m})$ enjoy the property
\begin{equation}\label{lambdachitm}
 \rho_{E(\varphi)}(\mathfrak{m}\mt_\chi) =\frac{\lambda_{\chi,\mt}(\mathfrak{m})}{\sqrt{\N\mathfrak{m}}}
 \rho_{E(\varphi)}(\mt_\chi),
 \qquad\mathfrak{m}\subseteq\mo,
\end{equation}
where $\lambda_{\chi,\mt}$ is a multiplicative function on nonzero
integral ideals satisfying the identity
\begin{equation*}
\lambda_{\chi,\mt}(\mathfrak{m})=
\begin{cases}
\sum_{\ma \mb =\mathfrak{m}}\chi(\ma \mb^{-1}),
&\gcd(\mathfrak{m},\mt\mt_\chi^{-1}\mc_{\chi})=\mo,\\
0,&\gcd(\mathfrak{m},\mc_\chi)\neq\mo,
\end{cases}\end{equation*}
and the general bound
\begin{align}\nonumber
|\lambda_{\chi,\mt}(\mathfrak{m})| &\leq\tau(\mathfrak{m})
\prod_{\substack{\mpr\mid\mt,\,\mpr\mid\mathfrak{m},\,\mpr\nmid\mc_\chi\\v_\mpr(\mt)=1\\\chi_\mpr^2(\varpi_\mpr)\neq-1}}\frac{\N\mpr}{|1+\chi_\mpr(\varpi_\mpr)|}
\prod_{\substack{\mpr\mid\mt,\,\mpr\mid\mathfrak{m},\,\mpr\nmid\mc_\chi\\v_\mpr(\mt)\geq
2}}(\N\mpr)^2\\[4pt]
\label{boundFourierEisen2} &\leq
F_{\chi,\mt}\tau(\mt)(\N(\mt\mt_\chi^{-1}))^{1/2}(\N\gcd(\mt\mt_\chi^{-1},\mathfrak{m}))\tau(\mathfrak{m}).
\end{align}

We are now ready to write down explicitly the Fourier expansion
\eqref{fourierwhittaker} for our specific $\phi=E(\varphi)\in
R^{(\mt)} V_{\chi,\chi^{-1}}(\mc_\chi^2)$ introduced at the
beginning of this section. We have
\begin{equation}\label{eisensteinconstant}\rho_{E(\varphi),0}(y)=\varphi\left(\begin{pmatrix} y & 0\\ 0 &
1\end{pmatrix}\right)+\int_{\mathbb{A}}
\varphi\left(\left(\begin{matrix} 0 & -1\\ y &
\xi\end{matrix}\right)\right)\,d\xi,\qquad y\in
\mathbb{A}^{\times},\end{equation} and
\begin{equation}\label{contsignature}\eps_{\chi,\chi^{-1}}(\sgn(y))=\chi_\infty(\sgn(y)),\qquad y\in
K_\infty^\times.\end{equation} With the notation \eqref{tchi},
\eqref{fchit} and \eqref{lambdachitm} we define
\begin{equation}\label{generalFourEis}
\lambda^{(\mt)}_{\chi,\chi^{-1}}(\mathfrak{m}):=
\begin{cases}
F_{\chi,\mt}^{-1}\tau(\mt)^{-1}(\N\mt)^{-1/2}(\N\mt_\chi)\lambda_{\chi,\mt}(\mathfrak{m}t_\chi^{-1}),&\mt_\chi\mid\mathfrak{m},\\
0,&\text{otherwise},
\end{cases}\end{equation}
and
\begin{equation}\label{fancyWhittaker}
W_{E(\varphi)}(y):=F_{\chi,\mt}\tau(\mt)(\N(\mt\mt_\chi^{-1}))^{1/2}
\rho_{E(\varphi)}(t_\chi)\,
\eps_{\chi,\chi^{-1}}(\sgn(y))\tilde{W}_{q/2, \nu_{\chi,
\chi^{-1}}}(y), \qquad y\in K_\infty^\times.\end{equation} Then
\eqref{fouriergeneralEisen2} follows from \eqref{fourierwhittaker},
\eqref{lambdachitm}, \eqref{generalFourEis}, \eqref{fancyWhittaker};
\eqref{boundFourierEisen3} follows in slightly sharper form from
\eqref{boundFourierEisen2}, \eqref{generalFourEis};
\eqref{newWphibound2} follows from \eqref{eisensteininnerproduct},
\eqref{analyticconductor}, \eqref{orth}, \eqref{globalphinorm},
\eqref{sizerho0}, \eqref{fancyWhittaker}.

It is worthwhile to review the special case when $\mc=\mc_\chi^2$.
Then $\mt=\mt_\chi=\mo$ and $\phi=E(\varphi)$ spans the space
$V_{\chi,\chi^{-1},q}(\mc_\chi^2)$ of newforms of pure weight $q$.
We can define $W_\phi$ intrinsically by \eqref{Kir}, and the
coefficients $\lambda^{(\mt)}_{\chi,\chi^{-1}}$ in this case
specialize to the Hecke eigenvalues given for
$\mathfrak{m}\subseteq\mo$ by
\begin{equation*}
\lambda_{\chi, \chi^{-1}}(\mathfrak{m})=
\begin{cases}
\sum_{\ma \mb =\mathfrak{m}}\chi(\ma \mb^{-1}),&\gcd(\mathfrak{m},\mc_\chi)=\mo,\\
0,&\text{otherwise}.
\end{cases}\end{equation*}
By \eqref{orth} and the above discussion, newforms $\phi_1,\phi_2\in
V_{\chi,\chi^{-1}}(\mc_\chi^2)$ satisfy the analogue of
\eqref{constant},
\begin{equation*}
\langle \phi_1,\phi_2\rangle=C_{\chi,\chi^{-1}}\langle
W_{\phi_1},W_{\phi_2}\rangle\end{equation*} with some positive
constant $C_{\chi,\chi^{-1}}\gg_{K,\eps} C(\chi, \chi^{-1})^{-\eps}$
depending only on $\chi$.

\subsection{The constant term of a certain Eisenstein series}

For a Rankin--Selberg type computation in the next section we need
to understand in more detail the constant term of a certain
Eisenstein series. For $s\in \CC$ let $\chi_1$, $\chi_2$ be the
quasicharacters defined by $\chi_1(y) := |y|^{s}$, $\chi_2(y):=
|y|^{-s}$ for $y \in \mathbb{A}^{\times}$. For a nonzero ideal
$\mc\subseteq\mo$ let us define $\varphi = \varphi(s,g) \in
H(\chi_1, \chi_2)$ by
\begin{equation}\label{phidef}
 \varphi\left(s,\begin{pmatrix} a &
x\\0 & b\end{pmatrix}k\right) :=
\begin{cases}
 \left|\frac{a}{b}\right|^{1/2+s}, &k \in \SO_2(K_{\infty}) \times \KK(\mc),\\
 0, & k \in \KK - (\SO_2(K_{\infty}) \times \KK(\mc)).
 \end{cases}
\end{equation}
The constant term of $E(\varphi(s),g)$ equals (cf.\
\eqref{eisensteinconstant} and \cite[(5.3)]{GJ})
\begin{equation}\label{constint}E_\text{const}(\varphi(s),g):=\varphi(s,g)+\int_{\mathbb{A}}
\varphi\left(s,\left(\begin{matrix} 0 & -1\\ 1 &
\xi\end{matrix}\right)g\right)\,d\xi,\qquad
g\in\GL_2(\mathbb{A}).\end{equation}
The aim of this section is to
prove that
\begin{equation}\label{aim}
\int_{\mathbb{A}} \varphi\left(s,\left(\begin{matrix} 0 & -1\\ 1 &
\xi\end{matrix}\right)g\right)\,d\xi=\frac{\Lambda_K(2s)}{\Lambda_K(1+2s)}H(s,g),
\end{equation}
where
\[\Lambda_K(s):=|D_K|^{s/2}\left(\pi^{-s/2}\Gamma(s/2)\right)^d\prod_\mpr\left(1-(\N\mpr)^{-s}\right)^{-1},\qquad\Re s>1,\] is the completed Dedekind zeta-function, and $H(s,g)$ is
a meromorphic function whose zeros and poles lie on $\Re s=0$ and
$\Re s=-1/2$, respectively, and which is constant at $s=1/2$:
\[H(1/2,g)=|\delta|(\N\mc)^{-1}\prod_{\mpr\mid\mc}(1+(\N\mpr)^{-1})^{-1}
=|D_K|^{-1}[\KK(\mo) :\KK(\mc)]^{-1}.\] This is a general feature
(cf.\ \cite[Prop.~3.7.5]{Bu} and \cite[p.~277]{GJ}), but we
preferred to prove it by direct calculation.

In order to understand the integral in \eqref{constint}, we define
$\delta\in\mathbb{A}^\times$ as in the previous section, and then
using the Iwasawa and Bruhat decompositions we write
\[g=\begin{pmatrix}a&x\\0&b\end{pmatrix}h,\qquad
x \in \mathbb{A},\ \ a, b \in \mathbb{A}^{\times},\ \
h\in\GL_2(\mathbb{A}),\] where $h_\infty\in\SO_2(K_\infty)$,
$h_\mpr\in\KK(\mo_\mpr)$ for $\mpr\nmid\mc$, and
$h_\mpr\in\GL_2(K_\mpr)$ is either $\left(\begin{smallmatrix}1
&0\\0&1\end{smallmatrix}\right)$ or of the form
$\left(\begin{smallmatrix}0 &-\delta_\mpr^{-1}\\\delta_\mpr&
\eta_\mpr\end{smallmatrix}\right)$ for $\mpr\mid\mc$. After simple
manipulations the integral in \eqref{constint} becomes
\[\left|\frac{a}{b}\right|^{1/2-s}|\delta|^{2s}\int_{\mathbb{A}} \varphi\left(s,\left(\begin{matrix} 0 &
-\delta^{-1}\\ \delta & \xi\end{matrix}\right)h\right)\,d\xi.\] The
new integral decomposes as a product of local factors in a natural
fashion, using that $\varphi$ is the tensor product of its
restrictions $\varphi_\infty$ and $\varphi_\mpr$ to the various
quasifactors of $\GL_2(\mathbb{A})$.

The infinite part of the new integral equals, by our choice of the
Haar measure on $K_\infty$, the right $\SO_2(K_\infty)$-invariance
of $\varphi_\infty(s,\cdot)$, the Iwasawa decomposition
\eqref{archIwasawa}, and the formula \cite[3.251.2]{GR},
\[|D_K|^{-1/2}\left(\int_{-\infty}^\infty\frac{d\xi}{(\xi^2+1)^{1/2+s}}\right)^d
=|D_K|^{-1/2}\left(\frac{\Gamma(1/2)\Gamma(s)}{\Gamma(1/2+s)}\right)^d.\]

We now calculate the local factor corresponding to a prime $\mpr$.
For simplicity we shall omit the subscripts in $\delta_\mpr$ and
$\eta_\mpr$. If $\mpr\nmid\mc$, then $\varphi_\mpr(s,\cdot)$ is
right $\KK(\mo_\mpr)$-invariant, hence by the Iwasawa decomposition
\eqref{iwasawa} the local factor corresponding to $\mpr$ equals
\begin{align*}
\int_{K_\mpr} \varphi_\mpr\left(s,\left(\begin{matrix} 0 & -\delta^{-1}\\
\delta & \xi\end{matrix}\right)\right)\,d\xi&=\int_{v_\mpr(\xi)\geq
0}1\,d\xi+\int_{v_\mpr(\xi)<0}|\xi|^{-1-2s}\,d\xi\\
&=1+\sum_{j=1}^\infty(\N\mpr)^{-j(1+2s)}(\N\mpr)^j\left(1-\frac{1}{\N\mpr}\right)\\
&=\frac{1-(\N\mpr)^{-1-2s}}{1-(\N\mpr)^{-2s}}.\end{align*} If
$\mpr\mid\mc$ then depending on the shape of $h_\mpr$ the local
factor is either
\begin{equation}\label{int1}\int_{K_\mpr}
\varphi_\mpr\left(s,\left(\begin{matrix} 0 & -\delta^{-1}\\ \delta &
\xi\end{matrix}\right)\right)\,d\xi\end{equation} or
\begin{equation}\label{int2}\int_{K_\mpr} \varphi_\mpr\left(s,\left(\begin{matrix} 0 &
-\delta^{-1}\\ \delta & \xi\end{matrix}\right)\begin{pmatrix} 0 &
-\delta^{-1}\\ \delta &
\eta\end{pmatrix}\right)\,d\xi.\end{equation} The integral
\eqref{int1} equals, by the Iwasawa decomposition \eqref{iwasawa},
\begin{align*}
\int_{K_\mpr} \varphi_\mpr\left(s,\left(\begin{matrix} 0 & -\delta^{-1}\\
\delta & \xi\end{matrix}\right)\right)\,d\xi &=\int_{v_\mpr(\xi)\leq
-v_\mpr(\mc)}|\xi|^{-1-2s}\,d\xi\\
&=\sum_{j=v_\mpr(\mc)}^\infty(\N\mpr)^{-j(1+2s)}(\N\mpr)^j\left(1-\frac{1}{\N\mpr}\right)\\
&=(\N\mpr)^{-(2s)
v_\mpr(\mc)}\frac{1-(\N\mpr)^{-1}}{1-(\N\mpr)^{-2s}}.
\end{align*}
If $v_\mpr(\eta)\leq -v_\mpr(\mc)$, then by the right
$\KK(\mc_\mpr)$ invariance of $\varphi_\mpr(s,\cdot)$ the integral
\eqref{int2} equals
\begin{align*}
\int_{K_\mpr} \varphi_\mpr\left(s,\left(\begin{matrix} 0 & -\delta^{-1}\\
\delta & \xi\end{matrix}\right)\begin{pmatrix} 0 & -\delta^{-1}\\
\delta & \eta\end{pmatrix}\right)\,d\xi &=
\int_{K_\mpr} \varphi_\mpr\left(s,\left(\begin{matrix} 0 & -\delta^{-1}\eta\\
\delta\eta^{-1} & \eta\xi-1\end{matrix}\right)\begin{pmatrix} 1 & 0\\
\delta\eta^{-1} & 1\end{pmatrix}\right)\,d\xi\\
&=|\eta|^{1+2s}\int_{K_\mpr} \varphi_\mpr\left(s,\left(\begin{matrix} 0 & -\delta^{-1}\\
\delta & \eta^2\xi-\eta\end{matrix}\right)\right)\,d\xi\\
&=|\eta|^{2s-1}(\N\mpr)^{-(2s)
v_\mpr(\mc)}\frac{1-(\N\mpr)^{-1}}{1-(\N\mpr)^{-2s}},
\end{align*}
where in the last step we combined the change of variable
$\xi\to\eta^{-2}(\xi+\eta)$ with our previous result for
\eqref{int1}. If $v_\mpr(\eta)> -v_\mpr(\mc)$, then \eqref{int2}
equals
\begin{align*}
\int_{K_\mpr} \varphi_\mpr\left(s,\left(\begin{matrix} 0 & -\delta^{-1}\\
\delta & \xi\end{matrix}\right)\begin{pmatrix} 0 & -\delta^{-1}\\
\delta & \eta\end{pmatrix}\right)\,d\xi &=
\int_{K_\mpr} \varphi_\mpr\left(s,\begin{pmatrix} -1 & -\delta^{-1}\eta\\
\delta\xi & \eta\xi-1\end{pmatrix}\right)\,d\xi\\
&=\int_{K_\mpr} \varphi_\mpr\left(s,\begin{pmatrix} 0 & -\delta^{-1}\xi^{-1}\\
\delta\xi & \eta\xi-1\end{pmatrix}\right)\,d\xi\\
&=\int_{K_\mpr} |\xi|^{-1-2s}\varphi_\mpr\left(s,\begin{pmatrix} 0 & -\delta^{-1}\\
\delta & \eta-\xi^{-1}\end{pmatrix}\right)\,d\xi.
\end{align*}
Using the change of variable $\xi\to(\eta-\xi)^{-1}$ and the Iwasawa
decomposition \eqref{iwasawa} we obtain
\begin{align*}
\int_{K_\mpr} \varphi_\mpr\left(s,\left(\begin{matrix} 0 & -\delta^{-1}\\
\delta & \xi\end{matrix}\right)\begin{pmatrix} 0 & -\delta^{-1}\\
\delta & \eta\end{pmatrix}\right)\,d\xi
&=\int_{K_\mpr} |\eta-\xi|^{2s-1}\varphi_\mpr\left(s,\begin{pmatrix} 0 & -\delta^{-1}\\
\delta & \xi\end{pmatrix}\right)\,d\xi\\
&=\int_{v_\mpr(\xi)\leq -v_\mpr(\mc)}
|\eta-\xi|^{2s-1}|\xi|^{-1-2s}\,d\xi.
\end{align*}
By $v_\mpr(\xi)\leq -v_\mpr(\mc)<v_\mpr(\eta)$ we have
$|\eta-\xi|=|\xi|$, hence by the change of variable $\xi\to\xi^{-1}$
\[\int_{K_\mpr} \varphi_\mpr\left(s,\left(\begin{matrix} 0 & -\delta^{-1}\\
\delta & \xi\end{matrix}\right)\begin{pmatrix} 0 & -\delta^{-1}\\
\delta & \eta\end{pmatrix}\right)\,d\xi =\int_{v_\mpr(\xi)\leq
-v_\mpr(\mc)}\frac{d\xi}{|\xi|^{2}} =\int_{v_\mpr(\xi)\geq
v_\mpr(\mc)}1\,d\xi=(\N\mpr)^{-v_\mpr(\mc)}.\] This equation is also
true for $\eta=0$, because then $|\eta-\xi|=|\xi|$ holds trivially.
Collecting the previous computations, we arrive at \eqref{aim}.

\subsection{L-functions on $\GL_2$ and $\GL_2 \times \GL_2$}\label{Lfunctionsection}

Let $(\pi, V_{\pi})$ be an irreducible cuspidal representation. Let
$\chi$ be a Hecke character of conductor $\mq$. The twisted
$L$-function
\[L(s,\pi\otimes \chi) = \sum_{ \mathfrak{m} }
\lambda_{\pi\otimes
\chi}(\mathfrak{m})(N\mathfrak{m})^{-s},\qquad\Re s > 1,\]
can be
continued to an entire function on $\CC$ and satisfies a functional
equation relating $s$ to $1-s$ with analytic conductor (see
\eqref{analyticconductor})
\begin{equation*}
 C(\pi \otimes \chi \otimes |\det|^{s-1/2}) \asymp_{\pi,\chi_\infty}
 (\N\mq)^2(1+|\Im s|)^2.
\end{equation*}
The Hecke eigenvalues $\lambda_{\pi\otimes \chi}(\mathfrak{m})$
satisfy the bound \eqref{boundlambda} and the identity
\[\lambda_{\pi\otimes\chi}(\mathfrak{m}) =
\lambda_{\pi}(\mathfrak{m})\chi(\mathfrak{m}),\qquad\gcd(\mathfrak{m},\mq\mc_{\pi})=\mo.\]
By \cite[Theorem 2.1]{Ha} we can express $L(1/2, \pi \otimes \chi)$
as an essentially finite series: there is a complex number $\eta$ of
modulus $1$ and a smooth function $V : (0, \infty) \to \CC$ with
rapidly decaying derivatives, both depending only on the archimedean
parameters of $\pi\otimes\chi$, such that
\[L(1/2, \pi\otimes \chi)=\Sigma+\eta\bar\Sigma,\qquad
\Sigma:=\sum_{\{0 \}\neq \mathfrak{m} \subseteq \mo}
\frac{\lambda_{\pi \otimes \chi}(\mathfrak{m})
}{\sqrt{\N\mathfrak{m}}}
V\left(\frac{\N\mathfrak{m}}{\sqrt{C(\pi\otimes \chi)}}\right).\]
Together with a smooth partition of unity and standard bounds for
 $\lambda_{\pi \otimes \chi}$ at ramified primes we obtain
(cf.\ e.g.\ \cite[Section~5.1]{BHM})
\begin{equation}\label{approx}
 L(1/2, \pi\otimes \chi) \ll_{\pi,\chi_\infty,\eps} (\N\mq)^{\eps} \max_{Y \leq c(\N\mq)^{1+\eps}}
 \Biggl|\sum_{\{0 \}\neq \mathfrak{m} \subseteq \mo}
 \frac{\lambda_{\pi} (\mathfrak{m})\chi(\mathfrak{m})}{\sqrt{\N\mathfrak{m}}}
 V\left(\frac{\N\mathfrak{m}}{Y}\right)\Biggr|,
\end{equation}
where $c=c(\pi,\chi_\infty,\eps)>0$ is a constant and $V:(0, \infty)
\to \CC$ is a smooth function supported on $[\frac{1}{2},2]$ such
that
$V^{(j)}(y)\ll_{\pi,\chi_\infty,j} 1$ for all for all $j \in \NN_0$. \\

The Rankin--Selberg convolution $L(s, \pi \otimes \tilde{\pi})$
(with $\tilde{\pi}$ the contragradient representation) is, up to
finitely many Euler factors at the primes dividing $\mc_\pi$, given
by
\begin{displaymath}
 \zeta_K(2s) \sum_{\{0 \}\neq \mathfrak{m} \subseteq \mo} \frac{|\lambda_{\pi}(\mathfrak{m})|^2}{(\N\mathfrak{m})^{s}}.
\end{displaymath}
It is an entire function of order 1 except for a simple pole at
$s=1$ and satisfies a functional equation relating $s$ to $1-s$, see
e.g.\ \cite[Remark~1.2]{HR}, and the references given there. By
\cite[Lemma~b]{HR}, the analytic conductor satisfies
\begin{equation}\label{rankinselberg}
 C(\pi)^{-B} \ll C(\pi \otimes \tilde{\pi}) \ll C(\pi)^B
\end{equation}
for some constant $B$ depending only on $K$. By standard contour
integration we obtain
\begin{equation}\label{squarebound}
 \sum_{\N \mathfrak{m} \leq x} |\lambda_{\pi}(\mathfrak{m})|^2 \ll C(\pi)^{B'} x
\end{equation}
for $x \geq 1$ and some constant $B'$ depending only on $K$.

Let $\phi_1,\phi_2 \in V_{\pi, q}(\mc_{\pi})$ be newforms of some
weight $q \in \ZZ^d$ and let $\mt_1,\mt_2\subseteq\mo$ be nonzero
ideals. For any integral ideal $\mc$ divisible by $\mt_1\mc_\pi$,
$\mt_2\mc_\pi$ the vectors $\psi_i:=R_{\mt_i}\phi_i$ lie in $V_{\pi,
q}(\mc)$ (cf.\ \eqref{oldforms}), and our aim is to express their
inner product in terms of the inner product of the Whittaker
functions $W_{\phi_{i}}$. For this purpose we shall apply the
Rankin--Selberg unfolding technique to the function
\begin{displaymath}
 F(s) := \int_{\GL_2(K)Z(\mathbb{A})\backslash \GL_2(\mathbb{A})}
 \psi_1(g)\bar{\psi}_2(g)\,E(\varphi(s),g)\,dg,
\end{displaymath}
where $\varphi(s,g)$ is defined by \eqref{phidef}. It is known from
the theory of Eisenstein series, that $E(\varphi(s),g)$ is
meromorphic in $s$ with all the singularities coming from the
constant term $E_\text{const}(\varphi(s),g)$, more precisely from
the integral in \eqref{constint}, see \cite[Section~5]{GJ}. The
result of the previous section shows that $E(\varphi(s), g)$ has a
pole at $s=1/2$ with constant residue
\[\res_{s=\frac{1}{2}}E(\varphi(s),g)=\frac{C_K}{[\KK(\mo):\KK(\mc)]},\qquad
C_K:=\frac{\res_{s=1}\Lambda_K(s)}{2|D_K|\Lambda_K(2)}.\] In
particular,
\begin{equation}\label{ranksel1}
 \res_{s=\frac{1}{2}}F(s)=C_K \frac{\langle R_{\mt_1}\phi_1, R_{\mt_1}\phi_2 \rangle}{[\KK(\mo):\KK(\mc)]} .
\end{equation}
On the other hand, unfolding the integral we see (cf.\
\cite[Prop.~7.47]{KL} or \cite[pp.~372--373]{Bu}) that
\begin{displaymath}
\begin{split}
 F(s) & = \int_{P(K)Z(\mathbb{A})\backslash \GL_2(\mathbb{A})} \psi_1(g)\bar{\psi}_2(g) \,\varphi(s,g) \,dg\\
 & =
 \int_{K^{\times}\backslash\mathbb{A}^{\times}}\int_{K\backslash \mathbb{A}} \int_{\KK} \psi_1\left(\left(\begin{matrix} y & x\\0 & 1\end{matrix}\right)k\right) \bar{\psi}_2\left(\left(\begin{matrix} y & x\\0 & 1\end{matrix}\right)k\right)\varphi\left(s,\left(\begin{matrix} y & x\\0 & 1\end{matrix}\right)k\right) dk\, dx\, \frac{d^{\times}y}{|y|}\\
 & = \int_{K^{\times}\backslash\mathbb{A}^{\times}}\int_{K\backslash \mathbb{A}} \int_{\SO_2(K_{\infty})\times \KK(\mc)} \psi_1\left(\left(\begin{matrix} y & x\\0 & 1\end{matrix}\right)k\right) \bar{\psi}_2\left(\left(\begin{matrix} y & x\\0 & 1\end{matrix}\right)k\right) |y|^{s-\frac{1}{2}} \,dk\, dx\, d^{\times}y .
 \end{split}
\end{displaymath}
Since $\psi_1\bar{\psi}_2$ is right $\SO_2(K_{\infty}) \times
\KK(\mc)$-invariant, we obtain
\[F(s)=\frac{1}{[\KK(\mo) :\KK(\mc)]}\int_{K^{\times}\backslash\mathbb{A}^{\times}}\int_{K\backslash \mathbb{A}}
\psi_1\left(\left(\begin{matrix} y & x\\0 &
1\end{matrix}\right)\right) \bar{\psi}_2\left(\left(\begin{matrix} y
& x\\0 & 1\end{matrix}\right)\right) |y|^{s-\frac{1}{2}} \, dx\,
d^{\times}y .\] We choose any finite ideles
$t_i\in\mathbb{A}_\text{fin}^\times$ representing the ideals
$\mt_i$, so that
$\psi_i(g)=\phi_i\left(g\left(\begin{smallmatrix}t_i^{-1}&0\\0&1\end{smallmatrix}\right)\right)$.
We insert the Fourier expansion \eqref{fouriersimple}, and integrate
over $x$ getting
\begin{displaymath}
\begin{split}
 F(s) & = \frac{(\N\mt_1\mt_2)^{\frac{1}{2}}}{[\KK(\mo):\KK(\mc)]}
 \int_{K^{\times}\backslash\mathbb{A}^{\times}} \sum_{r \in K^{\times}}
 \frac{\lambda_{\pi}(ry_{\text{fin}}t_1^{-1})\bar \lambda_{\pi}(ry_{\text{fin}}t_1^{-1})}{\N(r y_{\text{fin}})}
 W_{\phi_1}(ry_{\infty})\bar{W}_{\phi_2}(ry_{\infty}) |y|^{s-\frac{1}{2}}\,d^{\times}y\\
 & =\frac{(\N\mt_1\mt_2)^{\frac{1}{2}}}{[\KK(\mo):\KK(\mc)]}
 \int_{\mathbb{A}^{\times}} \frac{\lambda_{\pi}(y_{\text{fin}}t_1^{-1})\bar \lambda_{\pi}(y_{\text{fin}}t_1^{-1})}{\N(y_{\text{fin}})}
 W_{\phi_1}(y_{\infty})\bar{W}_{\phi_2}(y_{\infty}) |y|^{s-\frac{1}{2}}\,d^{\times}y,\\
 & =\frac{(\N\mt_1\mt_2)^{\frac{1}{2}}}{[\KK(\mo):\KK(\mc)]}
 \left(\int_{K_\infty^\times} W_{\phi_1}(y)\bar{W}_{\phi_2}(y)
 |y|^{s-\frac{1}{2}}\,d^{\times}y\right)\left(
 \int_{\mathbb{A}_\text{fin}^{\times}}
 \frac{\lambda_{\pi}(yt_1^{-1})\bar\lambda_{\pi}(yt_2^{-1})}{(\N(y))^{\frac{1}{2}+s}}\,d^{\times}y\right).
\end{split}
\end{displaymath}
We choose any $t\in\mathbb{A}_\text{fin}^\times$ representing the
ideal $\gcd(\mt_1,\mt_2)$, and we make the change of variable $y\to
yt_1t_2t^{-1}$ in the second integral. We obtain
\[F(s)=\frac{(\N\gcd(\mt_1,\mt_2))^{\frac{1}{2}+s}}{(\N\mt_1\mt_2)^{s}[\KK(\mo):\KK(\mc)]}
 \left(\int_{K_\infty^\times} W_{\phi_1}(y)\bar{W}_{\phi_2}(y)
 |y|^{s-\frac{1}{2}}\,d^{\times}y\right)\left(\int_{\mathbb{A}_\text{fin}^{\times}}
 \frac{\lambda_{\pi}(yt_2')\bar\lambda_{\pi}(yt_1')}{(\N(y))^{\frac{1}{2}+s}}\,d^{\times}y\right),\]
where the finite ideles $t'_i:=t_it^{-1}$ represent the coprime
integral ideals $\mt_i:=\mt_i\gcd(\mt_1,\mt_2)^{-1}$. We conclude
\[\res_{s=\frac{1}{2}}F(s) = \frac{\langle W_{\phi_1}, W_{\phi_2}\rangle }{(\N\mt_1'\mt_2')^{1/2}[\KK(\mo)
 :\KK(\mc)]}
 \res_{s=1}\sum_{\{0\} \neq \mathfrak{m}\subseteq\mo}
 \frac{\lambda_{\pi}(\mathfrak{m}\mt_2')\bar\lambda_{\pi}(\mathfrak{m}\mt_1')}{(\N\mathfrak{m})^{s}},\]
whence by \eqref{ranksel1} also
\[\langle R_{\mt_1}\phi_1, R_{\mt_1}\phi_2 \rangle=
\frac{\langle
W_{\phi_1},W_{\phi_2}\rangle}{C_K(\N\mt_1'\mt_2')^{1/2}}
 \res_{s=1}\sum_{\{0\} \neq \mathfrak{m}\subseteq\mo}
 \frac{\lambda_{\pi}(\mathfrak{m}\mt_2')\bar\lambda_{\pi}(\mathfrak{m}\mt_1')}{(\N\mathfrak{m})^{s}}.\]

We now draw some useful consequences of this identity. First,
combining the special case $\mt_1=\mt_2=\mo$ with
\eqref{weightdecomp}, \eqref{orth}, and taking into account the
ramified primes, we arrive at \eqref{constant} with some positive
constant $C_\pi$ depending only on $\pi$ which satisfies
\begin{equation}\label{ranksel3}
(\N\mc_{\pi})^{-\eps} \res_{s=1}L(s,\pi \otimes \tilde{\pi})
\ll_{K,\eps} C_\pi \ll_{K,\eps} (\N\mc_{\pi})^{\eps}
\res_{s=1}L(s,\pi \otimes \tilde{\pi}).
\end{equation}
Second, comparing the special case with the general one, we infer
(cf.\ \cite[p.~73]{ILS})
\begin{equation}\label{shiftedinner}
\begin{split}
 \langle R_{\mt_1}\phi_1, R_{\mt_2}\phi_2 \rangle =
 \frac{\langle\phi_1,\phi_2\rangle}{(\N\mt'_1\mt'_2)^{1/2}}
 &\prod_{\mpr^{\nu_1} \| \mt_1'} \left(\sum_{k=0}^{\infty} \frac{\lambda_{\pi}(\mpr^k)\bar{\lambda}_{\pi}(\mpr^{k+\nu_1})}{(\N\mpr )^k}\right) \left(\sum_{k=0}^{\infty} \frac{|\lambda_{\pi}(\mpr^{k})|^2 }{(\N\mpr )^k}\right)^{-1} \\
 &\prod_{\mpr^{\nu_2} \| \mt_2'} \left(\sum_{k=0}^{\infty} \frac{\lambda_{\pi}(\mpr^{k+\nu_2})\bar{\lambda}_{\pi}(\mpr^{k})}{(\N\mpr )^k}\right) \left(\sum_{k=0}^{\infty} \frac{|\lambda_{\pi}(\mpr^{k})|^2 }{(\N\mpr )^k}\right)^{-1}.
\end{split}
\end{equation}
An important feature here is that the ratio of the two inner
products is independent of the weight $q\in\ZZ^d$. This independence
is key to the existence of the operators \eqref{newR}; it can also
be verified directly by using the Maa{\ss} shift operators.
\\

We are now ready to prove the following lemma.

\begin{lemma}\label{lemma2} Let $(\pi, V_{\pi})$ be an irreducible cuspidal representation
of $\GL_2(K)\backslash \GL_2(\mathbb{A})$ with unitary central
character, and let $\phi \in V_{\pi}(\mc_{\pi})$. Then
\begin{displaymath}
C(\pi)^{-\eps} \| \phi \| \ll_{K, \eps} \| W_\phi \| \ll_{K, \eps}
C(\pi)^{\eps} \| \phi \|.
\end{displaymath}
The implied constants are ineffective and depend only on $K$ and
$\eps$.
\end{lemma}

\begin{proof}
By \eqref{constant} and \eqref{ranksel3} it remains to show
\begin{equation*}
C(\pi)^{-\eps}\ll_{K,\eps}\res_{s=1}L(s,\pi \otimes \tilde{\pi})
\ll_{K,\eps}C(\pi)^{\eps}.\end{equation*} For $K=\QQ$ this is known
from the work of Iwaniec~\cite[Theorem~2]{Iw} (upper bound) and
Hoffstein--Lockhart~\cite{HL} (lower bound). The same bounds are
also available in the number field case and can be obtained as
follows: the upper bound follows verbatim as in \cite[pp.~72--73,
especially the comment between (20) and (21)]{Iw} once we have the
multiplicativity relation \eqref{hecke}, and we know that $L(s, \pi
\otimes \tilde{\pi})$ is of order $1$ and holomorphic except for a
simple pole at $s=1$, and satisfies a suitable functional equation
with conductor satisfying \eqref{rankinselberg}. For the lower
bound, \cite[Prop.~1]{HL} together with \eqref{rankinselberg} gives
the desired bound. Note that $L(s,\pi \otimes \tilde{\pi})$ has
nonnegative coefficients (which incidentally holds by
\cite[Lemma~a]{HR} in a much more general context). To verify the
hypothesis ``no Siegel zeros" for the application of this result, we
distinguish between two cases depending on whether $\text{ad}^2\pi$
is cuspidal or not. In the first case, the absence of Siegel zeros
follows from \cite[Theorem~5]{Ba}. In the second case, the
discussion in \cite[p.~180]{HL} shows that $L(s,\text{ad}^2\pi)$
factors into a Dirichlet $L$-function with character associated to
some quadratic extension $K'/K$ and a Hecke $L$-function
$L_{K'}(s,\chi)$, both with conductor bounded by $\ll_K C(\pi)$. For
both factors, we can bound possible Siegel zeros away from $1$ by
the theorem of Siegel--Brauer--Stark (see \cite{Fo} and \cite{St}).
\end{proof}

\subsection{Sobolev norms}

The right action of $\GL_2(K_\infty)$ on $L^2(\GL_2(K) \backslash
\GL_2(\mathbb{A}),\omega)$ induces an action of its Lie algebra
$\mathfrak{gl}(K_\infty)$ on the subspace of differentiable
functions. We recall this action for the Lie subalgebra
$\fg:=\mathfrak{sl}(K_\infty)$ generated by the independent vectors
\begin{displaymath}
 \qquad H_j := \left(\begin{matrix} e_j & 0\\ 0 & -e_j\end{matrix}\right),
 \qquad R_j := \left(\begin{matrix} 0 & e_j \\ 0 & 0 \end{matrix}\right),
 \qquad L_j := \left(\begin{matrix} 0 & 0 \\ e_j & 0 \end{matrix}\right),
 \qquad 1\leq j\leq d,
\end{displaymath}
where $e_j = (0, \ldots, 0, 1, 0, \ldots, 0)$ with $1$ at position
$j$. The corresponding differential operators are (cf.\
\cite[Prop.~2.2.5]{Bu})
\begin{align}
\label{dHdef} dH_j &=
-2y_j\sin(2\vartheta_j)\partial_{x_j}+2y_j\cos(2\vartheta_j)\partial_{y_j}+\sin(2\vartheta_j)\partial_{\vartheta_j},\\
\label{dRdef} dR_j &=
y_j\cos(2\vartheta_j)\partial_{x_j}+y_j\sin(2\vartheta_j)\partial_{y_j}+\sin^2(\vartheta_j)\partial_{\vartheta_j},\\
\label{dLdef} dL_j &=
y_j\cos(2\vartheta_j)\partial_{x_j}+y_j\sin(2\vartheta_j)\partial_{y_j}-\cos^2(\vartheta_j)\partial_{\vartheta_j}.
\end{align}
The action of $\fg$ induces an action of its universal enveloping
algebra $U(\fg)$ by higher order differential operators. This action
commutes with the spectral decomposition \eqref{spectral}, hence for
each $\mathcal{D}\in U(\fg)$ and any sufficiently smooth $\phi \in
L^2(\GL_2(K) \backslash \GL_2(\mathbb{A}),\omega)$ decomposing as
\[\phi = \sum_{\pi\in \mathcal{C}_\omega} \phi_{\pi} + \sum_{\chi^2=\omega}
\phi_{\chi} + \int_{\mathcal{E}_\omega} \phi_{\varpi}\,d\varpi\]
with $\phi_{\pi} \in V_{\pi}$, $\phi_\chi\in V_\chi$,
$\phi_\varpi\in V_\varpi$, it follows that (cf.\ \cite[(33)]{BH})
\begin{equation}\label{DPlancherel}\|\mathcal{D}\phi\|^2=\sum_{\pi\in \mathcal{C}_\omega}
\|\mathcal{D}\phi_{\pi}\|^2 + \sum_{\chi^2=\omega}
\|\mathcal{D}\phi_{\chi}\|^2 + \int_{\mathcal{E}_\omega}
\|\mathcal{D}\phi_\varpi\|^2\,d\varpi.\end{equation} We now define
for any $\mu\in\NN_0$ and any (sufficiently) smooth vector $\phi$
the Sobolev norm
\[\|\phi\|_{S^{\mu}} := \sum_{\ord(\mathcal{D}) \leq \mu} \|\mathcal{D}
\phi\|,\] where $\mathcal{D}$ ranges over all monomials in $H_{j_1},
R_{j_2}, L_{j_3}$ of order at most $\mu$ in $U(\fg)$. Clearly,
\[\|\phi\|_{S^{\mu}}^2 \asymp_\mu \sum_{\ord(\mathcal{D}) \leq \mu} \|\mathcal{D}
\phi\|^2,\] therefore by \eqref{DPlancherel} also
\begin{equation}\label{decompnorm}
 \| \phi \|_{S^{\mu}}^2 \asymp_\mu \sum_{\pi \in \mathcal{C}_\omega} \| \phi_{\pi} \|^2_{S^{\mu}} + \sum_{\chi^2=\omega} \| \phi_{\chi} \|^2_{S^{\mu}}
 + \int_{\mathcal{E}_\omega} \| \phi_\varpi\|^2_{S^{\mu}} \, d\varpi.
\end{equation}

Let $(\pi, V_{\pi})$ be an automorphic representation of
$\GL_2(\mathbb{A})$ generated by a cusp form of arbitrary central
character $\omega$ or an Eisenstein series with trivial central
character $\omega=1$, i.e.\ one of $V_{\pi}$ with $\pi\in
\mathcal{C}_\omega$ or $V_{\chi,\chi^{-1}}$ with $\chi$ an arbitrary
Hecke character which is nontrivial on $\Kplusdiag$. Earlier we
introduced for each ideal $\mt\subseteq\mo$ an isometric embedding
$R^{(\mt)}:V_\pi(\mc_\pi)\hookrightarrow V_\pi$ and a vector space
isomorphism $R^{(\mt)}V_{\pi}(\mc_{\pi})\to
L^2(K_\infty^\times,d^\times y)$, written as $\phi\mapsto W_\phi$,
which satisfy \eqref{RTdecomp2}--\eqref{newWphibound3} and
\eqref{fouriergeneralEisen2}--\eqref{newWphibound2}. Using that the
right actions of $\GL_2(K_\infty)$ and
$\GL_2(\mathbb{A}_\text{fin})$ on $V_\pi$ commute, it is easy to see
that $U(\fg)$ acts on each subspace $R^{(\mt)}V_{\pi}(\mc_{\pi})$
separately, which then induces an action on
$L^2(K_\infty^\times,d^\times y)$. Interestingly, this action is
independent of $\mt$. Indeed, \eqref{dHdef}--\eqref{dRdef} show that
$H_j$ acts by $2y_j\partial_{y_j}$ and $R_j$ acts by $2\pi i y_j$.
Then, the $j$-th Casimir element
\begin{equation}\label{jCasimir}-4\Delta_j=H_j^2+2R_jL_j+2L_jR_j=H_j^2-2H_j+4R_jL_j\end{equation}
acts by the scalar $-4\lambda_{\pi,j}$ (cf.\ \eqref{LB} and
\cite[p.~153]{Bu}), hence $R_jL_j$ acts by
$-\lambda_{\pi,j}+y_j^2\partial^2_{y_j}$ and $L_j$ acts by $(2\pi
i)^{-1}(-\lambda_{\pi,j}y_j^{-1}+y_j\partial^2_{y_j})$. These
formulae justify for any $\mu\in\NN_0$ and any $2\mu$ times
differentiable function $W:K_\infty^\times\to\CC$ the definition of
the Sobolev norm
\[\|W\|_{S^{\mu}} := \sum_{\ord(\mathcal{D}) \leq \mu} \|\mathcal{D}
W\|,\] where $\mathcal{D}$ is as before, and the bound (cf.\
\eqref{short} and \cite[Lemma~8.4]{Ve})
\begin{equation}\label{akshaynorm}
 \|W\|_{S^{\mu}} \ll_{\mu} \left(\max_{1\leq j \leq d}\tilde\lambda_{\pi,j}\right)^{\mu}
 \|W\|_{A^{2\mu}},
\end{equation}
where
\begin{equation}\label{defnorm}
 \| W\|_{A^{\mu}} := \sum_{\substack{\mu_1 + \cdots + \mu_d \leq
 \mu\\\kappa_1\leq\mu_1,\dots,\kappa_d\leq\mu_d}}
 \left(\int_{K_\infty^\times} \left|\partial_{y_1}^{\kappa_1} \cdots \partial_{y_d}^{\kappa_d}W(y)\right|^2
 \prod_{j=1}^d\left(|y_j| + |y_j|^{-1}\right)^{\mu_j}\,d^\times y\right)^{1/2}.
\end{equation}

\begin{lemma}\label{lemma3} Let $(\pi,V_\pi)$ be an automorphic representation of
$\GL_2(K)\backslash \GL_2(\mathbb{A})$ as before, and let $\mt
\subseteq \mo$ be an ideal. Let $a, b, c \in \NN_0$, $0<\eps<1/4$,
and $\theta$ as in \eqref{boundLaplace}. Let
$P\in\CC[x_1,\dots,x_d]$ be a polynomial of degree at most $a$ in
each variable, and consider the differential operator
$\mathcal{D}:=P(y_1\partial_{y_1},\dots,y_d\partial_{y_d})$. Then
for $\phi \in R^{(\mt)}V_{\pi}(\mc_\pi)$ and $y\in K_\infty^\times$
we have, using the notation \eqref{short},
\[\mathcal{D}W_{\phi}(y) \ll_{a, b, c,P,K,\eps}
(\N\mt)^\eps(\N\mc_\pi)^{\eps}(\N\tilde{\lambda}_{\pi})^{-c}
\|\phi\|_{S^{d(5+a+b+2c)}}\prod_{j=1}^d |y_j|^{1/2-\theta
-\eps}\min(1,|y_j|^{-b}).\]
\end{lemma}

\begin{proof} To start with, let us fix $q \in \ZZ^d$ and assume
$\phi \in R^{(\mt)}V_{\pi, q}(\mc_\pi).$ By \eqref {orth},
\eqref{convent}, \eqref{explicitKir}, \eqref{oldKir},
\eqref{fancyWhittaker} we have
\begin{equation}\label{Whitnorm}
 |W_{\phi}(y)| = \| W_{\phi}\|\,|\tilde{W}_{q/2,\nu_{\pi}}(y)|,\qquad y\in
 K_\infty^\times,
\end{equation}
where by \eqref{newWphibound3}, \eqref{newWphibound2},
\eqref{analyticconductor},
\[\|W_\phi\|\ll_{K,\eps}(\N\mt)^\eps(\N\mc_\pi)^{\eps}(\N\tilde{\nu}_{\pi})^{\eps}\|\phi\|.\]
In addition, \eqref{dRdef}--\eqref{dLdef} show that
$(R_j-L_j)\phi=iq\phi$. Now we infer, using
\eqref{whittaker2}--\eqref{convent}, that
\begin{displaymath}
\begin{split}
W_{\phi}(y)&\ll_{K,\eps}(\N\mt)^\eps(\N\mc_\pi)^{\eps}(\N\tilde{\nu}_{\pi})^{\eps}\|\phi\|
\prod_{j=1}^d (\tilde\nu_{\pi,j}+|q_j|)^{1+\theta}(|y_j|^{1/2-\eps}
+ |y_j|^{1/2 - \theta-\eps})\\
&\ll_{K,\eps}(\N\mt)^\eps(\N\mc_\pi)^{\eps}(\N\tilde{\nu}_{\pi})^{1+\theta+\eps}
\Bigl\|\prod_{j=1}^d\bigl(1-(R_j-L_j)^2\bigr)\phi\Bigr\|\prod_{j=1}^d
\frac{|y_j|^{1/2-\eps} +
|y_j|^{1/2-\theta-\eps}}{(1+|q_j|)^{1-\theta}}.
\end{split}
\end{displaymath}
For an arbitrary $\phi \in R^{(\mt)}V_{\pi}(\mc_\pi)$ with weight
decomposition (cf.\ \eqref{weightdecomp}--\eqref{oldforms})
\[\phi=\sum_{q\in\ZZ^d}\phi_{q},\qquad\phi_q\in R^{(\mt)}V_{\pi, q}(\mc_\pi),\]
we apply the operator
$\mathcal{D}':=\prod_{j=1}^d\bigl(1-(R_j-L_j)^2\bigr)$ on both sides
obtaining the weight decomposition
\[\mathcal{D}'\phi=\sum_{q\in\ZZ^d}\mathcal{D}'\phi_{q},\qquad\mathcal{D}'\phi_q=(1+q^2)\phi_q\in R^{(\mt)}V_{\pi, q}(\mc_\pi).\]
In particular,
$\|\mathcal{D}'\phi\|^2=\sum_{q\in\ZZ^d}\|\mathcal{D}'\phi_{q}\|^2$,
hence the previous bound and Cauchy--Schwarz yield
\[W_{\phi}(y)\ll_{K,\eps}
(\N\mt)^\eps(\N\mc_\pi)^{\eps}(\N\tilde\lambda_{\pi})
\|\phi\|_{S^{2d}}\prod_{j=1}^d (|y_j|^{1/2-\eps} + |y_j|^{1/2 -
\theta-\eps}).\] Depending on $a,b,c\in\NN_0$ and $y\in
K_\infty^\times$, we replace $\phi$ by $\mathcal{D}''\phi\in
R^{(\mt)}V_{\pi}(\mc_\pi)$ with
\[\mathcal{D}'':=\Biggl(\prod_{1\leq j\leq d}(1/2 + H_j^2 + 2R_jL_j + 2 L_j
R_j)^{c+1}\Biggr) \Biggl(\prod_{\substack{1\leq j\leq
d\\|y_j|>1}}R_j^{b+1}\Biggr)P(H_1,\dots,H_d),\] then we obtain the
general bound of the lemma by combining \eqref{jCasimir} and
$|1/2-4\lambda_{\pi,j}|>\tilde\lambda_{\pi,j}/5$ for each $1\leq
j\leq d$.
\end{proof}

\begin{lemma}\label{lemma4} Let $(\pi, V_{\pi})$ be an irreducible cuspidal representation
of $\GL_2(K)\backslash \GL_2(\mathbb{A})$ with unitary central
character, and let $\phi \in V_{\pi}(\mc_{\pi})$ be such that $\|
\phi \|_{S^{3d}}$ exists. Then
\begin{displaymath}
\|\phi \|_{\infty} := \sup_{g \in \GL_2(\mathbb{A})}| \phi(g)|
\ll_{\pi,K} \| \phi \|_{S^{3d}}.
\end{displaymath}
\end{lemma}

\begin{remark} It is relatively easy to show that the implied constant depends polynomially on
$C(\pi)$, and the order of the Sobolev norm could also be lowered
easily. However, it seems hard and would be interesting to find
close to optimal bounds for the sup-norm of an automorphic form in
terms of the $L^2$-norm (or some small Sobolev norm) and the various
parameters of $\pi$. For strong results in this direction see the work of Bernstein
and Reznikov~\cite{BR}.
\end{remark}

\begin{proof}
Let us first assume that $\phi \in V_{\pi, q}(\mc_{\pi})$, i.e.\
$\phi$ is of pure weight $q \in \ZZ^d$. Let $g \in
\GL_2(\mathbb{A})$, and let $i_1, \dots, i_h \in
\mathbb{A}_{\text{fin}}^\times$ be $h$ finite ideles representing
the ideal classes of $K$. By strong approximation
\cite[Theorem~3.3.1]{Bu} there are $\gamma \in \GL_2(K)$,
$g'\in\GL_2(K_\infty)$, and $k\in\KK(\mo)$ such that
\begin{equation*}
 g = \gamma \left(g' \times \left(\begin{smallmatrix}i_j & 0\\ 0&
1\end{smallmatrix}\right)k\right)
\end{equation*}
for some $1 \leq j \leq h$. It follows from \cite[p.~36 and
p.~67]{Fr} that there are elements $a_1,\dots,a_{2^d h}\in\GL_2(K)$
regarded as elements of $\GL_2(K_\infty)$ and some $\delta
>0$ depending only on $K$ such that for suitable $z\in
Z(K_\infty)$, $\gamma'\in\SL_2(\mo)$ regarded as an element of
$\SL_2(K_\infty)$, and $k'\in\SO_2(K_\infty)$ we have
\begin{equation*}
 g'=z\gamma'a_{j'} \left(\begin{smallmatrix}y' & x'\\ 0 &
 1\end{smallmatrix}\right)k'
\end{equation*}
for some $1 \leq j' \leq 2^d h$ and some
$\left(\begin{smallmatrix}y' & x'\\ 0 &
 1\end{smallmatrix}\right) \in P(K_{\infty})$ with
$y'_1,\dots,y'_d > \delta$. Combining with the previous display we
obtain
\begin{displaymath}
 g = z \gamma\gamma'a_{j'}
 \left( \left(\begin{smallmatrix} y' & x' \\ 0 &
 1\end{smallmatrix}\right)k'\times a_{j'}^{-1}\gamma'^{-1}\left(\begin{smallmatrix}i_j & 0\\ 0&
1\end{smallmatrix}\right)k\right),
\end{displaymath}
where $z\in Z(K_\infty)$ is now regarded as an element of
$Z(\mathbb{A})$, and the first (resp. second) occurrences of
$\gamma'$ and $a_{j'}$ are regarded as elements of
$\GL_2(\mathbb{A})$ (resp. of $\GL_2(\mathbb{A}_\text{fin})$). Here
$\gamma\gamma'a_{j'}\in\GL_2(K)$, while
$a_{j'}^{-1}\gamma'^{-1}\left(\begin{smallmatrix}i_j & 0\\
0& 1\end{smallmatrix}\right)k$ lies in a fixed compact subset of
$\GL_2(\mathbb{A}_\text{fin})$ depending only on $K$ which can be
covered by finitely many left cosets of the open subgroup
$\KK(\mc_{\pi})$. It follows that
\begin{displaymath}
 g = z \tilde{\gamma} \left(\begin{smallmatrix} y & x \\ 0 & 1\end{smallmatrix}\right) (\tilde{k}_{\infty} \times \tilde{k}_{\text{fin}})
\end{displaymath}
for some $\tilde{\gamma} \in \GL_2(K)$, $\tilde{k} =
\tilde{k}_{\infty} \times \tilde{k}_{\text{fin}} \in
\SO_2(K_{\infty})\times \KK(\mc_{\pi})$, and
$\left(\begin{smallmatrix} y & x \\ 0 & 1\end{smallmatrix}\right)
\in P(\mathbb{A})$, where $y = y_{\infty} \times y_{\text{fin}}$ is
such that all coordinates of $y_{\infty}$ exceed $\delta$ and
$y_{\text{fin}}$ takes values from a finite set depending only on
$K$ and $\mc_{\pi}$. Thus the Fourier expansion
\eqref{fouriersimple} together with \eqref{boundlambda},
\eqref{Whitnorm} and Lemma~\ref{lemma2} gives
\begin{displaymath}
 |\phi(g)| = \left|\phi\left(\left(\begin{matrix} y & x \\ 0 &
1\end{matrix}\right) \right)\right|
 \leq \sum_{\substack{r \in (y_{\text{fin}}^{-1})\\r\neq 0}} \frac{|\lambda_{\pi}(r y_{\text{fin}})|}{\sqrt{\N (r y_\text{fin})}} |W_{\phi}(r y_\infty)|
 \ll_{\pi,K} \|\phi\|\sum_{\substack{r \in (y_{\text{fin}}^{-1})\\r\neq 0}} |\tilde{W}_{q/2, \nu_{\pi}}(r
 y_{\infty})|,
 \end{displaymath}
where $(y_{\text{fin}}^{-1})=y_{\text{fin}}^{-1}\mo$ is the
fractional ideal corresponding to $y_{\text{fin}}^{-1}$. We fix some
$0 < \eps < 1/20$ and let
\begin{align*}
 R_1 &:= \biggl\{ r \in (y_{\text{fin}}^{-1}) \ \Big|\
 |r^{\sigma_j}| < \delta^{-1}(|q_j|+|\nu_{\pi,j}|+1)\prod_{j=1}^d(|r^{\sigma_j}|+|q_j|)^\eps\biggr\},\\
 R_2 &:= (y_{\text{fin}}^{-1}) - R_1.
\end{align*}
Using \eqref{whittaker1}--\eqref{convent} and the property of
$y_\infty$, we find
\begin{displaymath}
 \sum_{\substack{r \in R_2\\r\neq 0}} |\tilde{W}_{q/2, \nu_{\pi}}(r y_{\infty})| \ll_{\pi,K,\eps} 1
\end{displaymath}
and
\begin{displaymath}
\sum_{\substack{r \in R_1\\r\neq 0}} |\tilde{W}_{q/2, \nu_{\pi}}(r
y_{\infty})|
 \ll_{\pi,K} \#R_1 \prod_{j=1}^d(1+|q_j|)^{1+\theta} \ll_{\pi,K,\eps} \prod_{j=1}^d(1+|q_j|)^{2+\theta+\eps}.
\end{displaymath}
By \eqref{dRdef}--\eqref{dLdef} we see now that for any $\phi \in
V_{\pi, q}(\mc_{\pi})$ we have
\begin{displaymath}
 \| \phi \|_{\infty} \ll_{\pi,K} \Bigl\|\prod_{j=1}^d (1+R_j-L_j)^3 \phi
 \Bigr\|\prod_{j=1}^d(1+|q_j|)^{-2/3} .
\end{displaymath}
Using Cauchy--Schwarz and Parseval we can infer for a general $\phi
\in V_{\pi}(\mc_{\pi})$ that
\begin{displaymath}
 \| \phi \|_{\infty}\ll_{\pi,K} \|\phi\|_{S^{3d}},
\end{displaymath}
assuming the right hand side exists.
\end{proof}

\subsection{Waldspurger's theorem and generalizations}

Let $r \in \mo$ be a nonzero squarefree integer, i.e. $0\leq
v_\mpr(r) \leq 1$ for all prime ideals $\mpr\subseteq\mo$. If
$\chi_r$ denotes the quadratic character associated to the extension
$K(\sqrt{r})/K$, the central value $L(1/2,\pi \otimes \chi_r)$ is
related to the square of the $r$-th Fourier coefficient of a
half-integral weight Hilbert modular form. The prototype of such a
theorem for $K = \QQ$ goes back to Waldspurger~\cite{Wa} with
refinements by Kohnen--Zagier~\cite{KZ,Koh}. For an arbitrary
totally real number field $K$, precise results of this type can be
found for example in \cite[Theorem~8.1]{KM} and
\cite[Theorem~4.3]{BM}. Using these, one can turn a bound for
twisted central $L$-values into a bound for the Fourier coefficients
of a half-integral weight Hilbert modular form. An explicit
statement of this phenomenon is \cite[Theorem~1.5]{BM} which we
recall below.

Let $\widetilde{\SL}_2$ denote the metaplectic double cover of
$\SL_2$, and let $(\tilde{\pi},V_{\tilde\pi})$ be an irreducible
cuspidal representation of $\widetilde{\SL}_2(K)\backslash
\widetilde{\SL}_2(\mathbb{A})$ orthogonal to the theta series
generated by quadratic forms in one variable. Let $\pi$ be the
unique irreducible cuspidal representation of
$\GL_2(K)Z(\mathbb{A})\backslash\GL_2(\mathbb{A})$ associated to
$\tilde\pi$ by the Shimura--Waldspurger correspondence
\cite{Wa1,Wa3}. Define the $r$-th Fourier coefficient of a smooth
vector $\tilde\phi\in V_{\tilde\pi}$ as
\begin{displaymath}
 \tilde{W}_{\tilde{\phi}}^r := \int_{K\backslash\mathbb{A}} \tilde{\phi}\left(\begin{pmatrix}1& x\\ 0 &
 1\end{pmatrix}\right)\psi(-rx)\,dx.
\end{displaymath}
Assume that $\tilde\phi$ is a pure tensor $\otimes_v\tilde\phi_v$,
and for each archimedean place $1 \leq j \leq d$ define a quantity
$e(\tilde\phi_j, r)$ as in \cite[(4.3)]{BM}, cf.\ also
\cite[Section~2.2]{BM}. For $\tilde\phi_j$ belonging to the
holomorphic discrete series, this quantity is calculated explicitly
in \cite[Prop.~8.8]{BM}. Assume that there is a bound
\begin{displaymath}
 L(1/2,\pi \otimes \chi_r) \ll_{\pi, K} |\N r|^{\beta}
\end{displaymath}
for some $\beta > 0$, then one has
\begin{equation}\label{baruchmao}
 \tilde{W}_{\tilde{\phi}}^r \prod_{j=1}^d e(\tilde\phi_j, r) \ll_{\tilde{\phi}, K} |\N r|^{\frac{\beta-1}{2}}.
\end{equation}
By the last two displays, Corollary~\ref{cor5} is an immediate
consequence of Theorem~\ref{theorem1}.

\subsection{Kuznetsov's formula}

There are several adelic \cite{Ye, KL1} and classical \cite{BMP1}
versions of the Kuznetsov formula over number fields available in the literature. For
our purposes, a slightly generalized version of the
``semi-classical" formula given in \cite{V1} (which in turn is based
on \cite{BrMi}) is the most suitable. The extension is needed
because \cite{V1} deals only with representations that are spherical
at infinity (i.e.\ totally even non-exceptional Hilbert--Maa{\ss} forms), while
we need to include holomorphic forms and totally even exceptional Hilbert--Maa{\ss} forms.
Fortunately, the necessary integral transforms together with sharp
estimates are provided in full detail in \cite{BMP1}, so we can
quote the results and restrict ourselves to a brief exposition.

We introduce the set \[\mathcal{S}:=\left\{\nu \in\CC : |\Re \nu| <
\frac{2}{3}\right\}\cup \left(\frac{1}{2} + \ZZ\right),\] and for
each $1 \leq j \leq d$ we consider an even function
$k_j:\mathcal{S}\to\CC$, holomorphic on the interior of
$\mathcal{S}$, which satisfies the decay condition $k_j(\nu) \ll
(1+|\nu|)^{-a}$ for some $a>2$. We write
\[k(\nu): = \prod_j
k_j(\nu_j),\qquad \nu \in \mathcal{S}^d.\]
Following
\cite[Definitions~2.5.2--2.5.4 and (25)]{BMP1}, we define the Bessel
transforms
\begin{displaymath}
\begin{split}
 \check{k}_j(t) :=&-i\int_{(0)} k_j(\nu) J_{2\nu}(4
 \pi\sqrt{t})\frac{\nu\,d\nu}{\cos(\pi\nu)}
 + \sum_{b\geq 2 \text{ even}} (-1)^{b/2}(b-1) k_j\left(\frac{b-1}{2}\right) J_{b-1}(4 \pi
 \sqrt{t}),\quad t>0;\\
 \check{k}_j(t) :=&-i\int_{(0)} k_j(\nu) I_{2\nu}(4 \pi\sqrt{|t|})\frac{\nu\,d\nu}{\cos(\pi\nu)},\quad t<0;\\
 \tilde{k}_j := & \frac{i}{2}\int_{(0)} k_j(\nu) \nu\tan(\pi\nu)d\nu + \sum_{b\geq 2 \text{ even}}\frac{b-1}{2}
 k_j\left(\frac{b-1}{2}\right).
\end{split}
\end{displaymath}
Let $\mc,\my_1,\my_2\subseteq\mo$ be nonzero ideals. Within a fixed
set of representatives of all ideal classes of $K$ we define $C$ as
the subset of ideals $\ma $ satisfying $\ma
^2\my_1\my_2\mathfrak{d}^2 \sim 1$. For each $\ma \in C$ we fix,
once and for all, a generator $\gamma$ of the principal ideal $\ma
^2\my_1\my_2\mathfrak{d}^2$. For any nonzero elements $c \in \mc\ma
^{-1}$, $r_1\in\my_1$, $r_2\in\my_2$ we define the Kloosterman sum
\[S(r_1, \my_1; r_2 ,\my_2; c, \ma ) :=
{\rm KS}(r_1, (\my_1\mathfrak{d})^{-1}; r_2\gamma^{-1}, (\my_2
\mathfrak{d})^{-1}; c, \ma ),\] where the right hand side is given
by \cite[Def.~2]{V1}. We only need to know Weil's bound for this
type of Kloosterman sum \cite[(13)]{V1}
\begin{equation}\label{weil}
 S(r_1, \my_1; r_2, \my_2; c, \ma )
 \ll_{K,\eps}
 (\N \gcd(r_1\my_1^{-1}, r_2\my_2^{-1}, c\ma ))^{1/2} (\N (c\ma ))^{1/2+\eps}.
\end{equation}
Since we will not need the details later, we suppress a detailed
discussion of the continuous spectrum contribution and follow
\cite{V1} to abbreviate this quantity (whose exact shape is
irrelevant for our purposes) by $\CSC$. Then we have the following
variant of Kuznetsov's formula:
\begin{equation}\label{kuz}
\begin{split}
&[\KK(\mo) : \KK(\mc)]^{-1}\sum_{\substack{\pi \in
\mathcal{C}(\mc)\\\eps_\pi=1}} C_{\pi}^{-1} \sum_{\mt \mid
\mc\mc_{\pi}^{-1}} k(\nu_{\pi})
\bar{\lambda}_{\pi}^{(\mt)}(r_1\my_1^{-1})
\lambda_{\pi}^{(\mt)} (r_2\my_2^{-1}) + \CSC\\
&= c_1 \delta(r_1\my_1^{-1}, r_2\my_2^{-1}) \prod_{j=1}^d
\tilde{k}_j + c_2 \sum_{\ma \in C} \sum_{u \in U/U^2} \sum_{c \in
\mc\ma^{-1}} \frac{S(r_1, \my_1; ur_2, \my_2; c, \ma )}{\N (c\ma )}
\prod_{j=1}^d \check{k}_j\left(\left(\frac{ur_1r_2}{\gamma
c^2}\right)^{\sigma_j}\right).
\end{split}
\end{equation}
Here $\pi$ runs over all totally even cuspidal representations of
trivial central character and conductor dividing $\mc$ (cf.
\eqref{defCE} and \eqref{fourierwhittaker}); $C_{\pi}$ was defined
in \eqref{constant} and estimated in Lemma~\ref{lemma2}; the
coefficients $\lambda_\pi^{(\mt)}$ are those in \eqref{Fourgen};
$c_1$, $c_2$ are certain positive constants depending \emph{only} on
$K$; finally $\delta(\ma , \mb ) = 1$ if and only if $\ma = \mb $.

We will only discuss the main ideas of the proof, since all
ingredients can be found in detail in \cite{V1, BrMi, BMP1}. The
transition between the classical and adelic versions of the Kuznetsov formula
is based on the fact that for any nonzero ideal $\my\subseteq\mo$
there is an embedding of coset spaces
\[\Gamma(\my,\mc)Z(K_\infty)\backslash\GL_2(K_\infty)\hookrightarrow \GL_2(K)Z(K_\infty)\backslash\GL_2(\mathbb{A})/\KK(c)\]
given by
\begin{equation*}
\Gamma(\my,\mc)Z(K_\infty)g\mapsto
\GL_2(K)Z(K_\infty)g\bigl(\begin{smallmatrix}\eta^{-1} & 0\\ 0 &
 1\end{smallmatrix}\bigr)\KK(\mc),\qquad
 g\in\GL_2(K_\infty),\end{equation*}
where $\eta\in \mathbb{A}^{\times}_{\text{fin}}$ is any finite idele
representing $\my$. Indeed, it is straightforward to see that the
above map is well-defined and injective by combining \eqref{gamma}
with \[\prod_{\mpr} \KK(\my_\mpr,
\mc_\mpr)=\bigl(\begin{smallmatrix}\eta^{-1} & 0\\ 0 &
 1\end{smallmatrix}\bigr)\KK(\mc)\bigl(\begin{smallmatrix}\eta & 0\\ 0 &
 1\end{smallmatrix}\bigr).\]
In addition, strong approximation \cite[Theorem~3.3.1]{Bu} shows
that the image consists of the double cosets whose union equals
(cf.\ Section~\ref{numberfields})
\[\left\{g\in\GL_2(\mathbb{A})\mid\det(g)\in\eta^{-1} K^\times K_\infty^\times\Omega\right\}.\]
Therefore if $\my_1, \dots, \my_h$ represent the ideal classes of
$K$, then we obtain a decomposition of spaces (cf.\
\cite[Section~6.1]{V1})
\[\GL_2(K)Z(K_\infty)\backslash\GL_2(\mathbb{A})/\KK(c)
\cong\coprod_{j=1}^h
\Gamma(\my_j,\mc)Z(K_\infty)\backslash\GL_2(K_\infty).\] The Haar
measure on $Z(K_\infty)\backslash\GL_2(\mathbb{A})$ defined in
Section~\ref{measures} gives rise to a Borel measure on the left
hand side assigning to each Borel set the measure of its preimage in
$\GL_2(K)Z(K_\infty)\backslash\GL_2(\mathbb{A})$ under the natural
projection. The Haar measure on
$Z(K_\infty)\backslash\GL_2(K_\infty)$ defined in
Section~\ref{measures} induces a Borel measure on the right hand
side. Now an important feature is that the measure of each Borel set
on the left hand side is exactly $[\KK(\mo) : \KK(\mc)]^{-1}$ times
the measure of the corresponding Borel set on the right hand side.
To see this claim, it suffices to show that for any nonzero ideal
$\my\subseteq\mo$ and for any Borel set $U\subseteq\GL_2(K_\infty)$
representing distinct cosets $\Gamma(\my,\mc)Z(K_\infty)g$ for $g\in
U$, we have
\[\vol\bigl(\GL_2(K)Z(K_\infty)\backslash \GL_2(K)Z(K_\infty)U\bigl(\begin{smallmatrix}\eta^{-1} & 0\\ 0 &
 1\end{smallmatrix}\bigr)\KK(\mc)\bigr)
 =[\KK(\mo) : \KK(\mc)]^{-1}\vol(Z(K_\infty)\backslash Z(K_\infty) U).\]
By the discussion above, the double cosets
$\GL_2(K)Z(K_\infty)g\bigl(\begin{smallmatrix}\eta^{-1} & 0\\
0&1\end{smallmatrix}\bigr)\KK(\mc)$ for $g\in U$ are distinct, hence
the left hand side equals
\begin{align*}
\vol\bigl(Z(K_\infty)\backslash Z(K_\infty)
U\bigl(\begin{smallmatrix}\eta^{-1} & 0\\ 0 &
 1\end{smallmatrix}\bigr)\KK(\mc)\bigr)
&= \vol(Z(K_\infty)\backslash Z(K_\infty)
U)\vol\bigl(\bigl(\begin{smallmatrix}\eta^{-1} & 0\\
0&1\end{smallmatrix}\bigr)\KK(\mc)\bigr)\\
&=\vol(Z(K_\infty)\backslash Z(K_\infty)
U)\vol(\KK(\mc)).\end{align*} Since $\vol(K(\mc))=[\KK(\mo) :
\KK(\mc)]^{-1}$ by $\vol(K(\mo))=1$, our claim follows.

Let $T$ denote the subgroup of elements
$\left(\begin{smallmatrix}\ast&0\\0&1\end{smallmatrix}\right)\in\OO_2(K_\infty)$,
i.e.\ those with coordinates $\left(\begin{smallmatrix}\pm
1&0\\0&1\end{smallmatrix}\right)$. Then $T$ represents
$\OO_2(K_\infty)/\SO_2(K_\infty)$ and the above discussion shows
\[\GL_2(K)Z(K_\infty)\backslash\GL_2(\mathbb{A})/T\KK(c)
\cong\coprod_{j=1}^h
\Gamma(\my_j,\mc)Z(K_\infty)\backslash\GL_2(K_\infty)/T\] with
similarly related Borel measures on the two sides. We denote by
$\FS$ the $L^2$-space of the left hand side, viewed as a Hilbert
space of measurable functions $\phi:\GL_2(\mathbb{A})\to\CC$ which
are left $\GL_2(K)Z(K_\infty)$-invariant and right
$T\KK(\mc)$-invariant. This space is analogous to $\FS$ in
\cite[Section~2.3]{V1} for the special case $\chi=1$, the only
difference being that instead of right $\OO_2(K_\infty)$-invariance
we require right $T$-invariance. We clearly have
\begin{equation}\label{FS}
\FS\cong\bigoplus_{j=1}^h
L^2(\Gamma(\my_j,\mc)Z(K_\infty)\backslash\GL_2(K_\infty)/T),\end{equation}
and in order to derive \eqref{kuz}, we follow \cite[Section~6]{V1}.
The proof is based on a geometric and spectral evaluation of a
certain inner product formed of two Poincar\'e series on $\FS$, each
of which is supported in only one component on the right hand side
of \eqref{FS}. The spectral expansion is carried out in an
orthonormal basis of the right hand side of \eqref{FS} which,
according to our discussion above, is provided by $[\KK(\mo) :
\KK(\mc)]^{-1/2}$ times any orthonormal basis of $\FS$. For the
latter we make use of the decomposition
\begin{equation}\label{FSspace}
\FS=\bigoplus_{\omega\in\widehat{C(K)}}L^2(\GL_2(K)\backslash\GL_2(\mathbb{A})/T\KK(\mc),\omega),\end{equation}
where each class group character $\omega$ is regarded as a character
of $\mathbb{A}^\times$ trivial on $K^\times K_\infty^\times\Omega$,
and the corresponding component is the right $T\KK(\mc)$-invariant
subspace of $L^2(\GL_2(K)\backslash\GL_2(\mathbb{A}),\omega)$, in
the notation of Section~\ref{spectralsection}. Utilizing
\eqref{spectral}, \eqref{RTdecomp2}, \eqref{RTdecomp3}, we can form
an orthonormal basis of $\FS$ from certain totally even automorphic
forms of pure weight, level $\mc$ and central character trivial on
$K^\times K_\infty^\times\Omega$. By averaging over $C(K)$ several
inner products associated with the same pair of Poincar\'e series,
we can ensure that only $\omega=1$ contributes to the final spectral
expansion. This outline explains the structure of the left hand side
of \eqref{kuz}.

To carry out the above plan, we need to work with slightly more
general Poincar\'e series than \cite[(89)]{V1}, namely we only
require right $T$-invariance instead of right
$\OO_2(K_\infty)$-invariance. Then the geometric evaluation
\cite[Section~6.3]{V1} goes through with no changes, but in the
spectral evaluation \cite[Section~6.4]{V1} the integrals over $A$
have to be replaced by integrals over $A\!\OO_2(K_\infty)$. Now we
choose the test function $f$ as in \cite[Section 5.3, see also
Def.~5.2.4]{BMP1}. The special integrals in \cite[(96a), (96b),
(98)]{V1} are evaluated in \cite[Section 6.5]{V1} using the
relations \cite[(25), Prop.~9.4, (26)]{BrMi}, respectively. In our
more general situation, we use the corresponding results \cite[(83),
Prop.~5.2.6, (84)]{BMP1}.\footnote{It may be noted that in
\cite{BMP1} the authors are faced with more subtle convergence
issues; this is, for example, reflected in the fact that in
\cite[Def.~5.2.4]{BMP1} weight $2$ Maa{\ss } forms rather than
weight $0$ Maa{\ss } forms are used which leads to the correction
factor $(\frac{1}{4} - \nu^2)^{-1}$ in \cite[(87)]{BMP1}.} This
yields the formula \eqref{kuz} as in \cite[Section~6.6]{V1}.

\begin{remark}
Unfortunately, there are different concurrent normalizations in the
literature which makes it a little tedious to compare the various
papers. For the convenience of the reader we give an account of the
differences. There are three sources of different
notation/normalization:
\begin{itemize}
\item {\em Groups.} As mentioned in Section~\ref{matrixgroups}, our congruence subgroups are slightly
different from those in \cite{V1}; our $\KK(\my, \mc) \subseteq
\GL_2(K_{\mpr })$ is precisely the group $K_{0, \mpr }(\mc,
(\my\mathfrak{d})^{-1})$ defined in \cite[Section~2.2]{V1}.
\item {\em Measures.} In \cite{BrMi, BMP1} the group $N(K_{\infty})$ of upper triangular
unipotent matrices is equipped with the measure $ \pi^{-d} dx_1
\cdots dx_d $ (with $dx$ the usual Lebesgue measure) whereas we have
normalized the measure in Section~\ref{measures} as
$|D_K|^{-1/2}dx_1 \cdots dx_d$. Venkatesh~\cite{V1} follows the
normalization in \cite{BrMi, BMP1}.\footnote{We remark, however,
that a comparison of \cite[(11)]{V1} and \cite[(5)]{BrMi} shows that
the factor $\vol(\Gamma_N\backslash N)^{-1}$ in \cite[(5)]{BrMi} is
wrongly adapted in \cite[(11)]{V1} as $(\N \mathfrak{d})^{-1/2}$
instead of $2^{d_\CC}\pi^{d_{\RR} +d_{\CC}} (\N
\mathfrak{d})^{-1/2}$, cf.\ \cite[(96)]{V1}.}
\item {\em Whittaker functions.} Our normalization of Whittaker functions coincides with that of
\cite{BMP1} for $\Re \nu = 0$, up to a factor of absolute value
$\sqrt{2\pi}$ at each archimedean place, cf.\ \eqref{normwhit2} and
\cite[(16)]{BMP1}. If $\Re \nu \neq 0$, the discrepancy between
\cite[(16)]{BMP1} and our definition \eqref{normwhit} is compensated
by \cite[(15)]{BMP1}. In \cite{BrMi, V1} only the weight 0 case is
treated, and hence the authors use $\sqrt{y} K_{\nu}(2 \pi y) =
\frac{1}{2}W_{0, \nu}(4 \pi y)$ (unnormalized!) as a Whittaker
function. This scales the Fourier coefficients up by a factor
$\pi^{-1}(2\cos (\pi \nu))^{1/2}$, cf.\ also the remark before
\cite[Def.~2.5.2]{BMP1}. Accordingly, the decay conditions of the
test function in \cite{BMP1} do not include exponentials, and the
measure in \cite[Def.~2.5.2]{BMP1} contains the function
$\tan(\pi\nu)$ rather than $\sin(\pi \nu)$.
\end{itemize}
\end{remark}

As a first application of the Kuznetsov formula and a warm-up for
later calculations we will deduce a weak Weyl law that will give an
upper bound of roughly the expected order of magnitude.

\begin{lemma}\label{weyl} Let $\mc, \mathfrak{m} \subseteq \mo$ be two nonzero ideals,
let $X\in [1, \infty)^d$, and write $\Xi := \prod_{j=1}^d X_j$. Then
for any $\eps>0$ we have
\begin{align*}
\sum_{\substack{\pi \in \mathcal{C}(\mc)\\\eps_\pi=1\\ |\nu_{\pi,
j}| \leq X_j}}
1 &\ll_{K,\eps} \Xi^{2+\eps} (\N\mc )^{1+\eps}\\
\intertext{and}
\sum_{\substack{\pi \in
\mathcal{C}(\mc)\\\eps_\pi=1\\ |\nu_{\pi, j}| \leq X_j}} \sum_{\mt
\mid \mc\mc_{\pi}^{-1}} |\lambda_{\pi}^{(\mt)} (\mathfrak{m})|^2
&\ll_{K,\eps} \Xi^{2+\eps} \left((\N\mc)^{1+\eps} +
(\N\gcd(\mathfrak{m},\mc))^{1/2}(\N\mathfrak{m})^{1/2+\eps}\right).
\end{align*}
\end{lemma}

\begin{remark}
This should be compared with Lemma~\ref{babyWeyl}. The first bound
with unspecified exponents is contained in \cite[(9.3)]{MV}.
\end{remark}

\begin{proof}
The first bound follows from the second with $\mathfrak{m} = \mo$ by
noting that $\lambda_\pi^{(\mo)}(\mo)=\lambda_\pi(\mo)=1$. To prove
the second bound, we choose an ideal class representative $\my \sim
\mathfrak{m}^{-1}$ from a fixed set depending on $K$, then
$\mathfrak{m} = r \my^{-1}$ with $r\in\my$ and $\N(r) \asymp_K \N
\mathfrak{m}$. We apply Kuznetsov's formula \eqref{kuz} with
$r_1=r_2=r$, $\my_1=\my_2=\my$, $k(\nu) := \prod_{j=1}^d
k_{X_j}(\nu_j)$, where for any $Z\geq 1$
\begin{equation}\label{testkZ}
k_Z(\nu):=
\begin{cases}
e^{(\nu^2 - 1/4)/Z^2}, & |\Re \nu| < \frac{2}{3},\\
1, & \text{$\nu \in \frac{1}{2}+\ZZ$ and $\frac{3}{2}\leq|\nu|\leq
Z$}.
\end{cases}
\end{equation}
By \cite[pp.~124--126]{BMP1} we have
\begin{equation}\label{boundstrafo}
 \check{k}(t) \ll Z^{2}\min(1, |t|^{1/2})\qquad\text{and}\qquad \tilde{k} \ll Z^2.
\end{equation}
By Lemma~\ref{lemma2} we have $[\KK(\mo) : \KK(\mc)] C_{\pi}
\ll_{K,\eps} \Xi^\eps(\N \mc)^{1+\eps}$ for the relevant $\pi$.
Hence the diagonal term contributes $\ll_{K, \eps} \Xi^{2+\eps}(\N
\mc)^{1+\eps}$. By \eqref{weil} and \eqref{boundstrafo}, the
off-diagonal contribution is at most\begin{displaymath}
\begin{split}
 \ll_{K,\eps} \Xi^{2+\eps} (\N \mc)^{1+\eps} \max_{\ma \in C}
 \sum_{0 \neq c \in \mc\ma^{-1}} \frac{(\N \gcd(\mathfrak{m},c\ma))^{1/2}}{(\N(c))^{1/2-\eps}}
 \prod_{j=1}^d \min\left(1, \left|\left( \frac{r^2}{c^2}\right)^{\sigma_j}\right|^{1/2}\right).
 \end{split}
\end{displaymath}
We now use \cite[Lemma~8.1]{BrMi} as in the proof of
\cite[Lemma~3.2.1]{BMP1}. We infer that the $c$-sum is
\begin{displaymath}
\begin{split}
&\ll_K \sum_{0 \neq (c) \subseteq \mc\ma^{-1}} \frac{(\N
\gcd(\mathfrak{m},c\ma))^{1/2}}{(\N (c))^{1/2-\eps}} \left(1 +
\left|\log\frac{\N(c^2)}{\N(r^2)}\right|^{d-1}\right) \min\left(1,
\left(\frac{\N(r^2)}{\N(c^2)}\right)^{1/2}\right)\\
&\ll_{K,\eps} \sum_{0 \neq (c) \subseteq \mc\ma^{-1}} \frac{(\N
\gcd(\mathfrak{m},c\ma))^{1/2}}{(\N (c))^{1/2-\eps}}
\left(\frac{\N(r^2)}{\N(c^2)}\right)^{1/4+\eps}.
\end{split}
\end{displaymath}
The last sum extends in a natural fashion to all nonzero ideals
contained in $\mc\ma^{-1}$, therefore by a standard argument it is
at most
\[\ll_{K,\eps}\frac{(\N\gcd(\mathfrak{m},\mc))^{1/2}
(\N\mathfrak{m})^{1/2+3\eps}}{(\N(\mc\ma^{-1}))^{1+\eps}}.\] Altogether the
off-diagonal term contributes
\begin{equation}\label{endweyl}
 \ll_{K, \eps}
 \Xi^{2+\eps}(\N\gcd(\mathfrak{m},\mc))^{1/2}(\N\mathfrak{m})^{1/2+\eps}.
 \end{equation}
\end{proof}

\section{Part II: Subconvexity and shifted convolution sums}

\subsection{Heuristic explanation of the exponent}\label{heuristic}

The Burgess exponent $3/8$ for $\GL_2$, or $3/16$ for $\GL_1$, seems
to be a universal barrier, and there are several quite distinct
methods that independently yield it (perhaps in a slightly weaker
version coming from possible non-tempered representations).
Therefore it might be instructive to sketch the subconvexity
argument neglecting all the technical details to show where the
exponents come from in our method. This is not intended to be a
proof of any kind, but an experienced reader will have little
difficulty in reconstructing a rigorous proof from the following
remarks. For simplicity let us assume that $K = \QQ$, and the
conductor of $\pi$ is $1$, and let us also assume the
Ramanujan--Petersson conjecture.
Moreover we will not display epsilons.\\

We consider the amplified moment
\begin{displaymath}
 \sum_{\omega \, (\text{mod } q)} \left|\sum_{\ell \sim L} \omega(\ell)\bar{\chi}(\ell)\right|^2 | L(1/2, \pi \otimes \omega)|^2.
\end{displaymath}
On the one hand, this is
\begin{equation}\label{lowerbound}
 \gg L^2 |L(1/2, \pi \otimes \chi)|^2,
\end{equation}
on the other hand, this is
\begin{displaymath}
 \ll q \sum_{\ell_1, \ell_2 \sim L} \chi(\ell_1)\bar{\chi}(\ell_2) \sum_{ \ell_1 m - \ell_2 n \equiv 0 \, (\text{mod } q)} \frac{\lambda_{\pi}(m)\bar{\lambda}_{\pi}(n)}{\sqrt{mn}} W\left(\frac{m}{q}\right) \bar W\left(\frac{n}{q}\right).
\end{displaymath}
We single out the term $\ell_1 m - \ell_2 n=0$ which essentially implies $\ell_1 =
\ell_2$, $m = n$ and hence contributes
\begin{equation}\label{diagcontr}
 \ll qL.
\end{equation}
We write the off-diagonal contribution of the inner sum as
\begin{equation}\label{beforeroughly}
 \sum_{h \sim L} \sum_{ \ell_1 m - \ell_2 n =qh} \frac{\lambda_{\pi}(m)\bar{\lambda}_{\pi}(n)}{\sqrt{mn}} W\left(\frac{m}{q}\right) \bar W\left(\frac{n}{q}\right).
\end{equation}
By the surjectivity of the Kirillov map, we can find a vector $\phi \in V_{\pi}$ such that the inner sum is the horocycle integral
\begin{displaymath}
 \int_{0}^1 \underbrace{\left(R_{\ell_1} \phi R_{\ell_2} \bar{\phi}\right)}_{=:\Phi}
 \left(\begin{pmatrix} (qL)^{-1} & x\\ 0 & 1\end{pmatrix}\right) e(-qhx)
 \,dx,
\end{displaymath}
where $R_{\ell}$ is the shift operator \eqref{defshift}. We
decompose the form $\Phi$ (which is of level $\sim L^2$) spectrally
(ignoring the continuous spectrum) as $\Phi = \sum_j \Phi_j$, so
that we can recast \eqref{beforeroughly} roughly as
\begin{equation}\label{roughly}
\sum_{h \sim L} \sum_{j} \frac{\lambda_j(qh)}{\sqrt{qh}}
W_{\Phi_j}\left(\frac{h}{L}\right)
\end{equation}
with the notation \eqref{Fourgen}--\eqref{oldKir}. In particular,
$W_{\Phi_j}$ is a multiple of the Whittaker function and therefore
decays rapidly in the spectral parameter $\lambda_j$. By Plancherel
and the fact that the Kirillov map is almost an isometry (see
\eqref{newWphibound3}), we have $\sum_j \| W_{\Phi_j} \|^2 \sim \|
\Phi \| \ll 1$, since the operators $R_{\ell}$ are isometries. By
Weyl's law, there are about $L^2$ eigenvalues in an interval of
constant length, so the $j$-sum has effectively about $L^2$ terms,
and hence each $W_{\Phi_j}(h/L)\sim W_{\Phi_j}(1)$ should be of size
$1/L$. At this point we can already sum trivially to get an
off-diagonal contribution of
\begin{equation}\label{upper1}
 q \underbrace{L^2}_{\text{amplifier}} \underbrace{L}_{\text{$h$-sum}} \underbrace{L^2}_{\text{$j$-sum}} \frac{1}{\sqrt{qL}} \frac{1}{L} = q^{1/2}
 L^{7/2}.
\end{equation}
Combining \eqref{lowerbound}, \eqref{diagcontr} and \eqref{upper1}
gives $L(1/2,\pi \otimes \chi) \ll q^{2/5}$ upon choosing $L =
q^{1/5}$.

However, we can do better by exploiting cancellation in the double
sum over $j$ and $h$. One way to see this is to recognize that the
$h$-sum mimics the central value $L(1/2, \pi_j)$ (the length is $L$
and the conductor is $L^2$), and on average over $j$ we should be
able to prove Lindel\"of, that is, on average we should have
$\sum_{h \sim L} \lambda_j(h)h^{-1/2} \sim 1$ rather than $L^{1/2}$.
This can be made precise as follows: by Cauchy--Schwarz,
\eqref{roughly} is bounded by
\begin{equation}\label{boundedby}
 \frac{1}{\sqrt{q}L} \left(\sum_{\lambda_j \sim 1} 1\right)^{1/2} \left( \sum_{h_1, h_2 \sim L} \frac{1}{\sqrt{h_1h_2}} \sum_{\lambda_j \sim 1}\lambda_j(h_1) \bar{\lambda}_j(h_2)\right)^{1/2}.
\end{equation}
The Kuznetsov formula translates the innermost sum into
\begin{displaymath}
 L^2\left(\delta_{h_1, h_2} + \sum_{L^2 \mid c} \frac{1}{c}S(h_1, h_2, c)
 f\left(\frac{h_1h_2}{c^2}\right)\right),
\end{displaymath}
where $S(h_1, h_2, c) \ll c^{1/2}$ and $f(x) \ll \min(1, x^{1/2})$, so that \eqref{boundedby} is by
Weil's bound and by trivial estimates $\ll L q^{-1/2}$, and hence the complete off-diagonal term is $\ll q^{1/2}L^3$. This yields
 $L(1/2,\pi \otimes \chi) \ll q^{3/8}$ upon choosing $L= q^{1/4}$.

\subsection{Shifted convolution sums}\label{section3}

In this section we will appeal to the spectral decomposition
\eqref{spectral} for trivial $\omega$. We will work with the
subspace
$L^2(\GL_2(K)\backslash\GL_2(\mathbb{A})/T\KK(\mc),\text{triv})$
which is also a component of the subspace $\FS$ according to
\eqref{FSspace}. To simplify notation, we drop the subscripts
$\omega$ in $\mathcal{C}_{\omega}(\mc)$, $\mathcal{E}_{\omega}(\mc)$
etc., and we use the abbreviations (cf.\ \eqref{defCE},
\eqref{fourierwhittaker}, \eqref{contsignature})
\[\int_{(\mc)} f_{\varpi} \,d\varpi:=\sum_{\substack{\pi \in \mathcal{C}(\mc)\\\eps_\pi=1}} f_{\pi} +
\int\limits_{\substack{\varpi\in\mathcal{E}(\mc)\\\eps_{\varpi}=1}}
f_{\varpi} \,d\varpi\] for any quantity $f$ indexed by irreducible
automorphic representations. The aim of this section is to prove the
following central result.

\begin{theorem}\label{theorem4}
Let $\pi_1$, $\pi_2$ be two irreducible cuspidal representations of
$\GL_2(K)\backslash \GL_2(\mathbb{A})$ with the same unitary central
character and signature character. Let $\ell_1,\ell_2\in\mo$ be
nonzero integers and write $\mc := \lcm(\ell_1\mc_{\pi_1},
\ell_2\mc_{\pi_2})$. Let $a, b, c \in \NN_0$, and let $W_1,
W_2:K_\infty^\times\to\CC$ be arbitrary functions such that $\|
W_{1,2}\|_{A^{\mu}}$ given by \eqref{defnorm} exist for $\mu:
=2d(8+a+b+2c)$. Let $P\in\CC[x_1,\dots,x_d]$ be a polynomial of
degree at most $a$ in each variable, and consider the differential
operator
$\mathcal{D}:=P(y_1\partial_{y_1},\dots,y_d\partial_{y_d})$. Then
for any $\varpi \in \mathcal{C}(\mc)\cup \mathcal{E}(\mc)$ with
$\eps_\varpi=1$ and for any $\mt \mid \mc\mc_{\varpi}^{-1}$ there
exists a function $W_{\varpi, \mt} : K_\infty^\times \to \CC$
depending only on $\pi_{1, 2}$, $W_{1,2}$, $\varpi$, $\mt$, $K$ such
that the following two properties hold.

$\bullet$ For $Y\in(0,\infty)^d$, an ideal $\my\subseteq\mo$ and a
nonzero $q \in\my$ there is a spectral decomposition
\begin{equation}\label{decompos}
\begin{split}
 \sum_{\substack{\ell_1r_1 - \ell_2r_2 = q\\ 0\neq r_{1, 2}\in \my }}
 & \frac{\lambda_{\pi_1}(r_1\my^{-1}) \bar{\lambda}_{\pi_2}(r_2 \my^{-1})}{\sqrt{\N (r_1r_2\my^{-2})}}
 W_1\left(\frac{(\ell_1 r_1)^{\sigma_1}}{Y_1},\ldots,\frac{(\ell_1
 r_1)^{\sigma_d}}{Y_d}\right)
 \bar W_2\left(\frac{(\ell_2 r_2)^{\sigma_1}}{Y_1},\ldots,\frac{(\ell_2
 r_2)^{\sigma_d}}{Y_d}\right)\\
 & = \int_{(\mc)} \sum_{\mt \mid \mc\mc_{\varpi}^{-1}}
 \frac{\lambda_{\varpi}^{(\mt)}(q\my^{-1}) }{\sqrt{\N (q\my^{-1})}}W_{\varpi, \mt}\left(\frac{q^{\sigma_1}}{Y_1}, \ldots, \frac{q^{\sigma_d}}{Y_d}\right)\,d\varpi,
\end{split}
\end{equation}
where $\lambda_{\varpi}^{(\mt)}(\mathfrak{m})$ is given by
\eqref{Fourgen} and \eqref{fouriergeneralEisen2}.

$\bullet$ For $y \in K_\infty^\times$, $0<\eps<1/4$, and $\theta$ as
in \eqref{boundLaplace}, there is a bound
\begin{displaymath}
\int_{(\mc)} \sum_{\mt \mid \mc\mc_{\varpi}^{-1}}
(\N\tilde{\lambda}_{\varpi})^{2c} \left|\mathcal{D}W_{\varpi,
\mt}(y)\right|^2 \,d\varpi
 \ll |\N(\ell_1\ell_2)|^{\eps} \|
W_1 \|_{A^{\mu}}^2\| W_2 \|_{A^{\mu}}^2 \prod_{j=1}^d
|y_j|^{1-2\theta -\eps}\min(1, |y_j|^{-2b})
\end{displaymath}
with an implied constant depending only on $\pi_{1,2}$, $a$, $b$,
$c$, $P$, $K$, $\eps$.
\end{theorem}

\begin{remark}
For $q \not\in \my$ the left hand side of \eqref{decompos} vanishes
trivially. The assumptions on $\pi_{1,2}$ only serve notational
convenience, and with a little more work one can show that the
implied constant depends polynomially on $C(\pi_1)C(\pi_2)$.
\end{remark}

\begin{remark}\label{L1bound}
One can combine the $L^2$-bound in Theorem~\ref{theorem4} for $c+1$
in place of $c$ with Cauchy--Schwarz and Lemma~\ref{weyl} (resp.
Lemma~\ref{babyWeyl}) to deduce an $L^1$-bound for the cuspidal
(resp. continuous) spectrum. For $\kappa:=2d(10+a+b+2c)$ one obtains
\[\int\limits_{\substack{\varpi\in\mathcal{C}(\mc)\\\eps_{\varpi}=1}}
\sum_{\mt \mid \mc\mc_{\varpi}^{-1}} (\N\tilde{\lambda}_{\varpi})^c
\left|\mathcal{D}W_{\varpi, \mt}(y)\right| \,d\varpi \ll
|\N(\ell_1\ell_2)|^{\frac{1}{2}+\eps}\| W_1 \|_{A^{\kappa}}\| W_2
\|_{A^{\kappa}} \prod_{j=1}^d |y_j|^{\frac{1}{2}-\theta
-\eps}\min(1,|y_j|^{-b})\] and, denoting by
$\mathfrak{l}\subseteq\mo$ the largest \emph{square} divisor of
$\lcm((\ell_1),(\ell_2))$,
\[\int\limits_{\substack{\varpi\in\mathcal{E}(\mc)\\\eps_{\varpi}=1}}
\sum_{\mt \mid \mc\mc_{\varpi}^{-1}} (\N\tilde{\lambda}_{\varpi})^c
\left|\mathcal{D}W_{\varpi, \mt}(y)\right| \,d\varpi \ll
(\N\mathfrak{l})^{\frac{1}{4}}|\N(\ell_1\ell_2)|^{\eps}\| W_1
\|_{A^{\kappa}}\| W_2 \|_{A^{\kappa}} \prod_{j=1}^d
|y_j|^{\frac{1}{2}-\theta -\eps}\min(1, |y_j|^{-b}),\] with implied
constants depending only on $\pi_{1,2}$, $a$, $b$, $c$, $P$, $K$,
$\eps$.
\end{remark}

\begin{proof}
By the surjectivity of the Kirillov map (see the remark after
\eqref{constant}) we can choose $\phi_i \in V_{\pi_i}(\mc_{\pi_i})$
for $i =1, 2$, such that $W_{\phi_i}=W_i$. Let
\begin{displaymath}
 \Phi := (R_{(\ell_1)}\phi_1)(R_{(\ell_2)} \bar\phi_2)
\end{displaymath}
with the notation as in \eqref{defshift}. Then $\Phi \in
L^2(\GL_2(K)\backslash\GL_2(\mathbb{A})/T\KK(\mc), \text{triv})$
with $\mc$ as in the theorem. Let $y \in \mathbb{A}^{\times}$ be
such that $y_{\infty} = (Y_1, \ldots, Y_d)$ and $(y_{\text{fin}}) =
\my$. By \eqref{defshift} and \eqref{fouriersimple} we have
\begin{displaymath}
\begin{split}
 \sum_{\substack{\ell_1r_1 - \ell_2r_2 = q\\ 0 \neq r_{1, 2}\in \my}}
 &\frac{\lambda_{\pi_1}(r_1\my^{-1}) \bar{\lambda}_{\pi_2}(r_2 \my^{-1})}{\sqrt{\N(r_1r_2\my^{-2})}}
 W_1\left(\frac{(\ell_1 r_1)^{\sigma_1}}{Y_1},\ldots,\frac{(\ell_1
 r_1)^{\sigma_d}}{Y_d}\right)
 \bar W_2\left(\frac{(\ell_2 r_2)^{\sigma_1}}{Y_1},\ldots,\frac{(\ell_2
 r_2)^{\sigma_d}}{Y_d}\right)\\
 &= \int_{K\backslash\mathbb{A}}\Phi\left(\begin{pmatrix} y^{-1} & x\\ 0&
1\end{pmatrix}\right) \psi(-qx) \,dx.
\end{split}
\end{displaymath}
By \eqref{spectral}, \eqref{RTdecomp2}, \eqref{RTdecomp3}, we have
an orthogonal decomposition
\[\Phi = \Phi_\text{sp}+\int_{(\mc)}\sum_{\mt\mid\mc\mc_\varpi^{-1}}\Phi_{\varpi,\mt}\,d\varpi,
\qquad \Phi_\text{sp}\in L_\text{sp},\quad \Phi_{\varpi,\mt}\in
R^{(\mt)}V_\varpi(\mc_\varpi),\] where $\Phi_\text{sp}$ is the
projection of $\Phi$ on the subspace generated by the functions
$g\mapsto\chi(\det g)$ with any quadratic Hecke character $\chi$ as
discussed in Section~\ref{spectralsection}. As $q$ is nonzero,
\eqref{decompos} is immediate from \eqref{Fourgen} and
\eqref{fouriergeneralEisen2} upon defining
$W_{\varpi,\mt}:=W_{\Phi_{\varpi,\mt}}$.

We proceed to establish the upper bound stated in the theorem. By
Lemma~\ref{lemma3} and \eqref{decompnorm} we have
\[\int_{(\mc)}
\sum_{\mt \mid \mc\mc_{\varpi}^{-1}}
 (\N\tilde{\lambda}_{\varpi})^{2c} \left|\mathcal{D}W_{\varpi, \mt}(y)\right|^2 \,d\varpi \ll
 (\N\mc)^{\eps}
\| \Phi \|_{S^{\alpha}}^2 \prod_{j=1}^d |y_j|^{1-2\theta
-\eps}\min(1, |y_j|^{- b}),\] where $\alpha:=d(5+a+b+2c)$. It
remains to show that
\begin{equation}\label{stilltoshow}
 \| \Phi \|_{S^{\alpha}} \ll \| W_1 \|_{A^{\mu}} \|W_2\|_{A^{\mu}}
\end{equation}
for $\mu$ as in the theorem. By Lemma~\ref{lemma4} any $\mathcal{D}
\in U(\fg)$ of order at most $\alpha$ satisfies
\begin{displaymath}
 \| \mathcal{D} R_{(\ell_i)} \phi_i \|_{\infty}
 = \| R_{(\ell_i)} \mathcal{D} \phi_i \|_{\infty}
 = \|\mathcal{D} \phi_i \|_{\infty} \ll_{\pi_i, K}
 \| \phi_i \|_{S^{\alpha+3d}},
 \end{displaymath}
therefore the Leibniz rule for derivations immediately shows
\begin{displaymath}
\begin{split}
 \| \Phi \|_{S^{\alpha}}
\ll_{\pi_1, \pi_2, K}
\|\phi_1\|_{S^{\alpha+3d}}\|\phi_2\|_{S^{\alpha+3d}}.
\end{split}
\end{displaymath}
An application of Lemma~\ref{lemma2} and \eqref{akshaynorm} now
yields \eqref{stilltoshow} and completes the proof of
Theorem~\ref{theorem4}.
\end{proof}

\subsection{A Burgess-like subconvex bound for twisted $L$-functions}

In this section we prove Theorem~\ref{theorem1}, borrowing several
important ideas from \cite{CPSS,Co}. For simplicity, we shall in general not
indicate the dependence of implied constants on $\pi$,
$\chi_\infty$, $K$. We regard $\chi$ as a Gr\"ossencharacter, i.e.\ a certain character of
the group of fractional ideals coprime to $\mq$. We extend $\chi$ to the group of all fractional ideals
by defining it to be zero for fractional ideals not coprime to $\mq$. There exists a pair of characters
$\chi_\text{fin}:(\mo/\mq)^{\times}\to S^1$ and $\chi_\infty:K_\infty^\times\to S^1$ such that
$\chi((r))=\chi_\text{fin}(r)\chi_\infty(r)$ for $r\in\mo$ coprime to $\mq$. We lift any character $\xi$
of $(\mo/\mq)^{\times}$ to a function $\xi:\mo\to\CC$ by defining $\xi(r)=\xi(r\ \text{mod}\ \mq)$ for $r\in\mo$ coprime to $\mq$ and $\xi(r)=0$ elsewhere.

Our starting point is the approximate functional
equation in the user-friendly version \eqref{approx}. We cut the sum
into (finitely many) pieces according to the narrow ideal class of
the ideal $\mathfrak{m}$. We fix a narrow ideal class and a
representative $\my$ coprime to $\mq$; we can assume $\N\my
\ll_{\eps} (\N\mq)^{\eps}$. Then it is enough to bound
\begin{equation}\label{thm11}
 \sum_{\substack{0 << r \in \my\\ r \text{ mod } U^{+}}} \frac{\lambda_{\pi}(r\my^{-1})\chi(r\my^{-1})}{\sqrt{\N(r \my^{-1})}} V\left(\frac{\N r}{Y}\right)
\end{equation}
for $Y \ll_{\eps} (\N\mq)^{1+\eps}$ and a smooth function
$V:(0,\infty)\to\CC$ supported on $[\frac{1}{2}, 2]$ such that
$V^{(j)}(y) \ll_{j} 1$ for all $j \in \NN_0$. Let us fix (once and
for all) a fundamental domain $\mathcal{F}_0$ for the action of
$U^{+}$ on the hyperboloid $\{y\in\Kplus\mid\N y=1\}$ such that its
image under the map $\Kplus\to\RR^d$, $y \mapsto (\log y^{\sigma_1},
\ldots, \log y^{\sigma_d})$, is a fundamental parallelotope of the
image of $U^{+}$ under the same map\footnote{The image of $U^{+}$ is
a lattice in the hyperplane of $K_\infty$ orthogonal to
$(1,...,1)$.}. The cone $\mathcal{F}:=\Kplusdiag\mathcal{F}_0$ is a
fundamental domain for the action of $U^{+}$ on $\Kplus$. We
introduce the following smooth variants of $\mathcal{F}_0$ and
$\mathcal{F}$: we fix a smooth and compactly supported function
$F_0:\{y\in\Kplus\mid\N y=1\}\to\CC$ such that $\sum_{u \in
U^{+}}F_0(u y) = 1$ for any $y\in\Kplus$ of norm $1$, and we extend
this to all of $\Kplus$ by $F(y) := F_0(y/(\N y)^{1/d})$. Note that
the support of $F_0$ is contained in some box
$[c_1,c_2]^d\subseteq\Kplus$, and the support of $F$ is contained in
the cone $\Kplusdiag[c_1,c_2]^d$ of this box. We can rewrite
\eqref{thm11} as
\[\sum_{0 << r \in \my}
\frac{\lambda_{\pi}(r\my^{-1})\chi(r\my^{-1})}{\sqrt{\N(r
\my^{-1})}}F(r)V\left(\frac{\N r}{Y}\right)\] which is really a
finite sum, because $\my$ is a lattice in $K_\infty$ and the terms
vanish outside the box $[\frac{1}{2}c_1Y^{1/d},2c_2Y^{1/d}]^d$. Let
us fix a smooth function $W:\Kplus\to\CC$ supported on
$[\frac{1}{3}c_1,3c_2]^d$ such that $W(y)=1$ on $[\frac{1}{2}c_1,
2c_2]^d$, then we can recast \eqref{thm11} as
\begin{equation}\label{recast}
\begin{split}
& \bar{\chi}(\my) \sum_{0 << r \in \my} \frac{\lambda_{\pi}(r\my^{-1})\chi(r)}{\sqrt{\N(r\my^{-1})}} F(r)V\left(\frac{\N r}{Y}\right) W\left(\frac{r}{Y^{1/d}}\right)\\
 = &\frac{ \bar{\chi}(\my)\chi_\infty(Y^{1/d})}{(2\pi i)^d} \int_{(i\RR)^d} \check{V}(v) \sum_{0 << r \in \my} \frac{\lambda_{\pi}(r\my^{-1})\chi_\text{fin}(r)}{\sqrt{\N(r\my^{-1})}}
 W_v\left(\frac{r}{Y^{1/d}}\right)\,dv,
\end{split}
\end{equation}
where $v:=(v_1,\dots,v_d)\in(i\RR)^d$ and
\begin{displaymath}
 \check{V}(v): = \int_{\Kplus} F(y) V(\N y)\chi_\infty(y)\prod_{j=1}^d y_j^{v_j}\,d^{\times}y, \qquad W_{v}(y) := W(y)\prod_{j=1}^d y_j^{-v_j}.
\end{displaymath}
At this point it is worthwhile to extend the notational convention
in \eqref{short} to all complex vectors $z\in\CC^d$ as follows:
\begin{equation}\label{tildenotation}
\tilde z:=(1+|z_j|)_{j=1}^d\in\RR_{>0}^d.
\end{equation}
The functions $F(y)V(\N y)$ and $W(y)$ are smooth of compact
support and $\chi_\infty(y)=\prod_{j=1}^dy_j^{s_j}$ for some fixed $s\in(i\RR)^d$, therefore we have the bounds
\begin{align}
 \check{V}(v) &\ll_{A,\chi_\infty}(\N\tilde v)^{-A}, \qquad A>0,\label{testbounds}\\
 \partial_{y_1}^{\mu_ 1} \cdots \partial_{y_d}^{\mu_d}W_{v}(y) &\ll_{\mu}
 \prod_{j=1}^d(1+|v_j|)^{\mu_j},\qquad
 \mu\in\NN_0^d.\label{testbounds2}
\end{align}

We fix $v\in (i\RR)^d$ and postpone the integration over $v$ to the
very end of the argument. For a character $\xi$ of
$(\mo/\mq)^{\times}$ we define
\begin{displaymath}
 \mathcal{L}_{\xi}(v) := \sum_{0 << r \in \my} \frac{\lambda_{\pi}(r\my^{-1})\xi(r)}{\sqrt{\N(r\my^{-1})}} W_v\left(\frac{r}{Y^{1/d}}\right),
\end{displaymath}
so that $\mathcal{L}_{\chi_\text{fin}}(v)$ is the sum on the right hand side of \eqref{recast}.
Observe that the sum is supported in the box
$[\frac{1}{3}c_1Y^{1/d},3c_2Y^{1/d}]^d$ whose cone
$\mathcal{C}\subseteq\Kplus$ is independent of $Y$ and can be
covered by finitely many $U^{+}$-translates of $\mathcal{F}$. We
consider an amplified second moment and choose a parameter $L$
satisfying $\log L \asymp \log(\N\mq)$. It is not hard to see that
\[\#\{\text{$\mathfrak{l} \subseteq \mo$ is a totally positive principal prime ideal}\mid
N\mathfrak{l}\in[L,2L],\ \mathfrak{l} \nmid \mq \} \gg_{\eps}
L(N\mq)^{-\eps},\] hence by positivity
\begin{equation*}
 |\mathcal{L}_{\chi_\text{fin}}(v)|^2 \ll_\eps \frac{(N\mq)^{\eps}}{L^2}
 \sum_{\xi \in \widehat{(\mo/\mq)^{\times}}}\,\Biggl|\mathcal{L}_{\xi}(v)\sum_{\substack{\ell \in \mo \cap\mathcal{F}\\ \N\ell\in[L,2L]\\(\ell) \text{ prime},\ (\ell) \nmid \mq}} \xi(\ell)\bar\chi_\text{fin}(\ell)\Biggr|^2.
\end{equation*}
By Plancherel's formula for $(\mo/\mq)^{\times}$ this is the same as
\begin{equation*}
|\mathcal{L}_{\chi_\text{fin}}(v)|^2 \ll_\eps\
\frac{\varphi(\mq)(\N\mq)^{\eps}}{L^{2}}\sum_{x\in(\mo/\mq)^{\times}}\,
\Biggl|\sum_{\substack{\ell \in \mo \cap\mathcal{F}\\ \N\ell\in[L,2L]\\(\ell) \text{ prime},\ (\ell) \nmid \mq}} \bar\chi_\text{fin}(\ell)\sum_{\substack{r \in
\my\cap\mathcal{C}\\\ell r\equiv x\, (\text{mod }\mq)}}\frac{\lambda_{\pi}(r\my^{-1})}{\sqrt{\N(r\my^{-1})}} W_v\left(\frac{r}{Y^{1/d}}\right)\Biggr|^2.
\end{equation*}
We can extend the summation over all $x\in\mo/\mq$ by positivity, then after opening the square we get
\begin{equation}\label{needtobound}
\begin{split}
|\mathcal{L}_{\chi_\text{fin}}(v)|^2 \ll_\eps\
&\frac{(\N\mq)^{1+\eps}}{L^{2}} \sum_{\substack{\ell_1,
\ell_2 \in \mo \cap\mathcal{F}\\ \N\ell_1, \N\ell_2\in[L,2L]\\
(\ell_1), (\ell_2) \text{ primes}\\ (\ell_1),(\ell_2) \nmid\mq}}\bar{\chi}(\ell_1)\chi(\ell_2)\\
&\times\sum_{\substack{\ell_1r_1 - \ell_2r_2 \in \mq\\ r_1, r_2 \in
\my\cap\mathcal{C}}}\frac{\lambda_{\pi}(r_1\my^{-1})\bar{\lambda}_{\pi}(r_2\my^{-1})}{\sqrt{\N(
r_1r_2\my^{-2})}} W_v\left(\frac{r_1}{Y^{1/d}}\right)\bar
W_v\left(\frac{r_2}{Y^{1/d}}\right).
\end{split}
\end{equation}
We single out the diagonal term $\ell_1 r_1 -\ell_2 r_2 = 0$ which
contributes at most
\begin{displaymath}
 \ll_\eps \frac{(\N\mq)^{1+\eps}}{L^{2}}\sum_{\substack{\ell \in \mo \cap\mathcal{F}\\ \N\ell \asymp L}}
 \sum_{\substack{r \in \my\cap\mathcal{C}\\\N r\asymp Y}} \frac{|\lambda_{\pi}(r\my^{-1})|^2}{\N(r\my^{-1})}
 \#\{(\ell', r') \in (\mo\cap\mathcal{F})\times (\my\cap\mathcal{C}) \mid \ell'r' = \ell r\},
\end{displaymath}
uniformly in $v\in(i\RR)^d$. The last factor is bounded
by $\ll_{\eps}(LY)^{\eps}$, and so by \eqref{squarebound} the
preceding display is at most
\begin{equation}\label{diag1}
 \ll_{\eps}\frac{(\N\mq)^{1+\eps}}{L^{2}} \#\{\mathfrak{l} \subseteq \mo \mid \N\mathfrak{l} \asymp L\}
 \sum_{\N \mathfrak{m} \ll Y} \frac{|\lambda_{\pi}(\mathfrak{m})|^2}{\N\mathfrak{m}} \ll_{\eps} (\N \mq)^{\eps}
 \frac{\N\mq}{L}.
\end{equation}

Let us now consider the off-diagonal contribution in
\eqref{needtobound}. If $[c_3,c_4]^d\subseteq\Kplus$ is a box
containing $\mathcal{F}_0$, then in \eqref{needtobound} the
variables satisfy $\ell_{1,2}\in[c_3L^{1/d},2c_4L^{1/d}]^d$ and
$r_{1,2}\in[\frac{1}{3}c_1Y^{1/d},3c_2Y^{1/d}]^d$, so that
\[\ell_1r_1-\ell_2r_2\in\mathcal{B}:=[-6c_2c_4(LY)^{1/d},6c_2c_4(LY)^{1/d}]^d.\]
We fix $\ell_{1,2}$ for the moment and we rewrite the off-diagonal
part of the inner sum in \eqref{needtobound} as
\begin{equation}\label{secondline}
 \sum_{0\neq q \in\mq\my\cap\mathcal{B}}\ \sum_{\substack{\ell_1r_1 - \ell_2r_2 =q \\ 0\neq r_1, r_2 \in \my}}
 \frac{\lambda_{\pi}(r_1\my^{-1})\bar{\lambda}_{\pi}(r_2\my^{-1})}{\sqrt{\N (r_1r_2\my^{-2})}}
 W_1\left(\frac{\ell_1r_1}{(LY)^{1/d}};v\right)\bar W_2\left(\frac{\ell_2r_2}{(LY)^{1/d}};v\right),
\end{equation}
where $W_i(\cdot;v):K_\infty^\times\to\CC$ for $i=1,2$ are smooth
functions defined by
\[W_i(y;v):=\begin{cases}
W_v(\ell_i^{-1}L^{1/d}y),&y\in\Kplus,\\
0,&\text{otherwise}.
\end{cases}\]
Note that $W_i(\cdot;v)$ is supported on
$[\frac{1}{3}c_1c_3,6c_2c_4]^d$ and by \eqref{testbounds2} it
satisfies
\begin{equation}\label{partialbound}\partial_{y_1}^{\mu_ 1} \cdots
\partial_{y_d}^{\mu_d}W_i(y;v) \ll_{\mu}
\prod_{j=1}^d(1+|v_j|)^{\mu_j},\qquad \mu\in\NN_0^d.\end{equation}
Using Theorem~\ref{theorem4}, we rewrite \eqref{secondline} as
\begin{equation}\label{spectralsum}
\sum_{0\neq q \in\mq\my\cap\mathcal{B}}\ \int_{(\mc)} \sum_{\mt\mid
\mc\mc_{\varpi}^{-1}} \frac{\lambda_{\varpi}^{(\mt )}(q\my^{-1})
}{\sqrt{\N(q\my^{-1})}} W_{\varpi,
\mt}\left(\frac{q}{(LY)^{1/d}};v\right)\,d\varpi,
\end{equation}
where $\mc := \mc_\pi\lcm((\ell_1),(\ell_2))$. At this point we can
already estimate the Eisenstein contribution trivially. On the one
hand, we can combine the second bound in Remark~\ref{L1bound} with
\eqref{defnorm} and \eqref{partialbound} to see that (cf.\
\eqref{tildenotation})
\[\int\limits_{\substack{\varpi\in\mathcal{E}(\mc)\\\eps_{\varpi}=1}} \sum_{\mt \mid \mc\mc_{\varpi}^{-1}}
\left|W_{\varpi, \mt }(y; v)\right| \,d\varpi \ll_{\eps}
(\N(\ell_1\ell_2))^{\eps} (\N\tilde v)^{44d},\] uniformly in $y \in
K_\infty^\times$, $v\in (i\RR)^d$, $\ell_{1, 2}$ and $\my$. On the
other hand, by \eqref{boundFourierEisen3} we have the uniform bound
\begin{displaymath}
 \lambda_{\varpi}^{(\mt )}(q\my^{-1}) \ll_{\eps} (\N\gcd(\mc, (q)))
 (\N(q))^{\eps},
\end{displaymath}
hence the Eisenstein contribution in \eqref{spectralsum} is at most
\[\ll_{\eps}(\N\tilde v)^{44d}(\N\mq)^\eps
\sum_{0\neq q\in\mq\cap\mathcal{B}}\frac{\N\gcd(\mc,
(q))}{\sqrt{\N(q)}}.\] In the last sum each principal ideal $(q)$
has multiplicity $\ll(\log(\N\mq))^{d-1}$. Indeed, any nonzero
principal ideal in $\mo$ has a generator $q$ such that
$|q^{\sigma_j}|\geq c_3$ for $1\leq j\leq d$, so the multiplicity in
question is at most the number of units $u\in U$ in the cube
$[-c_5(LY)^{1/d},c_5(LY)^{1/d}]^d$ for $c_5:=6c_2c_4/c_3$ which is
$\ll(\log(\N\mq))^{d-1}$ by Dirichlet's unit theorem (or its proof).
At any rate, the last sum is
\[\ll_{\eps}(\N\mq)^\eps
\sum_{\substack{(q) \subseteq \mq\\ \N (q) \ll LY}}
\frac{\N\gcd(\mc, (q))}{\sqrt{\N(q)}} \ll_{\eps}
(\N\mq)^{-1+2\eps}(LY)^{1/2}\ll_{\eps}
(\N\mq)^{-1/2+3\eps}L^{1/2},\] hence the Eisenstein contribution in
\eqref{spectralsum} is at most
\begin{equation}\label{contribution1}
\ll_\eps (\N\tilde v)^{44d}(\N\mq)^{-1/2+\eps}L^{1/2}.
\end{equation}

Let us now turn to the cuspidal contribution in \eqref{spectralsum}.
Choosing $a=0$, $b=1$, $c$ very large in Theorem~\ref{theorem4} and
combining the inequality there with \eqref{defnorm} and
\eqref{partialbound}, Cauchy--Schwarz and Lemma~\ref{weyl}, we see
for any $\eps>0$ that the contribution of all
$\varpi\in\mathcal{C}(\mc)$ with $\tilde{\lambda}_{\varpi, j} \geq
(\N \mq)^{\eps}$ for some $j$ is negligible. Let us introduce the
notation
\[\mathcal{C}(\mc,\eps):=\{\varpi\in\mathcal{C}(\mc)
\mid\text{$\tilde{\lambda}_{\varpi, j} \leq (\N \mq)^{\eps}$ for $1
\leq j \leq d$}\},\qquad \eps>0,\] and
\[\mathcal{B}(\xi):=\{y\in\mathcal{B}\mid\sgn(y)=\xi\},\qquad \xi \in \{\pm 1\}^d.\]
Then it suffices to bound, for fixed primes $\ell_1,\ell_2>>0$ and
$\eps>0$, $\xi\in \{\pm 1\}^d$ the quantity
\begin{equation}\label{quantity}\sum_{q \in\mq\my\cap\mathcal{B}(\xi)}
\sum_{\substack{\varpi\in\mathcal{C}(\mc,\eps)\\\mt\mid
\mc\mc_{\varpi}^{-1}}} \frac{\lambda_{\varpi}^{(\mt
)}(q\my^{-1})}{\sqrt{\N(q\my^{-1})}} W_{\varpi,
\mt}\left(\frac{q}{(LY)^{1/d}};v\right).\end{equation} We separate
the variables $\varpi$ and $q$ by Mellin inversion. For $s \in
\CC^d$ with $\Re s_j \geq -1/4$ we write
\begin{displaymath}
 \widehat{W}^{(\xi)}_{\varpi, \mt }(s; v) := \int_{\Kplus} W_{\varpi, \mt }(\xi y; v) \,\prod_{j=1}^d y_j^{s_j}
 \,d^{\times}y
\end{displaymath}
and recast \eqref{quantity} as
\begin{displaymath}
\begin{split}
 & \left(\frac{1}{2\pi i}\right)^d \int_{(i\RR)^d} (LY)^{(s_1+\cdots+s_d)/d}
 \sum_{\substack{\varpi\in\mathcal{C}(\mc,\eps)\\\mt\mid
\mc\mc_{\varpi}^{-1}}} \widehat{W}^{(\xi)}_{\varpi, \mt }(s;v)
 \sum_{q \in\mq\my\cap\mathcal{B}(\xi)}\frac{\lambda_{\varpi}^{(\mt
)}(q\my^{-1})}{\sqrt{\N(q\my^{-1})}}
 \prod_{j=1}^d |q^{\sigma_j}|^{-s_j} \,ds\\
& \ll \int_{(i\RR)^d}
\Biggl(\sum_{\substack{\varpi\in\mathcal{C}(\mc,\eps)\\\mt\mid
\mc\mc_{\varpi}^{-1}}} \bigl|\widehat{W}^{(\xi)}_{\varpi, \mt
}(s;v)\bigr|^2\Biggr)^{1/2} \Biggl(
\sum_{\substack{\varpi\in\mathcal{C}(\mc,\eps)\\\mt\mid
\mc\mc_{\varpi}^{-1}}} \Biggl| \sum_{q
\in\mq\my\cap\mathcal{B}(\xi)} \frac{\lambda_{\varpi}^{(\mt
)}(q\my^{-1})}{\sqrt{\N(q\my^{-1})}}
 \prod_{j=1}^d |q^{\sigma_j}|^{-s_j} \Biggr|^2\Biggr)^{1/2}\,|ds|.
\end{split}
\end{displaymath}
Using the differential operator
$\mathcal{D}:=\prod_{j=1}^d(1+y_j\partial_{y_j})^3$ the first sum
is, for any $s\in(i\RR)^d$,
\begin{align*}&\ll
(\N\tilde s)^{-3}\int_{K_\infty^\times}\int_{K_\infty^\times}
\sum_{\substack{\varpi\in\mathcal{C}(\mc,\eps)\\\mt\mid
\mc\mc_{\varpi}^{-1}}}\left|\mathcal{D}W_{\varpi,
\mt}(y;v)\right|\left|W_{\varpi, \mt }(z;v)\right|\ d^\times
y\ d^\times z,\\
&\ll (\N\tilde s)^{-3}\int_{K_\infty^\times}\int_{K_\infty^\times}
\Biggl(\sum_{\substack{\varpi\in\mathcal{C}(\mc,\eps)\\\mt\mid
\mc\mc_{\varpi}^{-1}}}\left|\mathcal{D}W_{\varpi,
\mt}(y;v)\right|^2\Biggr)^{1/2}
\Biggl(\sum_{\substack{\varpi\in\mathcal{C}(\mc,\eps)\\\mt\mid
\mc\mc_{\varpi}^{-1}}}\left|W_{\varpi, \mt
}(z;v)\right|^2\Biggr)^{1/2}\,d^\times y\ d^\times z.
\end{align*}
We apply Theorem~\ref{theorem4} with $(a,b,c)=(3,1,0)$ and
$(a,b,c)=(0,1,0)$, then by \eqref{defnorm} and \eqref{partialbound}
the integrand is
\[\ll_{\eps}(\N q)^\eps(\N\tilde v)^{84d}\prod_{j=1}^d\min(|y_j|^{1/4},|y_j|^{-1/2})\min(|z_j|^{1/4},|z_j|^{-1/2}),\]
so that the previous display is
\[\ll_{\eps}(\N q)^\eps(\N\tilde v)^{84d}(\N\tilde s)^{-3}.\]
We infer that \eqref{quantity} is bounded by
\begin{equation}\label{leftwith}
\ll_\eps(\N q)^\eps (\N\tilde v)^{42d}\sup_{s\in(i\RR)^d}
\Biggl(\sum_{\substack{\varpi\in\mathcal{C}(\mc,\eps)\\\mt\mid
\mc\mc_{\varpi}^{-1}}} \Biggl| \sum_{q
\in\mq\my\cap\mathcal{B}(\xi)} \frac{\lambda_{\varpi}^{(\mt
)}(q\my^{-1})}{\sqrt{\N(q\my^{-1})}} \prod_{j=1}^d
|q^{\sigma_j}|^{-s_j} \Biggr|^2\Biggr)^{1/2}.
\end{equation}
Let us write for any nonzero ideal $\ma\subseteq\mo$ and any
$s\in(i\RR)^d$
\[f(\ma;s):=
\sum_{\substack{q \in\mathcal{B}(\xi)\\(q)=\ma\my}}\prod_{j=1}^d
|q^{\sigma_j}|^{-s_j},\] then similarly as in the proof of
\eqref{contribution1} we have
\begin{equation}\label{fbound}|f(\ma;s)|\leq
\#\{q\in\mathcal{B}\mid(q)=\ma\my\}\ll(\log(\N\mq))^{d-1}\ll_{\eps}(\N\mq)^\eps,\end{equation}
while the $q$-sum in \eqref{leftwith} equals the following sum over
integral ideals $\mathfrak{m}$:
\begin{displaymath}
 \sum_{\N\mathfrak{m} \ll LY/\N(\mq\my)} \frac{\lambda_{\varpi}^{(\mt )}(\mathfrak{m}\mq)}{\sqrt{\N(\mathfrak{m}\mq)}}
 f(\mathfrak{m}\mq;s).
\end{displaymath}
We now need to ``factor out" $\lambda_{\varpi}^{(\mt )}(\mq )$. This
is completely elementary, but a little tricky. First we rewrite the
previous expression as
\begin{displaymath}
 \sum_{\mq\mid\mq ' \mid \mq^{\infty}}
 \sum_{\substack{\N\mathfrak{m} \ll LY/\N(\mq'\my)\\ \gcd(\mathfrak{m},\mq) =\mo}}
 \frac{\lambda_{\varpi}^{(\mt )}(\mathfrak{m}\mq')}{\sqrt{\N(\mathfrak{m}\mq')}} f(\mathfrak{m}\mq';s).
\end{displaymath}
Using the construction of $\lambda_{\varpi}^{(\mt )}(\mathfrak{m})$
as given in \eqref{oldKir} and the preceding remarks, together with
the Hecke relation \eqref{hecke}, we proceed as in
\cite[pp.~73--74]{BHM} to show that
\begin{displaymath}
 \lambda_{\varpi}^{(\mt )}(\mathfrak{m}\mq')=
 \sum_{\mb \mid \gcd\left(\mc,\mq',\frac{\mq'}{\gcd(\mc, \mq')}\right)}
 \mu(\mb )\lambda_{\varpi}\left(\frac{\mq'}{\mb \gcd(\mc, \mq')}\right)
 \lambda_{\varpi}^{(\mt )}\left(\frac{\mathfrak{m}\gcd(\mc, \mq')}{\mb }\right).
\end{displaymath}
Since $\gcd(\mc, \mq')$ divides $\mc_{\pi}$ (where $\pi$ is the
representation whose $L$-function we want to estimate), we can bound
the $q$-sum in \eqref{leftwith} by
\begin{displaymath}
 \ll_{\eps} \sum_{\mq\mid\mq' \mid \mq^{\infty}}
 (\N \mq')^{-1/2+\theta+\eps} \sum_{\mb \mid \mc_{\pi}}
 \Biggl| \sum_{\substack{\N\mathfrak{m} \ll LY/\N(\mq'\my)\\ \gcd(\mathfrak{m}, \mq)=\mo} }
 \frac{\lambda_{\varpi}^{(\mt )}(\mathfrak{m}\mb )}{\sqrt{\N(\mathfrak{m})}}
 f(\mathfrak{m}\mq';s)\Biggr|.
\end{displaymath}
Before we substitute this back into \eqref{leftwith}, we add a
suitable positive contribution of the continuous spectrum, and with
the notation \eqref{testkZ} we majorize the characteristic function
of
\[\{\nu\in\mathcal{S}:|1/4-\nu^2|\leq(\N\mq)^\eps\}\] with $O_K(1)$
times the function
\begin{displaymath}
 k(\nu) := \prod_{j=1}^d k_Z(\nu_j), \qquad Z := (\N \mq)^{\eps/2} \geq 1.
\end{displaymath}
Using \eqref{fbound} and Lemma~\ref{lemma2} we conclude that the
$\varpi$-sum in \eqref{leftwith} is
\begin{displaymath}
\begin{split}
 &\ll_{\eps} (\N\mq)^{-1+2\theta+2\eps} \max_{\mb _1,\mb _2 \mid
 \mc_{\pi}}\
 \sum_{\N\mathfrak{m_1},\N\mathfrak{m_2}\ll LY/\N\mq }(\N(\mathfrak{m}_1\mathfrak{m}_2))^{-1/2}\\
 & \times \left|\sum_{\varpi \in \mathcal{C}(\mc)} \frac{1}{C_{\varpi}}
 \sum_{\mt \mid \mc\mc_{\varpi}^{-1}} k(\nu_{\varpi})
 \bar{\lambda}^{(\mt )}_{\varpi}(\mathfrak{m}_1\mb _1)
 \lambda^{(\mt )}_{\varpi}(\mathfrak{m}_2\mb _2) +
 \CSC\right|.
 \end{split}
\end{displaymath}
Note that $LY/\N\mq\ll_\eps (\N\mq)^\eps L$. We are now in a
position to apply the Kuznetsov formula \eqref{kuz} to the second
line of the preceding display. We proceed very similarly as in the
proof of Lemma~\ref{weyl} and estimate the right-hand side of
\eqref{kuz} trivially, using \eqref{boundstrafo} and \eqref{weil}.
The diagonal contribution is
\begin{equation*}
\ll_\eps(\N\mq)^{-1+2\theta+\eps}\sum_{\N\mathfrak{m}\ll LY/\N\mq}
(\N\mathfrak{m})^{-1}L^2\ll (\N\mq)^{-1+2\theta+2\eps}L^2,
\end{equation*}
while the off-diagonal contribution is (use \eqref{endweyl} with
$\Xi\to (\N \mq)^{d\varepsilon/2}$ and $\N\mathfrak{m}\to
(\N\mathfrak{m}_1\mathfrak{m}_2)^{1/2}$)
\begin{align*}
&\ll_\eps(\N\mq)^{-1+2\theta+\eps}\sum_{\N\mathfrak{m}_1,\N\mathfrak{m}_2\ll
LY/\N\mq }
(\N\gcd(\mathfrak{m}_1,\mathfrak{m}_2,\mc))^{1/2}(\N(\mathfrak{m}_1\mathfrak{m}_2))^{-1/4+\eps}\\
&\ll_\eps(\N\mq)^{-1+2\theta+2\eps}\left(\sum_{\N\mathfrak{m}\ll
LY/\N\mq}(\N\gcd(\mathfrak{m},\mc))^{1/4}(\N\mathfrak{m})^{-1/4}\right)^2\\
&\ll_\eps (\N\mq)^{-1+2\theta+3\eps}L^{3/2}.
\end{align*}
Going back to \eqref{quantity} and using \eqref{leftwith} it follows
that the contribution of $\varpi\in\mathcal{C}(\mc,\eps)$ in
\eqref{spectralsum} is
\begin{equation*}
\ll_\eps (\N\tilde v)^{42d}(\N\mq)^{-1/2+\theta+\eps}L.
\end{equation*}
Together with \eqref{contribution1} and the remarks preceding
\eqref{quantity} this implies that \eqref{secondline} is at most
\[\ll_\eps (\N\tilde v)^{2B}(\N\mq)^{-1/2+\theta+\eps}L,\]
where $B=B(d,\eps)>0$ is a certain constant. By summing trivially
over $\ell_{1, 2}$ in the off-diagonal part of \eqref{needtobound}
and recalling also \eqref{diag1} we infer that
\[\mathcal{L}_{\chi_\text{fin}}(v) \ll_\eps
(\N\tilde
v)^{B}(\N\mq)^\eps\left((\N\mq)^{1/2}L^{-1/2}+(\N\mq)^{1/4+\theta/2}L^{1/2}\right).\]
The right hand side is smallest when we take
$L:=(\N\mq)^{1/4-\theta/2}$, then
\[\mathcal{L}_{\chi_\text{fin}}(v) \ll_\eps (\N\tilde v)^{B}
(\N\mq)^{\frac{1}{2} -\frac{1}{8}(1-2\theta)+\eps}.\] This result in
combination with \eqref{testbounds} for $A:=2+B$ shows that
\eqref{recast} is at most
\[\ll_\eps (\N\mq)^{\frac{1}{2}
-\frac{1}{8}(1-2\theta)+\eps},\] whence by \eqref{approx}
\[L(1/2,\pi\otimes\chi)\ll_{\pi,\chi_\infty,\eps} (\N\mq)^{\frac{1}{2}
-\frac{1}{8}(1-2\theta)+\eps}.\] The proof is complete.

\subsection{Spectral decomposition of a Dirichlet series}\label{section5}

We keep notation developed in Sections~\ref{section2} and
\ref{section3}. As another application of Theorem~\ref{theorem4} we
shall prove

\begin{theorem}\label{theorem5}
Let $\pi_1$, $\pi_2$ be two irreducible cuspidal representations of
$\GL_2(K)\backslash \GL_2(\mathbb{A})$ with the same unitary central
character and signature character. Let $\ell_1,\ell_2\in\mo$ be
totally positive integers and write $\mc := \lcm(\ell_1\mc_{\pi_1},
\ell_2\mc_{\pi_2})$. Let $c,\beta\in\NN_0$ such that
$\beta>d(66+12c)$. Then for any $\varpi \in \mathcal{C}(\mc)\cup
\mathcal{E}(\mc)$ with $\eps_\varpi=1$ and for any $\mt \mid
\mc\mc_{\varpi}^{-1}$ there exists a holomorphic function
€\[F_{\varpi, \mt } : \{s \in \CC^d \mid 1/2 +\theta < \Re s_j <
3/2\} \to \CC\] depending only on $\pi_{1, 2}$, $\beta$, $\varpi$,
$\mt$, $K$ such that the following two properties hold.

$\bullet$ For an ideal $\my\subseteq\mo$ and $0<<q \in\my$ there is
a spectral decomposition in the domain $1<\Re s_j<3/2$
\begin{displaymath}
\begin{split}
& \sum_{\substack{\ell_1r_1 - \ell_2r_2 = q\\ 0 << r_{1, 2}\in \my
}} \frac{ \lambda_{\pi_1}(r_1\my^{-1}) \bar{\lambda}_{\pi_2}(r_2
\my^{-1}) N(\ell_1r_1\ell_2r_2)^{(\beta-1)/2} } {\prod_{j=1}^d
((\ell_1r_1+ \ell_2r_2)^{\sigma_j})^{s_j+\beta-1}} = \prod_{j=1}^d
q_j^{1/2-s_j} \int_{(\mc)} \sum_{\mt \mid \mc\mc_{\varpi}^{-1}}
\lambda_{\varpi}^{(\mt )}(q\my^{-1}) F_{\varpi, \mt }(s) \, d\varpi.
\end{split}
\end{displaymath}

$\bullet$ For $0<\eps<1/2$ there is a uniform bound in the domain
$1/2+\theta+\eps<\Re s_j<3/2$
\begin{displaymath}
\int_{(\mc)} \sum_{\mt \mid \mc\mc_{\varpi}^{-1}}
(\N\tilde{\lambda}_{\varpi})^c \left|F_{\varpi, \mt }(s)\right| \,
d\varpi \ll (\N\my)^{-1/2}(\N(\ell_1\ell_2))^\eps(\N\tilde
s)^{d(46+8c)}
\end{displaymath}
with the notation \eqref{tildenotation} and an implied constant
depending only on $\pi_{1,2}$, $\beta$, $K$, $\eps$. In particular,
the left hand side of the spectral identity can be continued
holomorphically to the larger domain $\Re s_j > 1/2 + \theta$ with
polynomial growth on vertical lines.
\end{theorem}

\begin{remark} The character assumptions on $\pi_{1,2}$ and the
positivity assumptions on $\ell_{1,2}$, $r_{1,2}$, $q$ are not
essential, they only serve simplicity of notation and exposition.
With a little more work one can show that the implied constant
depends polynomially on $C(\pi_1)C(\pi_2)$ and $\beta$.
\end{remark}

\begin{remark}
Selberg~\cite{Se} asks for the meromorphic continuation of a
Dirichlet series associated to shifted convolution sums. Progress
over $\QQ$ in this direction was made by Good~\cite{Go1,Go2},
Sarnak~\cite{Sa,Sa2}, Jutila~\cite{Ju1,Ju2}, Motohashi~\cite{Mo} and
the authors~\cite{BH}. A version of Theorem~\ref{theorem5} for
$\pi_{1,2}$ whose archimedean components belong to the discrete
series appears in \cite{CPSS}.
\end{remark}

\begin{proof} This is similar to the proof of Theorem~2
in \cite{BH}, so we present only the main steps, and omit
convergence issues that are discussed in detail in \cite{BH}. For
$c$ and $\beta$ as in the statement, $\gamma:=d(45+8c)$ and $t\geq
0$ we consider the functions
\[W_{\beta}(t; z) := t^{\beta/2} e^{-zt}\qquad\text{and}\qquad
G_{\gamma}(t) :=
\begin{cases}
(t(1-t))^{\gamma}, & 0 \leq t \leq 1,\\
0, & t > 1.
\end{cases}\]
By Laplace inversion we have, for any $y_1,y_2,Y>0$,
\begin{equation}\label{laplace}
 \left(\frac{y_1y_2}{Y^2}\right)^{\beta/2}G_{\gamma}\left(\frac{y_1+y_2}{Y}\right) =
 \frac{1}{2\pi i} \int_{(1)} \check{G}_{\gamma}(z) \,W_{\beta}\left(\frac{y_1}{Y}; z\right)W_{\beta}\left(\frac{y_2}{Y};z\right)\,dz,
\end{equation}
where
\[\check{G}_{\gamma}(z) :=\int_0^{\infty}G_\gamma(t) \,e^{zt}\,dt,\qquad z\in\CC.\]
We note the bound
\begin{equation}\label{importantbound2}\check{G}_{\gamma}(z)
\ll_\gamma |z|^{-\gamma-1},\qquad \Re z=1,\end{equation} which
follows easily by partial integration. We will also need the Mellin
transform of $G_\gamma(t)$ as given by \cite[3.196.3]{GR},
\begin{equation}\label{mellin}
 \int_0^{\infty} G_{\gamma}(t) \,t^{s-1} \,dt = \frac{\Gamma(s+
\gamma)\Gamma(\gamma+1)}{\Gamma(s+1+2\gamma)},\qquad \Re s>-\gamma.
\end{equation}
Let $\Re z_j = 1$ and $Y_j>0$ for $1 \leq j \leq d$ (we will later
integrate over all $z_j$ and $Y_j$), and let us write $z = (z_1,
\dots z_d)$ and $Y=(Y_1,\dots Y_d)$. We apply Theorem~\ref{theorem4}
and Remark~\ref{L1bound} with $a=0$, $b=1$, $c$ as in the statement
of Theorem~\ref{theorem5}, $W_1(y;z):=\prod_{j=1}^d
W_\beta(y_j;z_j)$, $W_2(y;z):=\prod_{j=1}^d \bar W_\beta(y_j;z_j)$.
We find
\begin{equation}\label{thm41}
\begin{split}
 \sum_{\substack{\ell_1r_1 - \ell_2r_2 = q\\ 0 << r_{1, 2}\in \my}}
 &\frac{\lambda_{\pi_1}(r_1\my^{-1}) \bar{\lambda}_{\pi_2}(r_2 \my^{-1})}{\sqrt{\N (r_1r_2\my^{-2})}}
 \prod_{j=1}^d W_{\beta}\left(\frac{(\ell_1r_1)^{\sigma_j}}{Y_j}; z_j\right)
 W_{\beta}\left(\frac{(\ell_2r_2)^{\sigma_j}}{Y_j}; z_j\right)\\
 & = \int_{(\mc)} \sum_{\mt \mid \mc\mc_{\varpi}^{-1}}
 \frac{\lambda_{\varpi}^{(\mt)}(q\my^{-1}) }{\sqrt{\N (q\my^{-1})}}
 W_{\varpi, \mt}\left(\frac{q^{\sigma_1}}{Y_1}, \ldots, \frac{q^{\sigma_d}}{Y_d};z\right)\,d\varpi
\end{split}
\end{equation}
for some functions $W_{\varpi, \mt}(\cdot;z): K_\infty^\times \to
\CC$ depending only on $\pi_{1, 2}$, $\beta$, $z$, $\varpi$, $\mt$,
$K$. Following Remark~\ref{L1bound} and using $\| W_{1,2}(\cdot;
z)\|_{A^{\kappa}} \ll (\N\tilde{z})^{\kappa}$ for $\beta>3\kappa$ we
see that
\begin{equation}\label{importantbound}
\begin{split}
\int_{(\mc)} \sum_{\mt \mid \mc\mc_{\varpi}^{-1}}
&(\N\tilde{\lambda}_{\varpi})^c \left|W_{\varpi, \mt}(y;z)\right|
\,d\varpi \\ &\ll (\N(\ell_1\ell_2))^{1/2+\eps}(\N\tilde
z)^{d(44+8c)} \prod_{j=1}^d |y_j|^{1/2-\theta
-\eps}\min(1,|y_j|^{-1}),
\end{split}
\end{equation}
the implied constant depending only on $\pi_{1,2}$, $\beta$, $c$,
$K$, $\eps$. We integrate both sides of \eqref{thm41} against
\begin{displaymath}
 \left(\frac{1}{2\pi i}\right)^d \int_{(1)} \cdots \int_{(1)} \,\prod_{j=1}^d \check{G}_{\gamma}(z_j) \,dz,
\end{displaymath}
 so that by \eqref{laplace}
\begin{equation}\label{thm42}
\begin{split}
\sum_{\substack{\ell_1r_1 - \ell_2r_2 = q\\ 0 << r_{1, 2}\in \my}}
&\lambda_{\pi_1}(r_1\my^{-1}) \bar{\lambda}_{\pi_2}(r_2 \my^{-1})\,
(\N(\ell_1r_1\ell_2r_2))^{(\beta-1)/2} \prod_{j=1}^d Y_j^{-\beta} G_{\gamma}\left(\frac{(\ell_1r_1 + \ell_2r_2)^{\sigma_j}}{Y_j}\right)\\
& = (\N(\ell_1\ell_2q\my))^{-1/2} \int_{(\mc)} \sum_{\mt \mid
\mc\mc_{\varpi}^{-1}} \lambda_{\varpi}^{(\mt )}(q\my^{-1})\,
H_{\varpi, \mt }\left(\frac{q^{\sigma_1}}{Y_1}, \dots,
\frac{q^{\sigma_d}}{Y_d}\right) \,d\varpi,
\end{split}
\end{equation}
where
\begin{displaymath}
H_{\varpi, \mt }(y) := \left(\frac{1}{2\pi i}\right)^d \int_{(1)}
\cdots \int_{(1)} W_{\varpi, \mt }(y; z) \,\prod_{j=1}^d
\check{G}_{\gamma}(z_j)\ dz.
\end{displaymath}
Note that by \eqref{importantbound2} and \eqref{importantbound}
these integrals converge absolutely and they satisfy the bound
\begin{equation}\label{importantbound3}
\int_{(\mc)} \sum_{\mt \mid \mc\mc_{\varpi}^{-1}}
(\N\tilde{\lambda}_{\varpi})^c \left|H_{\varpi, \mt}(y)\right|
\,d\varpi
\ll_{\pi_{1,2},\beta,c,K,\eps}(\N(\ell_1\ell_2))^{1/2+\eps}
\prod_{j=1}^d |y_j|^{1/2-\theta -\eps}\min(1,|y_j|^{-1}).
\end{equation}
Now for $s \in \CC^d$ with $1<\Re s_j<3/2$ we integrate both sides
of \eqref{thm42} against\begin{displaymath}
 \int_{\Kplus }\,\prod_{j=1}^d Y_j^{1-s_j} \,d^{\times}Y
\end{displaymath}
and use also \eqref{mellin}: we arrive at the spectral identity of
Theorem~\ref{theorem5} with
\[F_{\varpi,\mt}(s):=(\N(\ell_1\ell_2\my))^{-1/2}
\Biggl(\prod_{j=1}^d\frac{\Gamma(s_j+\beta +2\gamma)}{\Gamma(s_j+
\beta - 1+ \gamma)\Gamma(\gamma+1)}\Biggr) \int_{\Kplus}
H_{\varpi,\mt}(y)\,\prod_{j=1}^d y_j^{s_j-1}\,d^\times y.\] By
\eqref{importantbound3} these functions are holomorphic in the
domain $1/2 + \theta +\eps< \Re s_j < 3/2$ and there they satisfy
the bound of Theorem~\ref{theorem5}. The proof is complete.
\end{proof}

\end{document}